\documentclass[english,refpage,intoc]{article}
\usepackage[T1]{fontenc}
\usepackage[latin9]{inputenc}
\usepackage{geometry}
\usepackage{babel}
\usepackage{refstyle}
\usepackage{enumitem}
\usepackage{amsmath}
\usepackage{amsthm}
\usepackage{amssymb}
\usepackage{stackrel}
\usepackage{makeidx}
\makeindex
\usepackage{graphicx}
\usepackage{setspace}
\usepackage{nomencl}
\usepackage{tikz}
\usepackage{pdfpages}

\providecommand{\makenomenclature}{\makeglossary}
\makenomenclature
\usepackage[unicode=true,pdfusetitle,
 bookmarks=true,bookmarksnumbered=false,bookmarksopen=false,
 breaklinks=false,pdfborder={0 0 1},backref=false,colorlinks=false]
 {hyperref}

\makeatletter

\AtBeginDocument{}
\AtBeginDocument{}
\AtBeginDocument{}
\AtBeginDocument{}
\RS@ifundefined{subsecref}
  {\newref{subsec}{name = \RSsectxt}}
  {}
\RS@ifundefined{thmref}
  {\def\RSthmtxt{theorem~}\newref{thm}{name = \RSthmtxt}}
  {}
\RS@ifundefined{lemref}
  {\def\RSlemtxt{lemma~}\newref{lem}{name = \RSlemtxt}}
  {}

\theoremstyle{plain}
\newtheorem{thm}{\protect\theoremname}[section]
\theoremstyle{plain}
\newtheorem*{rem*}{\protect\remarkname}
\theoremstyle{definition}
\newtheorem{defn}[thm]{\protect\definitionname}
\theoremstyle{definition}
\newtheorem{example}[thm]{\protect\examplename}
\theoremstyle{plain}
\newtheorem{assumption}[thm]{\protect\assumptionname}
\theoremstyle{plain}
\newtheorem{claim}[thm]{\protect\claimname}
\theoremstyle{plain}
\newtheorem{rem}[thm]{\protect\remarkname}
\theoremstyle{plain}
\newtheorem{lem}[thm]{\protect\lemmaname}
\newlist{casenv}{enumerate}{4}
\setlist[casenv]{leftmargin=*,align=left,widest={iiii}}
\setlist[casenv,1]{label={{\itshape\ \casename} \arabic*.},ref=\arabic*}
\setlist[casenv,2]{label={{\itshape\ \casename} \roman*.},ref=\roman*}
\setlist[casenv,3]{label={{\itshape\ \casename\ \alph*.}},ref=\alph*}
\setlist[casenv,4]{label={{\itshape\ \casename} \arabic*.},ref=\arabic*}
\theoremstyle{plain}
\newtheorem{prop}[thm]{\protect\propositionname}
\theoremstyle{remark}
\newtheorem{notation}[thm]{\protect\notationname}
\theoremstyle{plain}
\newtheorem{cor}[thm]{\protect\corollaryname}
\theoremstyle{definition}
\newtheorem*{example*}{\protect\examplename}
\theoremstyle{plain}
\newtheorem*{question*}{\protect\questionname}

\AtBeginDocument{
  
}

\makeatother

\def\nomname{Nomenclature}
\providecommand{\assumptionname}{Assumption}
\providecommand{\casename}{Case}
\providecommand{\claimname}{Claim}
\providecommand{\corollaryname}{Corollary}
\providecommand{\definitionname}{Definition}
\providecommand{\examplename}{Example}
\providecommand{\lemmaname}{Lemma}
\providecommand{\notationname}{Notation}
\providecommand{\propositionname}{Proposition}
\providecommand{\questionname}{Question}
\providecommand{\remarkname}{Remark}
\providecommand{\theoremname}{Theorem}

\colorlet{thmcolor}{lightgray}

\begin{document}
\pagenumbering{Alph}

\global\long\def\d{\text{d}}
\global\long\def\rn{\mathbb{R}^{n}}
\global\long\def\q{\mathbb{Q}}
\global\long\def\m{\mu}
\global\long\def\n{\nu}
\global\long\def\ro{\rho}
\global\long\def\x{\times}
\global\long\def\r{\mathbb{R}}
\global\long\def\mm{\mathcal{M}}
\global\long\def\iff{\Longleftrightarrow}
\global\long\def\dm{\frac{\d\mu}{\d\left|\mu\right|}}
\global\long\def\dn{\frac{\d\nu}{\d\left|\nu\right|}}
\global\long\def\b{\bigskip}
\global\long\def\f{\mathcal{F}}

\let\myIndex\theindex\renewcommand{\theindex}{\cleardoublepage\myIndex\addcontentsline{toc}{section}{\indexname}}
\renewcommand{\nomname}{List of Symbols}

\title{Duality Theorems and Vector Measures in Optimal
Transportation Theory}
\author{Shlomi Gover}
\date{Submitted April 2020}
\maketitle

\thispagestyle{empty}

\newpage{}

\thispagestyle{empty}
The research thesis was done under the supervision of Professor Gershon Wolansky.

\bigskip

The generous financial help of the Crown Family Doctoral Fellowship is gratefully acknowledged.
\newpage{}

\thispagestyle{empty}
\tableofcontents{}
\thispagestyle{empty}

\newpage{}
\section*{}
\addcontentsline{toc}{section}{Abstract}
\pagenumbering{arabic}
\setcounter{page}{1}
\begin{abstract}

The optimal transportation problem, first suggested by Gaspard Monge in the 18th century and later revived in the 1940s by Leonid Kantorovich, deals with the question of transporting a certain measure to another, using transport maps or transport plans that minimize the total cost of transportation. This problem is very popular since it has a variety of applications in economics, physics, computer science and more.

One of the main tools in this theory is the duality theorem, which states that the optimal total cost equals the value of a different optimization problem called the dual problem.

In this work, I show how the problem and duality theorem can be generalized to an abstract formulation, in which I omit the use of measures. I show how this generalization implies a wide range of different optimal transport problems, and even other problems from game theory, linear programming and functional analysis. 

In particular, I show how the optimal transport problem can be generalized to deal with vector measures as a result of the abstract theorem and discuss the properties of this problem: I present its corresponding duality theorem and formulate conditions for the existence of transport plans, transport maps and solutions to the dual problem.

\end{abstract}
\newpage{}

\section*{Conventions}
\addcontentsline{toc}{section}{Conventions}

\begin{itemize}
\item $\mathbb{N}=\{1,2,3,\ldots\}$ \index{N@$\mathbb{N}$} is the set of natural numbers, $\mathbb{R}$\index{R@$\mathbb{R}$} is the set of real numbers. $\r_+$\index{R@$\mathbb{R}_+$} is the set of non-negative real numbers. $\rn$ is the $n$-dimensional Euclidean space.
\item For any $a,b\in\r$, $a\wedge b=\min\{a,b\}$\index{1@$\wedge$}. $[a]_+=\max\{a,0\}$\index{1@$[\cdot]_+$}.
\item For two topological spaces $X$ and $Y$, $C\left(X,Y\right)$\index{C@$C(X,Y)$}
is the space of continuous functions $f:X\rightarrow Y$. We denote
$C\left(X\right):=C\left(X,\r\right)$\index{C@$C(X)$}. $C_b(X)\subset C(X)$\index{Cb@$C_b(X)$} is the subset of bounded continuous functions. When $S$ is a normed space, the supremum norm on $C(X,S)$ is defined by $\|f\|_{\infty}:=\sup_{x\in X}\|f(x)\|_S$.\index{C1@$\|\cdot\|_{\infty}$}
\item For a measurable space $X$, $\mathcal{P}\left(X\right)$/$\mm_{+}\left(X\right)$
/$\mm\left(X\right)$
is the set of
probability/non-negative/signed measures on $X$ respectively.\index{P@$\mathcal{P}$}\index{m+@$\mathcal{M}_+$}\index{m@$\mm$}
\item For a topological space $X$ equipped with its Borel $\sigma$-algebra, \index{support, supp}the support of a measure $\mu\in\mm_+\left(X\right)$
is the smallest closed set with measure $\mu\left(X\right)$, denoted
by $\text{supp}\ \mu$.
\item An integral of a function wrt a measure is denoted by $\int_{X}f\left(x\right)\d\mu\left(x\right)$
or $\int f\d\mu$ (when the latter is used whenever the meaning
for the underlying space and variable are clear).
\item For every $x\in X$, $\delta_{x}\in\mathcal{P}\left(X\right)$
is the delta measure satisfying $\delta_x(A):=\begin{cases}
1 & ;\ x\in A\\
0 & ;\ x\notin A
\end{cases}$\index{d@$\delta$}
\item For $\mu,\nu\in\mm_{+}\left(X\right)$ we say $\mu$ is absolutely continuous wrt $\nu$ and denote $\mu<<\nu$\index{1@$<<$}
if $\nu(A)=0\Rightarrow\mu(A)=0$ for every measurable $A\subset X$. In this case, by the Radon-Nikodym theorem (Theorem 3.2.2 in \cite{bogachev2007measure}), there exists a function $\frac{\d\mu}{\d\nu}:X\rightarrow\r$
(called the Radon-Nikodym derivative\index{Radon-Nikodym derivative}) such that $\mu\left(A\right)=\int_{A}\frac{\d\mu}{\d\nu}\left(x\right)\d\nu\left(x\right)$
for every measurable $A\subset X$.
\item For a topological vector space $X$, we denote by $X^{*}$ \index{1@$^*$} 
the \index{dual space}dual space of $X$ which consists of the continuous
linear functionals $f:X\rightarrow\r$. The duality relation between
two elements $x\in X,\ x^{*}\in X^{*}$ is denoted by $\left\langle x,x^{*}\right\rangle$,
$\left\langle x^{*},x\right\rangle$ \index{1@$\left\langle \cdot,\cdot\right\rangle$}or $x^{*}\left(x\right)$. The weak* topology\index{weak* topology} on $X^*$ is the topology generated by the functionals $\{\left\langle x,\cdot\right\rangle:X^*\rightarrow\r;\ x\in X\}$.
\item For a function $f:X\rightarrow Y$ and $B\subset Y$, we denote the
preimage $f^{-1}\left(B\right):=\left\{ x\in X;\ f\left(x\right)\in B\right\}$.
\item Given two subsets $A,B$ of a vector space and a scalar $\alpha$, we denote
\[
A+B=\left\{ a+b;\ a\in A,b\in B\right\} ,\ A-B=\left\{ a-b;\ a\in A,b\in B\right\} ,\ \alpha A=\left\{ \alpha a,\ a\in A\right\} 
\]
\item For any two sets $A\subset B$. The characteristic function $\chi_{A}:B\rightarrow\r$\index{x@$\chi$}
is defined by 
\[
\chi_{A}\left(x\right):=\begin{cases}
1 & ;\ x\in A\\
0 & ;\ x\notin A
\end{cases}
\]
\item A \index{preorder}preorder on set $A$ is a subset
$R\subset A\x A$ satisfying reflexivity- $\left(a,a\right)\in R\ \forall a\in A$,
and transitivity- $\left(a,b\right),\left(b,c\right)\in R\Rightarrow\left(a,c\right)\in R\ \forall a,b,c\in A$.
\item For a subset $A$ of a metric space $(X,d)$ we denote the diameter\index{diam@$\text{diam}$/diameter} of $A$: $$\text{diam}(A)=\sup\{d(x,y);\ x,y\in A\}$$.
\end{itemize}

\newpage{}

\section{Introduction}

\subsection{The Monge-Kantorovich problem}

In the 18th century, Gaspard Monge introduced a new problem
of transporting soil from one place to the other in the Euclidean
space of 2 or 3 dimensions \cite{Monge1781}, in such a way that the total cost (distance
times mass) of transporting will be minimal. A  modern formulation
of the problem is as follows: Let $X,Y$ be two Polish spaces (complete, metric and separable),
$\mu\in\mathcal{P}\left(X\right)$ and $\nu\in\mathcal{P}\left(Y\right)$
two Borel probability measures, and $c:X\times Y\rightarrow\r$
a lower semi-continuous non-negative cost function. A \index{transport map}transport
map taking $\mu$ to $\nu$ is a Borel measurable map $T:X\rightarrow Y$
satisfying $T_{\#}\mu=\nu$\index{T@$T_\#$} (meaning
$\mu\left(T^{-1}\left(E\right)\right)=\nu\left(E\right)$ for every
Borel measurable $E\subset Y$). The problem of Monge is to minimize the
total cost $\int_{X}c\left(x,Tx\right)\d\mu\left(x\right)$
over all transport maps taking $\mu$ to $\nu$. In this non-linear problem, even if there is some transport map, a minimizer transport map may fail to exist (see Example 4.9 in \cite{Villani2016}).
In the relaxation of the problem, first made by Leonid Kantorovich in the 1940s \cite{Kantorovich2006}, we replace the transport maps with transport plans \index{transport plan} $\pi\in\mathcal{P}\left(X\x Y\right)$,
which are probability measures with \index{marginal}marginals $\mu\in\mathcal{P}(X),\nu\in\mathcal{P}(Y)$ (meaning $\pi\left(A\times Y\right)=\mu\left(A\right)$
and $\pi\left(X\times B\right)=\nu\left(B\right)$ for every Borel measurable
$A\subset X,\ B\subset Y$). 
In this relaxation, we minimize the total cost $\int_{X\times Y}c\left(x,y\right)\d\pi\left(x,y\right)$
over all transport plans.  This is a relaxation to the problem of
Monge since every transport map induces a transport plan supported
on its graph $\pi=(\text{Id}_X\x T)_{\#}\m$ giving the same total cost $\int_Xc(x,Tx)\d\m=\int_{X\x Y}c(x,y)\d\pi(x,y)$.
The set of transport plans, denoted by $\Pi\left(\mu,\nu\right)$\index{pimn@$\Pi\left(\mu,\nu\right)$},
is always non-empty (it contains the product measure of $\m$ and $\n$) and is weak*-compact, thus the total cost (which is a weak* lower semicontinuous functional on $\Pi(\m,\n)$) attains a minimum on this set (see Theorem 4.1 in \cite{Villani2016}). The minimum in the relaxed problem may be strictly
smaller than the minimum (or infimum) in the problem of Monge when the optimal
plan "splits mass". The following theorem states a sufficient condition for equality between the two problems:
\begin{thm}
\label{thm:TheoremA+B}(Theorems A and B in \cite{PRATELLI20071})
If $\mu$ is non-atomic then there exist transport map $T:X\rightarrow Y$
such that $T_{\#}\mu=\nu$ and 
\[
\inf_{T:T_{\#}\mu=\nu}\int_{X}c\left(x,Tx\right)\d\mu\left(x\right)=\min_{\pi\in\Pi\left(\mu,\nu\right)}\int_{X\times Y}c\left(x,y\right)\d\pi\left(x,y\right)
\]
\end{thm}

Under some further assumptions, it is known that
the infimum in the problem of Monge is attained, see for example Theorem 2.44 in \cite{Villani2003} and  Theorems 9.3, 9.4 in \cite{Villani2016}.

The optimal total cost of the relaxed problem equals
to an optimal value of a different optimization problem called the dual problem\index{dual problem}. This duality is an
important tool in transportation theory and it is one of the main ideas on
which this present work is based on. 
\begin{thm}
\label{thm:Duality}(Kantorovich duality - Theorem 5.10 in \cite{Villani2016}) Assume $c(x,y)\geq a(x)+b(y)\ \forall x,y$ for some $a\in L^1(\m),b\in L^1(\n)$ and denote \index{f1c@$\Phi\left(c\right)$}
\[
\Phi\left(c\right):=\left\{ \left(\psi,\varphi\right)\in L^1\left(\mu\right)\x L^1\left(\nu\right),\ \psi\left(x\right)+\varphi\left(y\right)\leq c\left(x,y\right)\ \right\} 
\]
then 
\begin{equation}
\min_{\pi\in\Pi\left(\mu,\nu\right)}\int_{X\times Y}c\left(x,y\right)\d\pi\left(x,y\right)
=
\sup_{\left(\varphi,\psi\right)\in\Phi\left(c\right)}\int_{X}\psi\left(x\right)\d\mu\left(x\right)+\int_{Y}\varphi\left(y\right)\d\nu\left(y\right)
\label{eq:ClaDuality}
\end{equation}
Moreover, if $c(x,y)\leq \alpha(x)+\beta(y)$ for some $\alpha\in L^1(\m),\beta\in L^1(\n)$ then the supremum is attained.
\end{thm}

\begin{rem}
The supremum in the dual problem will not change when restricting $\psi,\varphi$ to be continuous. By the proof of the above theorem, one can conclude that if $X,Y$ are
compact and $c$ is continuous, then there exists continuous maximizers.
\end{rem}

In the \textquotedbl semi-discrete case\textquotedbl\index{semi-discrete},  where $Y$
is a finite set, the optimal transport problem can be used to model
a market with a set of consumers willing to buy a single product
from a set of vendors $Y$:

\begin{itemize}
\item $X$ is the set of consumers, $\mu\in\mathcal{P}(X)$ is the distribution
of demand over $X$.
\item $Y=\{1,\ldots,m\}$ is a finite set of vendors selling a single product, $\nu\in\mathcal{P}(Y)\subset\r^m$
is the supply of the product for the different vendors.
\item The cost function $c\left(x,i\right):=c_{i}\left(x\right)$ represents
how much a consumer $x$ pays to vendor $i\in Y$.
\item A transport map $T:X\rightarrow Y$ determines which vendor sells
to each consumer and needs to satisfy "demand=supply" for each vendor, meaning $\mu\left(T^{-1}\left(i\right)\right)=\nu_{i}\ \forall i$
(namely $T_{\#}\mu=\nu$). The optimal transport map minimizes
the total cost of all consumers:
\[
\int_{X}c\left(x,Tx\right)\d \m(x)=\sum_{i}\int_{T^{-1}\left(i\right)}c_{i}\left(x\right)\d\mu(x)
\]
\item A transport plan will allow consumers to divide their purchase between
the different vendors.
\end{itemize}

\subsection{Generalization to vector measures}

Inspired by Monge and Kantorovich, many similar problems have been presented since, many of which carry the same structure of minimizing
the integral of a given function over some set of measures.
In this work, I introduce some examples of such problems,  one of them is the optimal transport problem for vector measures, which generalizes the classical problem of Monge and Kantorovich.  One motivation
for this generalization is the  transport of an incompressible fluid:
Such transport should preserve the volume since the fluid cannot
be stretched nor compressed while being transported. Hence we need
to consider only transport maps satisfying $T_{\#}\mu=\nu$ and $T_{\#}\lambda=\lambda$,
where $\lambda$ is the Lebesgue measure. More generally, we consider
all transport maps satisfying $T_{\#}\mu_{i}=\nu_{i}$ for all $i=1,\ldots,n$,
where $\left(\mu_{1},\ldots,\mu_{n}\right),\left(\nu_{1},\ldots,\nu_{n}\right)$
are two given vector measures.

Another motivation for this generalization is in the semi-discrete
case where this problem can model a market with several products: Now $\nu$ is a set of $n$
measures on the set of vendors $Y=\{1,\ldots,m\}$ which represent the supply of the product set $P=\{1,\ldots,n\}$
for each vendor. Explicitly, for any $ i\in Y$ and $j\in P$, $\n_{ij}$ is the supply of product $j$ with vendor $i$ and $\mu_j$ represents the demand for product $j$. The constraints on the transport maps are now
$
\mu_j(T^{-1}(i))=\nu_{ij}
$
for every product $j$ and vendor $i$, or equivalently $T_{\#}\mu_{j}=\nu_{j}:=(\n_{1j},\ldots,\n_{mj})$ for every $j$.

\subsection{Results}

One of the main results of this work is the abstract duality theorem (Theorem
\ref{Thm:AbsDuality}) which generalizes the classical Kantorovich
duality. This theorem can be applied to the problem of vector measures (Theorem
\ref{thm:GMK}) and other problems, for which I give examples in Chapters 4,5 and 6.
 For the problem of vector measures, I give equivalent conditions to the existence of transport plans (Theorem \ref{thm:Blackwell}), using a generalization of a theorem by Blackwell (appear in \cite{blackwell1951}) to infinite-dimensional measures, and sufficient conditions to the existence of optimal transport map (Theorems \ref{thm:ExTransMap} and \ref{thm:ExTransMap generalcase}).
 
\subsection{Outline of the work}
In Chapter 2, I introduce
two abstract duality theorems (Theorems \ref{Thm:AbsDuality} and \ref{thm:Convex Abstract}), prove them and discuss their connection with the
classical duality theorem. I also compare
these theorems to other known results. In Chapter 3, I formulate
the optimal transport problem for vector measures (finite and infinite-dimensional), discuss its relaxation and show its corresponding duality
theorem as a consequence of the abstract theorem. Moreover, I formulate
conditions for the existence of transport plans and transport maps for this problem.
In Chapter 4, I give a few examples of various transport
problems, which are also implied by the abstract theorems.
In Chapter 5, I show how the abstract theorems can be applied to other problems in linear programming, game theory and functional analysis.
In Chapter 6, I present a new "semi-dynamic" transport problem and again use the abstract theorems to formulate the duality corresponding to this problem.

\newpage{}

\section{The abstract duality theorem}

In this chapter, I present a generalization of the celebrated Kantorovich
duality theorem (Theorem \ref{thm:Duality}) and review its connection with other known theorems. The abstract theorem deals with linear functionals on subspaces
and their extensions, instead of continuous functions and measures.
Its proof is strongly based on the Hahn-Banach theorem, which guarantees
the existence of such extensions, and can be found in any standard
functional analysis book (for example Theorem 5.53 in \cite{Aliprantis2006Hitch} and Theorem 13.1.2 in \cite{shalit2017first}).
The theorem uses the following notions:
\begin{defn}
$\ $
\begin{enumerate}
\label{def:sublinear}
\item A \index{convex cone}\emph{convex cone} is a subset of a vector
space which is closed under addition and multiplication by a non-negative scalar.
\item \index{sublinear} A functional $P:K\rightarrow\r\cup\left\{ -\infty\right\} $
defined on a convex cone $K$ is called \emph{sublinear }if it is
subadditive\index{subadditive}:
\[
P\left(x+y\right)\leq P\left(x\right)+P\left(y\right)\ \forall x,y\in K
\]
 and positively-homogeneous\index{positively-homogeneous}:
\[
P\left(\alpha x\right)=\alpha P\left(x\right)\ \forall\alpha\geq0,\ x\in K
\]
while using the conventions $r+\left(-\infty\right)=-\infty\ \forall r\in \r$, $\lambda\cdot\left(-\infty\right)=-\infty\ \forall \lambda >0$.
\end{enumerate}
\end{defn}

\begin{thm}
\index{Hahn-Banach theorem}\label{thm:Hahn-Banach}(Hahn-Banach) Let $X$
be a real vector space, $P:X\rightarrow\mathbb{R}$ a sublinear functional
and $f:Y\rightarrow\mathbb{R}$ a linear functional defined on a linear
subspace $Y\subseteq X$ such that $f\left(y\right)\leq P\left(y\right)$
for every $y\in Y$. Then there exists a linear functional $F:X\rightarrow\r$
such that $F\left(y\right)=f\left(y\right)\ \forall y\in Y$ and $F\left(x\right)\leq P\left(x\right)\ \forall x\in X$.
\end{thm}

\begin{defn}
A \emph{topological vector space} (or in short tvs)\index{topological vector space/tvs} is a vector space equipped with a topology under which the operations of addition and multiplication by a scalar are continuous.
\end{defn}
\subsection{Settings and notations}

The following assumptions and notations are needed to state our theorem and will
be used along this chapter.
\begin{itemize}
\item $V$ is a non-empty topological vector space. $V^{*}$
is its topological dual. $u\in V$ is a given element.
\item $\q\subset V$ and $K$ are two non-empty convex cones.
\item For every $x,y\in V$ we denote $x\geq y\Leftrightarrow x-y\in\q$
(a preorder on $V$).
\item $h:K\rightarrow V$ is $\q$-superlinear\index{Q@$\mathbb{Q}$-superlinear}, meaning $h\left(\lambda k\right)=\lambda h\left(k\right)$
for all $\lambda\geq0$ (in particular $h\left(0\right)=0$) and $h\left(k+k'\right)\geq h\left(k\right)+h\left(k'\right)$.
\item $k^{*}:K\rightarrow\mathbb{R}$ is a sublinear functional.
We say $k^{*}$ is $h$-positive\index{h@$h$-positive functional} if $k^{*}\left(k\right)\geq0$ whenever
$h(k)\geq0$ (meaning $k\in h^{-1}\left(\q\right)$).
\item $\Pi\left(k^{*},h\right)\subset V^{*}$\index{pikh@$\Pi\left(k^{*},h\right)$}
is the set of $v^{*}\in V^{*}$ such that $v^{*}\left(q\right)\geq0\ \forall q\in\q$
(this is an equality whenever $\q$ is a subspace) and $v^{*}\circ h\left(k\right)\leq k^{*}\left(k\right)\ \forall k\in K$
(this is an equality whenever $K,h$ and $k^{*}$ are linear).
\item For any $v\in V$, we denote $\Phi\left(v,h\right):=\left\{ k\in K\ ;\ h\left(k\right)\geq v\right\} =h^{-1}\left(v+\q\right)$\index{f1vh@$\Phi\left(v,h\right)$}.
\end{itemize}

We will need the following assumption which correlates $h$ with $\q$ and the topology of $V$ with $k^*$:
\begin{assumption}
\label{assu:Sublinear}There exists $s:V\rightarrow K$
such that $s\left(v\right)\in\Phi\left(v,h\right)\ \forall v\in V$ and
$k^{*}\circ s$ is bounded from above on a neighborhood of $0$.
\end{assumption}

In certain cases, we may use the following stronger/equivalent assumptions:

\begin{claim}
If $V$ is a normed space and $\cap_{v\in V:\left|\left|v\right|\right|=1}\Phi\left(v,h\right)\neq\emptyset$,
then Assumption \ref{assu:Sublinear} holds.
\end{claim}

\begin{proof}
Let $k\in\cap_{v\in V:\left|\left|v\right|\right|=1}\Phi\left(v,h\right)$
then $k\left|\left|v\right|\right|\in\Phi\left(v,h\right)$ for every
$v\in V$, so we can choose $s\left(v\right)=k\left|\left|v\right|\right|$.
\end{proof}

\begin{claim}
\label{claim:AssFD}
If $V$ is equipped with the discrete topology, Assumption \ref{assu:Sublinear} is equivalent to $\Phi(v,h)\neq\emptyset\ \forall v\in V$.
\end{claim}
\begin{proof}
It is clear that Assumption \ref{assu:Sublinear} implies $\Phi(v,h)\neq\emptyset\ \forall v\in V$ (not only when $V$ is equipped with the discrete topology). For the other direction, if $V$ is equipped with the discrete topology, choose $s(0)=0$ (which belongs to $\Phi(0,h)$ since $h(0)=0$) and $s(v)$ to be some element in $\Phi(v,h)$ for any $v\neq0$. In the discrete topology $\{0\}$ is a neighborhood of $0$ and $k^*\circ s$ is bounded on this neighborhood.
\end{proof}

\begin{rem}
\label{rem:AssFD}
If $V$ is finite-dimensional, every linear functional $f:V\rightarrow\r$ is continuous. Since in our theorems the only part of the topology is to determine the continuous linear functionals, we may assume $V$ is equipped with the discrete topology. Therefore by the above claim when $V$ is finite-dimensional we may replace Assumption \ref{assu:Sublinear} with the assumption $\Phi(v,h)\neq\emptyset\ \forall v\in V$.
\end{rem}

\bigskip
The following diagram illustrates the different objects defined above and their connections:
\begin{center}
   \includegraphics[scale=0.3]{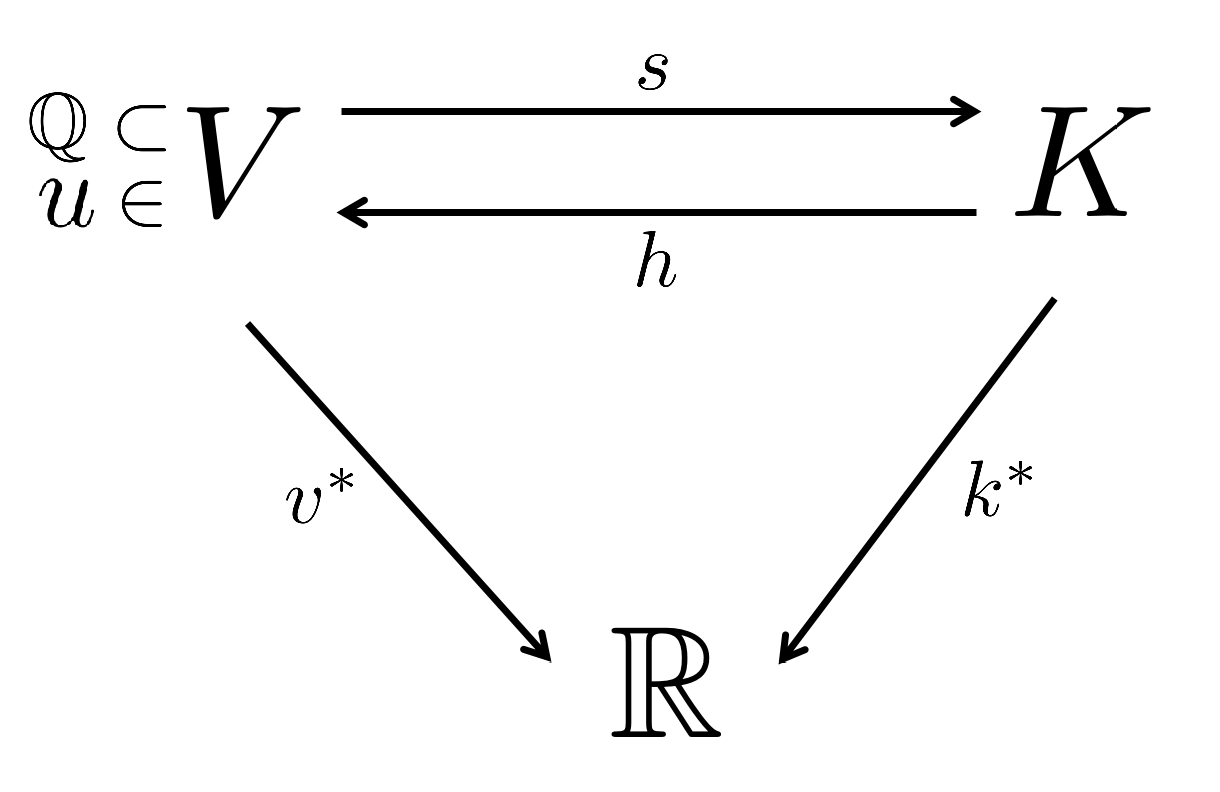} 
\end{center}
 
 \fbox{
\parbox{15cm}{
\begin{example}
\label{exa:Kant}For two compact metric measure spaces $\left(X,\mu\right),\left(Y,\nu\right)$,
the connection with the optimal transport problem can be obtained by
taking:
\begin{itemize}
\item $V=C\left(X\x Y\right)$ equipped with the supremum norm, $V^{*}=\mm\left(X\x Y\right)$.
\item $\q=$ The cone of non-negative functions in $V$.
\item $K=C\left(X\right)\x C\left(Y\right)$, $h\left(\psi,\varphi\right):=\psi+\varphi$,
$k^{*}\left(\psi,\varphi\right):=\int\psi\d\mu+\int\varphi\d\nu$.
\end{itemize}

Then $\Pi\left(k^{*},h\right)=\Pi\left(\mu,\nu\right)$ is the set
of measures on $X\x Y$ with marginals $\mu,\nu$, and 
\[
\Phi\left(v,h\right):=\left\{ \left(\psi,\varphi\right)\in K\ ;\ \psi\left(x\right)+\varphi\left(y\right)\geq v\left(x,y\right)\ \forall x,y\right\} 
\]
Assumption \ref{assu:Sublinear} holds by taking $s(v)=(\|v\|_{\infty},0)$.
\end{example}
}
}

\subsection{Statement of the theorem}

In our theorem, the primal problem is to maximize a linear functional
acting on a given element in the space, among all functionals that are non-negative on $\q$ and
dominated by
$k^{*}$. The theorem states the existence of a maximizer and formulates
its dual problem:

\bigskip

\colorbox{thmcolor}{
\parbox{\linewidth}{
\begin{thm}
\index{abstract duality}(Abstract duality theorem) \label{Thm:AbsDuality}
Under Assumption \ref{assu:Sublinear} 

\begin{equation}
\max_{v^{*}\in\Pi\left(k^{*},h\right)}v^{*}\left(u\right)=\inf_{k\in\Phi\left(u,h\right)}k^{*}\left(k\right)\label{eq:absdual}
\end{equation}
and the equality is finite (rather than $-\infty=-\infty$) iff $k^{*}$
is $h$-positive.
\end{thm}
}}

\begin{rem}
An equivalent formulation of the theorem is 
\[
\min_{v^{*}\in\Pi\left(k^{*},h\right)}v^{*}\left(u\right)=\sup_{k\in\Phi\left(-u,h\right)}-k^{*}\left(k\right)
\]
and in case $k^{*}$ and $K$ are linear 
\[
\min_{v^{*}\in\Pi\left(k^{*},h\right)}v^{*}\left(u\right)=\sup_{-k\in\Phi\left(-u,h\right)}k^{*}\left(k\right)
\]
\end{rem}

\begin{rem}
\label{rem:KantDual}By taking $V,K,h,k^{*}$ as in Example \ref{exa:Kant},
and using the above remark, Theorem \ref{Thm:AbsDuality} implies a weaker version of the Kantorovich duality (Theorem \ref{thm:Duality}) for compact $X,Y$.
\end{rem}

\subsection{Preliminary lemmas}

To prove Theorem \ref{Thm:AbsDuality}, we will need the following
lemmas about the functional $p:V\rightarrow\mathbb{R}\cup\left\{ -\infty\right\} $
defined by
\[
p\left(v\right):=\inf_{k\in\Phi\left(v,h\right)}k^{*}\left(k\right)
\]
Notice that $p<\infty$ under Assumption \ref{assu:Sublinear} since it implies $\Phi\left(v,h\right)\not=\emptyset\ \forall v\in V$.
\begin{lem}
\label{lem:psublinear}$p$ is sublinear.
\end{lem}

\begin{proof}
Need to show subadditivity and positive-homogeneity:
\begin{itemize}
\item Positive-homogeneity: Let $\lambda>0,\ v\in V$ then 
\[
k\in\Phi\left(\lambda v,h\right)\Leftrightarrow h\left(k\right)-\lambda v\in\q\Leftrightarrow h\left(\lambda^{-1}k\right)-v\in\q\Leftrightarrow\lambda^{-1}k\in\Phi\left(v,h\right)
\]
hence
\[
p\left(\lambda v\right):=\inf_{k\in\Phi\left(\lambda v,h\right)}k^{*}\left(k\right)=\inf_{\lambda^{-1}k\in\Phi\left(v,h\right)}k^{*}\left(k\right)=\inf_{k\in\Phi\left(v,h\right)}k^{*}\left(\lambda k\right)=\lambda\inf_{k\in\Phi\left(v,h\right)}k^{*}\left(k\right)=:\lambda p\left(v\right)
\]
\item Subadditivity: Let $v,v'\in V$ and $k\in\Phi\left(v,h\right)$, $k'\in\Phi\left(v',h\right)$ (by Assumption \ref{assu:Sublinear} $\Phi\left(v,h\right),\Phi\left(v',h\right)\neq\emptyset$) then
\[
h\left(k+k'\right)\geq h\left(k\right)+h\left(k'\right)\geq v+v'
\]
thus $k+k'\in\Phi\left(v+v',h\right)$ and
\[
\begin{array}{ccccc}
p\left(v+v'\right)=\inf_{\Phi\left(v+v',h\right)}k^{*}\leq  k^{*}\left(k+k'\right)  \leq  k^{*}\left(k\right)+k^{*}\left(k'\right)\end{array}
\]
We finish our proof by taking the infimum over all $k\in\Phi\left(v,h\right)$
and $k'\in\Phi\left(v',h\right)$ on the RHS of the inequality above.
\end{itemize}
\end{proof}
\begin{lem}
\label{lem:Char of Pi}$\Pi\left(k^{*},h\right)=\left\{ v^{*}\in V^{*};\ v^{*}(v)\leq p(v)\ \forall v\in V\right\} $.
\end{lem}

\begin{proof}
$ $
\begin{itemize}
\item Let $v^{*}\in V^{*}$ such that $v^{*}\left(v\right)\leq p\left(v\right)\ \forall v\in V$,
we need to show $v^{*}\in\Pi\left(k^{*},h\right)$: For every $q\in\q$,
\[
h\left(0\right)=0\geq-q\Rightarrow0\in\Phi\left(-q,h\right)
\]
 hence
\[
-v^{*}\left(q\right)=v^{*}\left(-q\right)\leq p\left(-q\right)=\inf_{k\in\Phi\left(-q,h\right)}k^{*}\left(k\right)\leq k^{*}\left(0\right)
\]
so since $k^*$ is sublinear $k^*(0)=0$ and we conclude $v^*(q)\geq0\ \forall q\in\q$.
$k\in\Phi\left(h\left(k\right),h\right)$ for every $k\in K$
therefore
\[
v^{*}\left(h\left(k\right)\right)\leq p\left(h\left(k\right)\right)=\inf_{k'\in\Phi\left(h\left(k\right),h\right)}k^{*}\left(k'\right)\leq k^{*}\left(k\right)\ \forall k\in K
\]
so overall $v^*\in\Pi(k^*,h)$.
\item Let $v^{*}\in\Pi\left(k^{*},h\right)$, we need to show $v^{*}\left(v\right)\leq p\left(v\right)\ \forall v\in V$: Let $v\in V$, for any $k\in K$ such that $h\left(k\right)\geq v$
\[
v^{*}\left(v\right)\leq v^{*}\left(h\left(k\right)\right)\leq k^{*}\left(k\right)
\]
so by taking the infimum over all such $k\in K$ we get 
\[
v^{*}\left(v\right)\leq\inf_{k\in\Phi\left(v,h\right)}k^{*}\left(k\right)=p\left(v\right)
\]
\end{itemize}
\end{proof}
\begin{lem}   
\label{lem:EqCon}Under Assumption \ref{assu:Sublinear} the following are equivalent:

\begin{enumerate}
\item $p\left(v\right)>-\infty$ for some $v\in V$.
\item $p$ is finite on $V$
\item For every $v\in V$ there exists $v^{*}\in\Pi\left(k^{*},h\right)$
such that $v^{*}\left(v\right)=p\left(v\right)$.
\item $\Pi\left(k^{*},h\right)\neq\emptyset$.
\item $k^{*}$ is $h$-positive.
\end{enumerate}
\end{lem}

\begin{proof}
$\ $

\begin{itemize}
\item $1\Rightarrow2$: Let $v'\in V$. Since $p$
is sublinear by Lemma \ref{lem:psublinear} then $$-\infty<p\left(v\right)\leq p\left(v'\right)+p\left(v-v'\right)\Rightarrow p\left(v'\right)>-\infty$$
\item $2\Rightarrow3$: Let $v\in V$. Define
a linear functional $f$ on the subspace $\mathrm{span}\left\{ v\right\} \subset V$
by $f\left(\alpha v\right):=\alpha p\left(v\right)\ \forall\alpha\in\mathbb{R}$.
We will show that $f\leq p$ on $\mathrm{span}\left\{ v\right\} $: By positive-homogeneity of $p$ we know that $p\left(0\right)=0$ thus
\[
\begin{array}{ccccccccc}
\alpha>0 & \Rightarrow & f\left(\alpha v\right) & = & \alpha p\left(v\right) & = & p\left(\alpha v\right)\\
\alpha=0 & \Rightarrow & f\left(\alpha v\right) & = & 0 & = & p\left(0\right) & = & p\left(\alpha v\right)\\
\alpha<0 & \Rightarrow & f\left(\alpha v\right) & = & \alpha p\left(v\right) & = & -p\left(-\alpha v\right) & \leq & p\left(\alpha v\right)
\end{array}
\]
where the last inequality is by the subadditivity of $p$. By the Hahn-Banach theorem (which we may use since $p$ is finite) there
exists a linear functional $v^{*}:V\rightarrow\r$ (not necessarily continuous) such that $v^{*}\left(v'\right)\leq p\left(v'\right)\ \forall v'\in V$
and $v^{*}\left(v'\right)=f\left(v'\right)\ \forall v'\in\mathrm{span}\left\{ v\right\} $
in particular $p\left(v\right)=v^{*}\left(v\right)$. By Assumption
\ref{assu:Sublinear}, there exists $s:V\rightarrow K$ such that $k^*\circ s$ is bounded from above on a neighborhood of $0$. Since $v^*(v')\leq p(v')\leq k^*\circ s(v')\ \forall v'\in V$, $v^*$ is also bounded from above on a neighborhood of $0$ therefore it is continuous - $v^*\in V^*$ (see for example Theorem 5.43 in \cite{Aliprantis2006Hitch}),
and by Lemma \ref{lem:Char of Pi} $v^{*}\in\Pi\left(k^{*},h\right)$.

\item $3\Rightarrow4$: Obvious.
\item $4\Rightarrow5$: Let $v^{*}\in\Pi\left(k^{*},h\right)$
then $0\leq v^{*}\left(h\left(k\right)\right)\leq k^{*}\left(k\right)$
for every $k\in h^{-1}\left(\q\right)=\Phi\left(0,h\right)$.
\item $5\Rightarrow1$: For every $k\in h^{-1}\left(\q\right)=\Phi\left(0,h\right)$,
$k^{*}\left(k\right)\geq0$ hence 
\[
p\left(0\right)=\inf_{k\in\Phi\left(0,h\right)}k^{*}\left(k\right)\geq0
\]
\end{itemize}
\end{proof}

\fbox{
\parbox{15cm}{
\begin{example}
\index{Hahn-Banach theorem}
We proved Lemma \ref{lem:EqCon} using the Hahn-Banach theorem. It is also possible to achieve the Hahn-Banach theorem using the implication $5\Rightarrow4$ in this lemma. To see that, we take
$V=X$ equipped with the discrete topology (or any other topology
under which $P$ is continuous at $0$), $Y$ a subspace of $X$, $f\leq P$ on $Y$, $K=X\x Y$, $h\left(x,y\right)=x+y$,
$k^{*}\left(x,y\right)=P\left(x\right)+f\left(y\right)$ and $\q=\left\{ 0\right\}$.
We verify Assumption \ref{assu:Sublinear} by taking $s:X\rightarrow X\x Y$
to be $s\left(x\right)=\left(x,0\right)$.
Then
\[
\Pi\left(k^{*},h\right)=\left\{ x^{*}\in X^{*};x^{*}\left(x+y\right)\leq P\left(x\right)+f\left(y\right)\forall\left(x,y\right)\in X\x Y\right\} 
\]

This is the set of all extensions of $f$ that are dominated by $P$.
 By Lemma \ref{lem:EqCon}, this set is not empty whenever
$k^{*}$ is $h$-positive which is always true since $f\leq P$ and therefore
\[
h\left(x,y\right)\in\q\Rightarrow x+y=0\Rightarrow f\left(-y\right)\leq P\left(-y\right)=P\left(x\right)\Rightarrow P\left(x\right)+f\left(y\right)\geq0
\]
\end{example}
}
}

\subsection{Proof of the theorem}

We are now ready to prove the abstract duality theorem:
\begin{proof}
(of Theorem \ref{Thm:AbsDuality})

\begin{casenv}
\item $k^{*}$ is $h$-positive:

By Lemma \ref{lem:EqCon} there exists $v^{*}\in\Pi\left(k^{*},h\right)$
such that $v^{*}\left(u\right)=p\left(u\right)$. We need to show $v^{*}$
is optimal in $\Pi\left(k^{*},h\right)$: By Lemma \ref{lem:Char of Pi}
for every $\pi\in\Pi\left(k^{*},h\right)$
\[
\pi\left(u\right)\leq p\left(u\right)=v^{*}\left(u\right)
\]
 and since $v^{*}\in\Pi\left(k^{*},h\right)$ we get
\[
v^{*}\left(u\right)=\max_{\pi\in\Pi\left(k^{*},h\right)}\pi\left(u\right)
\]
Altogether
\[
\inf_{k\in\Phi\left(u,h\right)}k^{*}\left(k\right)=\max_{\pi\in\Pi\left(k^{*},h\right)}\pi\left(u\right)
\]

\item $k^{*}$ is not $h$-positive:

By Lemma \ref{lem:EqCon} $p\equiv-\infty$ and $\Pi\left(k^{*},h\right)=\emptyset$
hence the duality holds $\left(-\infty=-\infty\right)$.
\end{casenv}
\end{proof}

\subsection{The non-compact case}

It is possible to use the abstract theorem (Theorem \ref{Thm:AbsDuality}) for optimal transport problems with non-compact spaces $X,Y$. We will now describe the tools needed for this application and derive the Kantorovich duality theorem (Theorem \ref{thm:Duality}) in a weaker form.

Later on, we will show how the abstract theorem may be applied to other problems with compact spaces. The method used here may be applied to these problems if one wishes to deal with the non-compact case. In this part of the section we assume $X,Y$ are locally compact, $\sigma$-compact Polish spaces.\\

When $X\x Y$ is not compact, the space of regular Borel measures on $X\x Y$ is no longer the dual of $C(X\x Y)$, but rather the dual of $C_{0}\left(X\x Y\right)$ \index{c@$C_0(X\x Y)$} - The space of functions vanishing at infinity (a function $f$ is vanishing at infinity if for each $\varepsilon>0$ there exists a compact set such that $|f|<\varepsilon$ on its complement). For compact $X\x Y$, $C_0(X\x Y)=C(X\x Y)$.\\

If we wish to consider bounded cost functions that are not necessarily in $C_0$, the dual space will be $C_b^{*}=$ the finitely-additive regular measures. However, the optimal transport problem for such cost functions can be represented as minimization problem over $\sigma$-additive measures as in the compact case, due to the following lemma.

\begin{lem}
\label{lem:Villani Functionals}(Lemma 1.25 in \cite{Villani2003})
A non-negative linear functional in $C_{b}\left(X\x Y\right)^{*}$
with marginals in $\mathcal{P}\left(X\right),\mathcal{P}\left(Y\right)$
belongs to $\mathcal{P}\left(X\x Y\right)$ .
\end{lem}

The most frequently used cost functions in optimal transport are of the form $c(x,y)=d(x,y)^p$ for the metric $d$ and $p\geq 1$, which are bounded iff $X,Y$ are. In the case $X,Y$ are not bounded, we will need to use weighted function spaces:

\begin{defn}
Let $Z$ be a metric space and $B:Z\rightarrow\r_{+}$.
We define the \emph{weighted space} \index{weighted space}
\[
B\cdot C_{b}\left(Z\right)=\left\{ B\cdot g;\ g\in C_{b}\left(Z\right)\right\} 
\]
\end{defn}

\begin{claim}
$B\cdot C_{b}\left(Z\right)$ is a Banach space wrt the norm $\left|\left|f\right|\right|_B=\sup_{z}\frac{\left|f\left(z\right)\right|}{B\left(z\right)}$
\end{claim}
\begin{proof}
One can easily verify $||\cdot||_B$ is indeed a norm. Completeness is implied by the completeness of $C_{b}\left(Z\right)$ since 
$Bg_n\rightarrow Bg$ in $B\cdot C_b(Z)$ is equivalent to $g_n\rightarrow g$ in the supremum norm.
\end{proof}

To deduce the Kantorovich duality for non-compact $X,Y$ and non-bounded $c$, we will need the following version of Lemma \ref{lem:Villani Functionals} for weighted spaces:

\begin{prop}
\label{prop:Villani bounded}
Let $B:Z\rightarrow\r_{+}$ such that $\frac{1}{B}\in C_b(Z)$. If $F\in\left(B\cdot C_{b}\left(X\x Y\right)\right)^{*}$ is a non-negative linear functional and there exist  $\mu\in\mathcal{P}\left(X\right),\ \nu\in\mathcal{P}\left(Y\right)$
such that 
\[
F\left(\psi+\varphi\right)=\int_X\psi\d\mu+\int_Y\varphi\d\nu
\]
for all $\psi\in C_{b}\left(X\right),\ \varphi\in C_{b}\left(Y\right)$
then there exists $\pi\in\Pi(\mu,\nu)$ such that $F\left(f\right)\geq\int f\d\pi$ for all non-negative $f\in B\cdot C_{b}\left(X\x Y\right)$.
\end{prop}

\begin{proof}
Since $||g||_B\leq||\frac{1}{B}||_{\infty}\cdot||g||_{\infty}\ \forall g\in C_{b}\left(X\x Y\right)$, $F$ acts continuously on the subset $C_{b}\left(X\x Y\right)\subset B\cdot C_{b}\left(X\x Y\right)$,
hence by Lemma \ref{lem:Villani Functionals} there exists some $\pi\in\mathcal{P}\left(X\x Y\right)$
such that $F\left(g\right)=\int g\d\pi\ \forall g\in C_{b}\left(X\x Y\right)$. For any $n\in\mathbb{N}$ and $0\leq g\in C_{b}\left(X\x Y\right)$
\[
\int g\cdot\left(B\wedge n\right)\d\pi=F\left(g\cdot\left(B\wedge n\right)\right)\leq F\left(g\cdot B\right)
\]
hence by the monotone convergence theorem $g\cdot B\in L^{1}\left(\pi\right)$ and 
\[
\int g\cdot B\d\pi=\lim_{n\rightarrow\infty}\int g\cdot\left(B\wedge n\right)\d\pi\leq F\left(g\cdot B\right)
\]
and therefore $F\left(f\right)\geq\int f\d\pi$ for all non-negative $f\in B\cdot C_{b}\left(X\x Y\right)$.
$\pi\in\Pi(\mu,\nu)$ since $F$ and $\pi$ coincide on $C_b(X\x Y)$.
\end{proof}

We may now conclude a weaker version of the Kantorovich duality theorem (Theorem \ref{thm:Duality}):

\begin{thm}
Let $\mu\in\mathcal{P}\left(X\right),\ \nu\in\mathcal{P}\left(Y\right)$ and $c\in C(X\x Y)$. Assume there exists non-negative $a\in L^{1}\left(\mu\right)$ and $b\in L^{1}\left(\nu\right)$ such that 
$
\frac{1}{a+b},\frac{c}{a+b}\in C_b(X\x Y)
$
then
\[
\min_{\pi\in\Pi\left(\mu,\nu\right)}\int_{X\x Y}c\d\pi=\sup_{\left(\psi,\varphi\right)\in\Phi\left(c\right)}\int_{X}\psi\d\mu+\int_{Y}\varphi\d\nu
\]
\end{thm}

\begin{proof}
In Theorem \ref{Thm:AbsDuality} take:
\begin{itemize}
    \item $V=(a+b)\cdot C_b (X\x Y)$ ($c\in V$ since $\frac{c}{a+b}\in C_b(X\x Y)$).
    \item $K=a\cdot C_b(X)\x b\cdot C_b(Y)$
    \item $h(\psi,\varphi)=\psi+\varphi$
    \item $k^*(\psi,\varphi)=\int_{X}\psi\d\mu+\int_{Y}\varphi\d\nu$ (finite on $K$ since $a\in L^1(\m)$ and $b\in L^1(\n)$).
    \item Assumption \ref{assu:Sublinear} holds for $s((a+b)g)=(a\|g\|_{\infty},b\|g\|_{\infty})$
\end{itemize}
Then $\Pi(k^*,h)$ is the set of non-negative linear functionals in $V^*$ with marginals $\mu,\nu$. By Theorem \ref{Thm:AbsDuality}
$$
\min_{v^*\in\Pi\left(k^*,h\right)}v^*(c)
=
\sup_{\psi+\varphi\leq c}\int_{X}\psi\d\mu+\int_{Y}\varphi\d\nu
$$
By Proposition \ref{prop:Villani bounded}, for each $v^*\in\Pi(k^*,h)$ there exists $\pi\in\Pi(\mu,\nu)$ such that $v^*(c)\geq\int_{X\x Y}c\d\pi$ thus
$$\min_{v^*\in\Pi(k^*,h)}v^*(c)\geq\min_{\pi\in \Pi(\mu,\nu)}\int_{X\x Y}c\d\pi$$
and since
$
\int _{X\x Y}c\d\pi
\geq
\sup_{\psi+\varphi\leq c}\int_{X}\psi\d\mu+\int_{Y}\varphi\d\nu
$
for any $\pi\in\Pi(\mu,\nu)$, we conclude
$$
\min_{\pi\in\Pi\left(\mu,\nu\right)}\int_{X\x Y}c\d\pi
=
\sup_{\psi+\varphi\leq c}\int_{X}\psi\d\mu+\int_{Y}\varphi\d\nu$$
\end{proof}

\fbox{
\parbox{15cm}{
\begin{example}
In particular, the above theorem hold when $X,Y\subset\rn$, $c\left(x,y\right)=\left|x-y\right|^{2}$
and $\mu,\nu$ have finite second moments: $\int_X|x|^2\d\mu(x)+\int_Y|y|^2\d\nu(y)<\infty$. To see that just take $a\left(x\right)=\left|x\right|^{2}+1$
and $b\left(y\right)=\left|y\right|^{2}+1$.
\end{example}
}
}

\subsection{The convex version of the abstract theorem}

One can reformulate Theorem \ref{Thm:AbsDuality} as a \index{minmax theorem}minmax
theorem using the following claim:
\begin{claim}
\label{claim:sublinearconvex}
With the same notation as in theorem
\ref{Thm:AbsDuality}:

\begin{equation}
\inf_{k\in\Phi\left(u,h\right)}k^{*}\left(k\right)=\inf_{\left(k,q\right)\in K\x\q}\sup_{v^{*}\in V^{*}}v^{*}\left(u-h\left(k\right)+q\right)+k^{*}\left(k\right)\label{eq:infmax}
\end{equation}
\begin{equation}
\sup_{v^{*}\in\Pi\left(k^{*},h\right)}v^{*}\left(u\right)=\sup_{v^{*}\in V^{*}}\inf_{\left(k,q\right)\in K\x\q}v^{*}\left(u-h\left(k\right)+q\right)+k^{*}\left(k\right)\label{eq:maxinf}
\end{equation}
\end{claim}

\begin{proof}
If $u-h\left(k\right)+q\neq0$
then there exists a sequence $\left\{ v_{n}^{*}\right\} \subset V^{*}$
such that $v_{n}^{*}\left(u-h\left(k\right)+q\right)\rightarrow\infty$
thus 
\[
\sup_{v^{*}\in V^{*}}v^{*}\left(u-h\left(k\right)+q\right)+k^{*}\left(k\right)=\infty
\]
otherwise $k\in\Phi(u,h)$, therefore
\[
\inf_{\left(k,q\right)\in K\x\q}\sup_{v^{*}\in V^{*}}v^{*}\left(u-h\left(k\right)+q\right)+k^{*}\left(k\right)=\inf_{k\in\Phi\left(u,h\right)}k^{*}\left(k\right)
\]
 
For the second equality, notice that for any $v^{*}\notin\Pi\left(k^{*},h\right)$
either there exists $k\in K$ such that $k^{*}\left(k\right)-v^{*}\circ h\left(k\right)<0$
or there exists $q\in\q$ such that $v^{*}\left(q\right)<0$. In either
case
\[
\inf_{\left(k,q\right)\in K\x\q}v^{*}\left(u-h\left(k\right)+q\right)+k^{*}\left(k\right)=-\infty
\]
hence
\[
\sup_{v^{*}\in V^{*}}\inf_{\left(k,q\right)\in K\x\q}v^{*}\left(u-h(k)+q\right)+k^{*}\left(k\right)=\sup_{v^{*}\in\Pi\left(k^{*},h\right)}\inf_{\left(k,q\right)\in K\x\q}v^{*}\left(u-h\left(k\right)+q\right)+k^{*}\left(k\right)
\]
Moreover, for any $v^{*}\in\Pi\left(k^{*},h\right)$
\[
\inf_{\left(k,q\right)\in K\x\q}v^{*}\left(q\right)+k^{*}\left(k\right)-v^{*}\left(h\left(k\right)\right)=0
\]
therefore
\[
\sup_{v^{*}\in V^{*}}\inf_{\left(k,q\right)\in K\x\q}v^{*}\left(u-h\left(k\right)+q\right)+k^{*}\left(k\right)=\sup_{v^{*}\in\Pi\left(k^{*},h\right)}v^{*}\left(u\right)
\]
\end{proof}
As a corollary from the above claim, an equivalent way of writing
the duality (\ref{eq:absdual}) is
\begin{equation}
\max_{v^{*}\in V^{*}}\inf_{\left(k,q\right)\in K\x\q}v^{*}\left(u-h\left(k\right)+q\right)+k^{*}\left(k\right)
=
\inf_{\left(k,q\right)\in K\x\q}\sup_{v^{*}\in V^{*}}v^{*}\left(u-h\left(k\right)+q\right)+k^{*}\left(k\right)\label{eq:minmax=00003Dmaxmin}
\end{equation}
If we replace $k^{*}$ with a convex functional defined on a convex
set $K$ (instead of a sublinear functional defined on a convex cone
as in Theorem \ref{Thm:AbsDuality}) and replace $h$ with a $\q$-concave\index{Q@$\q$-concave} operator, meaning $$h(tk_1+(1-t)k_0)\geq th(k_1)+(1-t)h(k_0)\ \forall k_0,k_1\in K,\ t\in[0,1]$$
(instead of a $\q$-superlinear operator as in Theorem \ref{Thm:AbsDuality})
then equation (\ref{eq:infmax}) will remain true, whereas equation
(\ref{eq:maxinf}) might fail, for example when $V=K=\r$, $h(x)=x$, $\q=\{0\}$, $k^*(x)=1$, $u=0$ then
$$
\max_{v^{*}\in\Pi\left(k^{*},h\right)}v^{*}\left(u\right)=0<1=
\max_{v^{*}\in V^{*}}\inf_{\left(k,q\right)\in K\x\q}v^{*}\left(u-h\left(k\right)+q\right)+k^{*}\left(k\right)
$$
Therefore there may be a duality gap in (\ref{eq:absdual}) and the
duality theorem will fail. However, the minmax formula
(\ref{eq:minmax=00003Dmaxmin}) will remain true. The proof of (\ref{eq:minmax=00003Dmaxmin}) is based on a separation
theorem which is a corollary of the Hahn-Banach theorem\index{Hahn-Banach theorem} (Theorem \ref{thm:Hahn-Banach}):
\begin{thm}
\label{thm:HB Separation}(Theorem 5.67 in \cite{Aliprantis2006Hitch})
If $A,B$ are two convex disjoint subsets of a tvs and $A$
has a non-empty interior then there exists a non-zero continuous linear functional
$v^{*}$ separating $A,B$, i.e. there exists $r\in\r$ such that $v^{*}\left(a\right)\leq r\leq v^{*}\left(b\right)$
for every $a\in A,b\in B$.
\end{thm}

To prove (\ref{eq:minmax=00003Dmaxmin}) we may weaken Assumption \ref{assu:Sublinear} into its "local version":

\begin{assumption}
\label{assu:convex}There exists $N\subset V$ a neighborhood of $0$ and  $s:N\rightarrow K$
such that $s(v)\in\Phi(v,h)\ \forall v\in N$ and
$k^{*}\circ s$ is bounded above on $N$.
\end{assumption}

By the following lemma, we may assume $N$ is absorbing: $\forall x\in V$ there is $\alpha>0$ such that $\{tx;\ t\in\left[0,\alpha\right]\}\subset N$.\index{absorbing}

\begin{lem}

\label{lem:tvs absorbing}
(Lemma 5.63 in \cite{Aliprantis2006Hitch}) In a tvs every neighborhood of zero is absorbing.
\end{lem}

The implication \ref{assu:Sublinear} $\Rightarrow$ \ref{assu:convex} is immediate, the other direction is also true under some restrictions:

\begin{claim}
\label{claim:2.24 implies 2.4}
    If $\q$ is a convex cone and $h$ is positively-homogeneous then \ref{assu:convex} $\Rightarrow$ \ref{assu:Sublinear}.
\end{claim}
\begin{proof}

     Let $s:N\rightarrow K$ from Assumption \ref{assu:convex}. Extend $s$ to $\tilde{s}:V\rightarrow K$ by choosing  $\tilde{s}(v):=\frac{1}{t}s(tv)$ for some $t>0$ such that $tv\in N$ (exists since $N$ is absorbing by lemma \ref{lem:tvs absorbing}) for each $v\in V\setminus N$ then $$
    h\circ \tilde{s}(v)-v
    =
    h(\frac{1}{t}s(tv))-v
    =
    \frac{1}{t}\left[h(s(tv))-tv\right]\in\q
    $$
    and since $\tilde{s}=s$ on $N$, $k^*\circ\tilde{s}$ is bounded above on $N$.
\end{proof}
The convex version of the abstract theorem is the following:

\bigskip

\colorbox{thmcolor}{
\parbox{\linewidth}{
\begin{thm}
\label{thm:Convex Abstract}Let $k^{*}$ be a convex functional defined
on a convex set $K$, $\q\subset V$ a convex subset which is also closed under addition, and $h:K\rightarrow V$ a $\q$-concave operator. Then under
Assumption \ref{assu:convex}
\[
\max_{v^{*}\in V^{*}}\inf_{\left(k,q\right)\in K\x\q}v^{*}\left(q-h\left(k\right)\right)+k^{*}\left(k\right)=\inf_{k\in h^{-1}(\q)}k^{*}\left(k\right)
\]
\end{thm}
}}

\begin{proof}
We will use the Hahn-Banach separation theorem, similarly to the proof
of the Fenchel-Rockafellar duality in \cite{Villani2003} (see Theorem \ref{thm:FRduality} in here). Denote $m=\inf_{k\in h^{-1}(\q)}k^{*}\left(k\right)$,
$m<\infty$ by Assumption \ref{assu:convex}.
\begin{casenv}
\item $m=-\infty$. There exist a sequence $k_{n}\in h^{-1}(\q)$
such that $k^{*}\left(k_{n}\right)\underset{n\rightarrow\infty}{\longrightarrow}-\infty$.
Equivalently, there exists a sequence $\left(k_{n},q_{n}\right)\in K\x\q$
such that $h\left(k_{n}\right)=q_{n}$ and $k^{*}\left(k_{n}\right)\underset{n\rightarrow\infty}{\longrightarrow}-\infty$. Therefore for any $v^*\in V^*$
\[
\inf_{\left(k,q\right)\in K\x\q}v^{*}\left(q-h\left(k\right)\right)+k^{*}\left(k\right)\leq v^{*}\left(q_n-h\left(k_{n}\right)\right)+k^{*}\left(k_{n}\right)=k^{*}\left(k_{n}\right)\underset{n\rightarrow\infty}{\longrightarrow}-\infty
\]
thus the duality hold.
\item $m\in\mathbb{R}$. The inequality $\leq$ is immediate, since by restricting the infimum to $k,q$ satisfying $h(k)=q$ (which implies $k\in h^{-1}(\q)$) the infimum is increased.

We need to show the other direction, which means to prove the existence
of some $v^{*}\in V^{*}$ such that 
\begin{equation}
v^{*}\left(q-h\left(k\right)\right)+k^{*}\left(k\right)\geq m\ \forall \left(k,q\right)\in K\x\q\label{eq:hard direction convex}
\end{equation}
We define two subsets of $V\x\r$:
\[
\begin{array}{ccc}
A & := & \left\{ \left(h\left(k\right)-q,\lambda\right)\in V\x\r;\ k\in K,\ q\in\q,\ k^{*}\left(k\right)<\lambda\right\} \\
\\
B & := & \left\{ 0\right\} \x\left(-\infty,m\right)
\end{array}
\]
$B$ is convex and we will show that so is $A$: Let $t\in\left[0,1\right],\ k_{1},k_{2}\in K,\ q_{1},q_{2}\in\q,\ \lambda_{1},\lambda_{2}\in\r$
such that $\left(h\left(k_{1}\right)-q_{1},\lambda_{1}\right),\left(h\left(k_{2}\right)-q_{2},\lambda_{2}\right)\in A$,
meaning $k^{*}\left(k_{1}\right)<\lambda_{1},k^{*}\left(k_{2}\right)<\lambda_{2}$.
By the concavity of $h$,
\[
th\left(k_{1}\right)+\left(1-t\right)h\left(k_{2}\right)\leq h\left(tk_{1}+\left(1-t\right)k_{2}\right)
\]
namely there exists some $q\in\q$ such that 
\[
th\left(k_{1}\right)+\left(1-t\right)h\left(k_{2}\right)=h\left(tk_{1}+\left(1-t\right)k_{2}\right)-q
\]
Denote $k_t:=tk_{1}+\left(1-t\right)k_{2}$, $q_t:=q+tq_{1}+\left(1-t\right)q_{2}$ and $\lambda_t:=t\lambda_{1}+\left(1-t\right)\lambda_{2}$ then
\[
\begin{array}{ccccc}
t\left(h\left(k_{1}\right)-q_{1},\lambda_{1}\right)+\left(1-t\right)\left(h\left(k_{2}\right)-q_{2},\lambda_{2}\right) & = \\ \left(h\left(tk_{1}+\left(1-t\right)k_{2}\right)-q-\left[tq_{1}+\left(1-t\right)q_{2}\right],t\lambda_{1}+\left(1-t\right)\lambda_{2}\right)
 & = & \left(h\left(k_{t}\right)-q_{t},\lambda_{t}\right)
\end{array}
\]

By the convexity of $k^{*}$: $$k^{*}\left(k_{t}\right)\leq tk^*(k_1)+(1-t)k^*(k_2)<\lambda_{t}$$
so $\left(h\left(k_{t}\right)-q_{t},\lambda_{t}\right)\in A$ and we conclude $A$ is convex. By
definition of $m$, $A,B$ are disjoint so to use Theorem \ref{thm:HB Separation},
we will need to show $A$ has a non-empty interior. By Assumption
\ref{assu:convex}, there exists a neighborhood $N\subset V$ of $0$ and $s:N\rightarrow K$
 such that $k^{*}\left(s\left(v\right)\right)<M$
for every $v\in N$ and some $M\in\r$. Moreover, for any $v\in N$ there exists $q\in\q$
such that $v=h\circ s\left(v\right)-q$. Thus $N\x\left(M,\infty\right)$
is an open subset of $A$ thus $A$ has a non-empty interior. By Theorem \ref{thm:HB Separation},
there exists a continuous non-zero linear functional $\left(u^*,\alpha\right)\in V^{*}\x\r$
that separates $A$ and $B$, i.e. for any $\left(h\left(k\right)-q,\lambda\right)\in A$
and $\left(0,\mu\right)\in B$
\begin{equation}
u^*\left(h\left(k\right)-q\right)+\alpha\lambda\geq \alpha\mu\label{eq:convexInq}
\end{equation}
Since $\lambda$ is not bounded above, $\alpha$ must be non-negative.
If $\alpha=0$ then 
$
u^*\left(h\left(k\right)-q\right)\leq0
$
for any $k,q$. By Assumption \ref{assu:convex} $u^*\leq0$ on $N$ and since $N$ is absorbing (by Lemma \ref{lem:tvs absorbing}) $u^*=0$ on $V$ by contradiction to $(u^*,\alpha)\neq0$, so $\alpha>0$. By continuity, we may consider (\ref{eq:convexInq}) with
$\lambda=k^{*}\left(k\right)$ and $\mu=m$ and get
\[
u^*\left(h\left(k\right)-q\right)+\alpha k^{*}\left(k\right)\geq \alpha m\ \forall k,q
\]
Divide by $\alpha$ and conclude
\[
\frac{u^*\left(q-h\left(k\right)\right)}{-\alpha}+ k^{*}\left(k\right)\geq m
\]
so $v^{*}=-\frac{u^{*}}{\alpha}$ will satisfy (\ref{eq:hard direction convex}).
\end{casenv}
\end{proof}

\begin{rem}
If $K,\q$ are convex cones and $k^*,-h$ are sublinear then:
\begin{itemize}
    \item By Claim \ref{claim:sublinearconvex}, Theorem \ref{thm:Convex Abstract} implies the duality (\ref{eq:absdual}) in Theorem \ref{Thm:AbsDuality}.
    \item By Claim \ref{claim:2.24 implies 2.4}, under the assumptions of Theorem \ref{Thm:AbsDuality}, Assumption \ref{assu:convex} (which is used is Theorem \ref{thm:Convex Abstract}) is equivalent to Assumption \ref{assu:Sublinear} (which is used in Theorem \ref{Thm:AbsDuality}).
\end{itemize}

However, Theorem \ref{thm:Convex Abstract} does not make Theorem \ref{Thm:AbsDuality} obsolete, since Theorem \ref{Thm:AbsDuality} gives a characterization of the functionals $k^*$ for which the duality is finite. Such characterization does not appear in Theorem \ref{thm:Convex Abstract}.
\end{rem}

\begin{rem}
\label{rem:minmaxNonNeg}
$\ $
\begin{itemize}
    \item If $\q$ is a convex cone, Theorem \ref{thm:Convex Abstract} can be written as
$$
\max_{v^{*}|_{\q}\geq 0}\inf_{k\in K}(k^{*}-v^{*}\circ h)\left(k\right)   
=
\inf_{k\in h^{-1}(\q)}k^*(k)
$$
    \item The duality in Theorem \ref{thm:Convex Abstract} does not contain an element $u\in V$ since $h$ is concave (and not superlinear as in Theorem \ref{Thm:AbsDuality}) and thus can be changed by a constant. Formally, replacing $h$ with $h-u$ for a given $u\in V$  leads to
    \[
    \max_{v^{*}\in V^{*}}\inf_{\left(k,q\right)\in K\x\q}v^{*}\left(q-h\left(k\right)+u\right)+k^{*}\left(k\right)=\inf_{k\in\Phi(u,h)}k^{*}\left(k\right)
    \]
\end{itemize}

\end{rem}

\bigskip
The following known claim about subdifferentiabilty of convex functions is an immediate corollary of Theorem \ref{thm:Convex Abstract} (for a similar result - see Theorem 7.12 in \cite{Aliprantis2006Hitch}):
\begin{cor}
If a convex function $f:X\rightarrow\r$ on a tvs $X$ is upper semi-continuous at some point $y\in X$ then it is subdifferentiable there, namely there exists some $x^*\in X^*$ such that for any $x\in X$
$$
x^*(x-y)\leq f(x)-f(y)
$$
\end{cor}

\begin{proof}
Assume $f$ is upper semi-continuous at $y\in X$.
Use Theorem \ref{thm:Convex Abstract} with: $K=V=X$, $h(x)=x-y$,
$\q=\left\{ 0\right\}$ and $k^*(x)=f(x)-f(y)$. Choose $s(x):=x+y$ to verify Assumption \ref{assu:Sublinear}:
$h\circ s(x)-x=0\in\q$ and $k^*\circ s(v)=f(x+y)-f(y)$ is bounded above on a neighborhood of $0$ by upper semi-continuity of $f$. Therefore
$$
\max_{x^*\in X^*}\inf_{x\in X}x^*(y-x)+f(x)-f(y)=\inf_{x-y=0}f(x)-f(y)=0
$$
so for the maximizer $x^*$ and any $x\in X$ we get
$$x^*(y-x)+f(x)-f(y)\geq \inf_{z\in X}x^*(y-z)+f(z)-f(y)=0$$
\end{proof}

\subsection{Existence of minimizers}

In Theorem \ref{Thm:AbsDuality} we proved that $\max_{v^*\in\Pi(k^*,h)} v^*(c)$ is attained whenever it is finite. So it is only natural to ask whether the dual problem $\inf_{k\in\Phi(c,h)}k^*(k)$ is attained as well. We already mentioned that in the Kantorovich duality (Theorem \ref{thm:Duality}) the dual problem always has a solution, but for the general abstract problem, this might not be the case (see Example \ref{ex:nonexist} given in the next chapter).
\\
The following theorem states the existence of a solution to the dual problem for a "dense" subset of $h$-positive functionals on $K$. The idea behind the theorem is to replace $K$ with a smaller subset on which the dual problem is known to have a solution, and then use Theorem \ref{thm:Convex Abstract} with that subset.

\bigskip

\colorbox{thmcolor}{
\parbox{\linewidth}{
\begin{thm}
\label{thm:ExMin}
Let $\q\subset V$ and $K$ be two convex cones,  $h:K\rightarrow V$ a $\q$-superlinear operator and $k^*:K\rightarrow$ a $h$-positive sublinear functional. Assume there exists a convex $D\subset K$ such that for any $n\in\mathbb{N}$ Assumption \ref{assu:Sublinear} (or equivalently, by Claim \ref{claim:2.24 implies 2.4}, Assumption \ref{assu:convex}) holds when replacing $K$ with $n\cdot D$ (meaning $\exists s:V\rightarrow n\cdot D$) and $\inf_{(n\cdot D)\cap\Phi(c,h)}k^*(k)$ is attained. Then for any $\varepsilon>0$ there exists some $v^*\in V^*$ such that $v^*\circ h$ is $h$-positive, $\inf_{k\in\Phi(c,h)}v^*\circ h(k)$ is attained and
    $$||v^*\circ h-k^*||_D:=\sup_{k\in D}v^*\circ h(k)-k^*(k)\leq\varepsilon$$

\end{thm}
}}

\begin{rem}
If $D$ is the unit ball of a normed space and $h,k^*$ are linear on that space then $||\cdot||_D$ is the dual norm on the dual space. In the general case, $\|\cdot\|_D$ may be negative.
\end{rem}

\begin{proof}

Let $k_{1}\in D\cap\Phi(u,h)$, $\varepsilon>0$. Since $k^*$ is $h$-positive, by Lemma \ref{lem:EqCon} there exists $v^*_0\in\Pi(k^*,h)$. Choose
$$
n=\max\{2,\frac{2}{\varepsilon}\left(k^*(k_1)+|v^*_0(u)|\right)\}>1
$$
By the convex version of the abstract theorem (Theorem \ref{thm:Convex Abstract} - see also Remark \ref{rem:minmaxNonNeg}) using $n\cdot D$ instead of $K$ we get the duality
\begin{equation}
  \max_{v^*(q)\geq 0\ \forall q\in\q}[v^*(u)+\inf_{k\in n\cdot D}(k^*-v^*\circ h)(k)]
  =
   \max_{v^*\in V^*}\inf_{(k,q)\in (n\cdot D)\x\q}[v^*(u+q)+k^*(k)-v^*\circ h(k)]
  =
  \inf_{k\in\Phi(u,h)\cap n\cdot D} k^*(k)
  \label{eq:DualMinEx}
\end{equation}

We notice that
\[
  \inf_{k\in n\cdot D}(k^*-v^*\circ h)(k)  = 
  n\inf_{k\in D}(k^*-v^*\circ h)(k)  = 
  -n\sup_{k\in D}(v^*\circ h-k^*)(k)  = 
  -n||v^*\circ h-k^*||_D
\]

and the infimum on the RHS of (\ref{eq:DualMinEx}) is attained therefore equality (\ref{eq:DualMinEx}) can be written as
\begin{equation}
\max_{v^*(q)\geq 0\ \forall q\in\q}v^*(u)-n||v^*\circ h-k^*||_D= k^*(k_n)   
\label{eq:dualExMin2}
\end{equation}

where $k_n\in (n\cdot D)\cap\Phi(u,h)$ is the minimizer in (\ref{eq:DualMinEx}).
Denote the maximizer in (\ref{eq:dualExMin2}) by $ v^*_n$ and $k^*_n:= v^*_n\circ h$ so by definition $ v^*_n\in\Pi(k^*_n,h)$. Since $k_n\in\Phi(u,h)$ we get $v^*_n(u)\leq v^*_n\circ h(k_n)= k^*_n(k_n)$ and since $k_n\in n\cdot D$ we get
$$k^*_n\left(k_n\right)-k^*\left(k_n\right)\leq n\cdot\sup_{k\in D}k^*_n(k)-k^*(k)=n\cdot||k^*_n-k^*||_D$$
so combining both inequalities with (\ref{eq:dualExMin2}) leads to
$$
 k^*(k_n)  =    v^*_n(u)-n||k^*_n-k^*||_D 
  \leq  k^*_n(k_n)-n||k^*_n-k^*||_D  \leq  k^*(k_n)
$$

So the inequality above is actually an equality so we conclude $v_n^*(u)=k_n^*(k_n)$ and
$$k^*_n(k_n)\geq\inf_{k\in\Phi(u,h)}k^*_n(k)\geq  v^*_n(u)=k^*_n(k_n)\Rightarrow\inf_{k\in\Phi(u,h)}k^*_n(k)=k^*_n(k_n)$$

Hence the infimum is attained. We are left with showing $||k^*_n-k^*||_D\leq\varepsilon$. By definition of $\Pi(k^*,h)$ $v^*_0\circ h\leq k^*$ thus $$v^*_0(u)-n||v^*_0\circ h-k^*||\geq v^*_0(u)$$ hence for $v^*_n$ the maximizer of (\ref{eq:dualExMin2}) we have $v^*_n(u)-n||k^*_n-k^*||\geq v^*_0(u)$ as well which implies

$$k^*_n(k_1)=  v^*_n\circ h(k_1)\geq  v^*_n(u)\geq n||k^*_n-k^*||_D+v^*_0(u)\geq nk^*_n\left(k_1\right)-nk^*\left(k_1\right)+v^*_0(u)$$

Therefore
$$k^*_n(k_1)
\leq
\frac{nk^*\left(k_1\right)-v^*_0(u)}{n-1}
\leq
2k^*\left(k_1\right)+|v^*_0(u)|$$
($n\mapsto\frac{n}{n-1}$ is decreasing on $(1,\infty)$ and $n\geq 2$) so by the last two inequalities we conclude
$$
||k^*_n-k^*||_D
\leq
\frac{k^*_n(k_1)-v^*_0(u)}{n}
\leq
\frac{k^*_n(k_1)+|v^*_0(u)|}{n}
\leq
\frac{2}{n}\left(k^*(k_1)+|v^*_0(u)|\right)
\leq
\varepsilon$$
\end{proof}

We now give an example showing how the above theorem may be applied to the Kantorovich duality (Theorem \ref{thm:Duality}). Although this example does not have any real significance, since we know the dual problem in the Kantorovich duality always admits a solution in the compact case, it is a good demonstration of the theorem above. 
In the next chapter, we will apply this theorem similarly to the problem of vector measures and get a result of higher value (see Theorem \ref{thm:ExMinVec}).

\begin{example}
\label{ex:ExMin}
Let $(X,\m),(Y,\n)$ be two compact finite measure spaces.
We can show the existence of a solution to the dual problem in  (\ref{eq:ClaDuality}) using Theorem \ref{thm:ExMin} 

\begin{itemize}
    \item $K=C(X)\x C(Y)$
    \item $D=\{(\psi,\varphi)\in K;\ |\varphi(y)|\leq 1,|\varphi(y)-\varphi(z)|\leq d(y,z)\ \forall y,z\in Y\}$
    \item $k^*(\psi,\varphi):=\int_X \psi\d\m+\int_Y\varphi\d\n$.
     \item $V=C(X\x Y)$ equipped with the supremum norm.
    \item $h(\psi,\varphi)=\psi+\varphi$
    \item $\q$ is the subset of non-negative functions in $V$.
    \item For any $n\in\mathbb{N}$ take $s(v):=(||v||_{\infty},0)\in\Phi(v,h)\cap n\cdot D$ which satisfies Assumption \ref{assu:Sublinear}.
\end{itemize}

We want to show that $\inf_{(n\cdot D)\cap\Phi(c,h)}k^*$ is attained:
Let $(\psi_m,\varphi_m)\in\Phi(c,h)\cap (n\cdot D)$ be a minimizing sequence of $k^*$. We may assume $\psi_m(x)=\max_y\{c(x,y)-\varphi_m(y)\}$. $\psi_m,\varphi_m$ are uniformly bounded and Lipschitz continuous if $c$ is Lipschitz continuous. By the Arzela-Ascoli theorem\index{Arzela-Ascoli theorem} (Theorem A.4.6 in \cite{shalit2017first}) $(\psi_m,\varphi_m)$ has a subsequence converging uniformly to some $(\psi_0,\varphi_0)\in K$. The uniform convergence respects the order between functions so $(\psi_0,\varphi_0)\in\Phi(c,h)\cap(n\cdot D)$. By the uniqueness of the limit:
$$
\int_X\psi_m\d\mu+\int_Y\varphi_m\d\nu\rightarrow\int_X\psi_0\d\mu+\int_Y\varphi_0\d\nu=\inf_{(n\cdot D)\cap\Phi(c,h)}k^*
$$
So $\inf_{(n\cdot D)\cap\Phi(c,h)}k^*(k)$ is attained with $k=(\psi_0,\varphi_0)$. \\

By Theorem \ref{thm:ExMin}, for any $\varepsilon>0$ there exists $\pi\in\mm(X\x Y)=V^*$ such that

$$
\inf\left\{\int_{X\x Y}\psi(x)+\varphi(y)\d\pi(x,y);\ (\psi,\varphi)\in\Phi(c,h)\right\}
$$

is attained and

$$
||\pi\circ h-k^*||_D=\sup_{(\psi,\varphi)\in D}\int_{X\x Y}\psi(x)+\varphi(y)\d\pi(x,y)-\int_X\psi\d\m-\int_Y\varphi\d\n\leq\varepsilon
$$

The supremum is bounded and $\psi$ may be any continuous function thus $\pi$ has marginal $\mu$ on $X$. Denote the marginal of $\pi$ on $Y$ by $\nu'$ then

$$
\inf\left\{\int_X\psi\d\m+\int_Y\varphi\d\nu';\ (\psi,\varphi)\in\Phi(c,h)\right\}
$$

is attained and

$$
||\nu'-\n||_D=\sup_{|\varphi(y)|\leq 1,|\varphi(y)-\varphi(z)|\leq d(y,z)}\int\varphi\d (\nu'-\n)\leq\varepsilon
$$

the norm above sometimes called the "Kantorovich-Rubinstein" norm\index{Kantorovich-Rubinstein norm} and is known to metricize the weak convergence for probability measures (see Proposition 7.1.5 in \cite{Ambrosio2005}, Theorem 8.3.2 in \cite{bogachev2007measure}, Theorem 5.9 in \cite{santambrogio2015optimal} or Theorem 7.12 in \cite{Villani2003}).
\end{example}

\subsection{Related theorems}

The following duality theorem can be found in \cite{Rachev2006}: Let
\begin{itemize}
\item $K,V$ be two vector spaces with dual spaces $K^{*},V^{*}$.
\item $\Pi:K^{*}\rightarrow2^{V^{*}}$ for which $\Pi\left(K^{*}\right):=\cup_{k^{*}\in K^{*}}\Pi\left(k^{*}\right)\subset V^{*}$
is a non-empty convex cone.
\item $c:\Pi\left(K^{*}\right)\rightarrow\r\cup\left\{ \infty\right\} $
is additive and positively-homogeneous.
\item $C\left(k^{*}\right):=\inf\left\{ c\left(v^{*}\right);\ v^{*}\in\Pi\left(k^{*}\right)\right\} $
\end{itemize}
\begin{thm}
\label{thm:Rachev}
(Theorem 4.6.1 part $\left(ii\right)$ in \cite{Rachev2006}) If $C$
is weak*-lower semi-continuous then

\[
C\left(k^{*}\right)=\sup\left\{ k^{*}\left(k\right);k\in K,\ \left\langle k,x^{*}\right\rangle \leq C\left(x^{*}\right)\ \forall x^{*}\in K^{*}\right\} 
\]
\end{thm}

To see the connection with Theorem \ref{Thm:AbsDuality}, take:
\begin{itemize}
    \item For any $k^*\in K^*$:
    \[
    \Pi\left(k^{*}\right)=\left\{ v^{*}\in V^{*};\ v^{*}\circ h= k^{*},\ v^{*}(q)\geq0\ \forall q\in\q\right\} 
    \]
    \item $\Pi(K^*)$ is the cone of functionals that are non-negative on $\q$.
    \item $c$ is taken as an element of $V$, which is linear on $V^*$ thus additive and positively-homogeneous on $\Pi(K^*)$.
\end{itemize}
then
$C\left(k^{*}\right)=\inf_{v^*\in\Pi(k^*,h)}v^*(c)$. So it seems the theorem above is more general than Theorem \ref{Thm:AbsDuality},  but Theorem \ref{Thm:AbsDuality} has some advantages: 
\begin{itemize}
    \item The primal problem admits a solution, and by Theorem \ref{thm:ExMin} the dual problem admits solution on some "dense" set.
    \item Instead of semicontinuity of $V$ we have Assumption \ref{assu:Sublinear}.
    \item $k^*$ may be sublinear on a convex cone, rather than linear on a linear space.
    \item We have a characterization of the functionals $k^*$ for which the duality is finite.
\end{itemize}

Another theorem, which is closely related to Theorem \ref{thm:Convex Abstract}, is the extended Hahn-Banach-Lagrange theorem\index{Hahn-Banach-Lagrange theorem}. The original Hahn-Banach-Lagrange theorem (Theorem 1.11 in \cite{Simons2008} or Corollary 3.2 in \cite{DinhHB}) is a generalization of the Hahn-Banach theorem for finite sublinear functionals and non-topological vector spaces. The extended version allows infinite values and topology:

\begin{thm}
\label{thm:EHBL}
(Extended Hahn-Banach-Lagrange, Theorem 3.1 in \cite{DinhHB}) Let $S:Y\rightarrow\r\cup\{\infty\}$ be a lower semicontinuous sublinear functional on a Hausdorff locally convex t.v.s $Y$, $C$ a non-empty convex set, $f:C\rightarrow\r$ a convex functional and $g:C\rightarrow Y$ convex wrt to the relation $x\leq y\iff S(x-y)\leq 0$. If
\begin{equation*}
\exists x\in C,\ \alpha\in\r\ s.t.\ g(x)\in\text{int}\{y\in Y;\ S(y)\leq\alpha\}
\end{equation*}
then
$$\max_{y^*\leq S}\inf_C y^*\circ g+f=\inf_C S\circ g+f$$
whenever the RHS is finite.
\end{thm}

\b

By taking
$$
S(x)=
\begin{cases} 
      0 & ;x\in -\q
      \\
      \infty & ;x\notin -\q 
   \end{cases}
$$
we get another related theorem (which is equivalent to the theorem above by \cite{DinhHB}):

\begin{thm}
(Farkas' lemma\index{Farkas' lemma} for cone-convex systems, Theorem 2.1 in \cite{DinhFarkas} and Theorem 3.1 in \cite{DinhHB}) Let $C$ be a non-empty and convex set, $\q$ a closed convex cone in a Hausdorff locally convex tvs $Y$, $f:C\rightarrow\r$ a proper convex functional and $g:C\rightarrow Y$ convex wrt to the relation $x\leq y\iff y-x\in \q$. If
\begin{equation}
\label{eq:SC1}
\exists x\in C,\ s.t.\ g(x)\in-\text{int}\q
\end{equation}
then
$$\max_{y^*|_{\q}\geq0}\inf_C y^*\circ g+f=\inf_{g^{-1}(-\q)} f$$
\end{thm}

Theorem \ref{thm:Convex Abstract} implies the above theorem by taking $V=Y,\ K=C,\ k^*=f,\ h=-g$. In Theorem \ref{thm:Convex Abstract}, $V$ may be a tvs that is not Hausdorff locally convex, and $\q$ may be a convex subset closed under addition that is not a convex cone. Moreover, condition (\ref{eq:SC1}) translates into: "$h(K)$ contains an interior point of $\q$", which is stronger than Assumption \ref{assu:convex}:
\begin{claim}
If $h(K)$ contains an interior point of $\q$ then Assumption \ref{assu:convex} holds.
\end{claim}
\begin{proof}
Let $k_0\in K$ such that $h(k_0)\in\text{int}\q$. There exists a neighborhood $N\subset\q$ of $h(k_0)$. Define $N_0:=h(k_0)-N$ (which is a neighborhood of $0$) and $s:N_0\rightarrow K$ by $s(v):=k_0$. For any $v\in N_0$, $h\circ s(v)-v=h(k_0)-v\in N\subset \q$ and $k^*\circ s=k^*(k_0)$ is constant on $N_0$ therefore bounded above.
\end{proof}

Notice that under Assumption (\ref{assu:convex}), $\q$ may have an empty interior.\\
The Farkas' lemma is a special case of the "fundamental duality formula"\index{fundamental duality formula}:
\begin{thm}
(Theorem 2.7.1 in \cite{zalinescu2002convex})
Let $X,Y$ be two Hausdorff locally convex tvs, $\Phi:X\x Y\rightarrow\r\cup\{\infty\}$ a proper convex function and assume
\begin{equation}
\label{eq:FDFass}
    \exists x_0\in X\ s.t.\ \Phi(x_0,0)<\infty\ and\ \Phi(x_0,\cdot)\ is\ continuous\ at\ 0
\end{equation}
Then
$$
\inf_{x\in X}\Phi(x,0)
=
\max_{y^*\in Y^*}\left(\inf_{(x,y)\in X\x Y}\Phi(x,y)-\left\langle y^*,y\right\rangle\right)
$$
\end{thm}

\bigskip

This theorem is also implied by Theorem \ref{thm:Convex Abstract}: $\Phi$ is convex thus so is $y\mapsto\inf_x\Phi(x,y)$ since for any $x_0,x_1\in X,\ y_0,y_1\in Y,\ t\in[0,1]$ we have

$$
\inf_x\Phi(x,ty_1+(1-t)y_0)
\leq
\Phi(tx_1+(1-t)x_0,ty_1+(1-t)y_0)
\leq
t\Phi(x_1,y_1)+(1-t)\Phi(x_0,y_0)
$$
and by taking the infimum over $x_0,x_1\in X$ we get
$$
\inf_x\Phi(x,ty_1+(1-t)y_0)
\leq
t\inf_x\Phi(x,y_1)+(1-t)\inf_x\Phi(x,y_0)
$$
Therefore we may choose $K=\{y\in Y;\ \inf_x\Phi(x,y)<\infty\}\subset Y=V,\ h(y)=y,\ \q=\{0\}$ and $k^*(y)=\inf_x\Phi(x,y)$.
We verify Assumption \ref{assu:convex}: By (\ref{eq:FDFass}) there is a neighborhood $N\subset Y$ of $0$ on which $\inf_x\Phi(x,\cdot)\leq\Phi(x_0,\cdot)\leq\Phi(x_0,0)+1$ thus $N\subset K$. Define $s:N\rightarrow K$ by $s(y):=y$, then $h\circ s(y)-y=0\in\q$ and
$$k^*\circ s(y)=\inf_x\Phi(x,y)\leq \Phi(x_0,y)\leq \Phi(x_0,0)+1\ \forall y\in N$$

\begin{rem}
The full version of this theorem gives possible alternatives to condition (\ref{eq:FDFass}), some are weaker than (\ref{eq:FDFass}).
\end{rem}

\bigskip

For more theorems similar to Theorem \ref{thm:Convex Abstract}, we refer to Theorem 10 in \cite{GRAD201086} and Theorem 3.1 in \cite{Daniele2007}. Both theorems use conditions which rely on the notions of quasi-interior and quasi-relative interior and may be harder to verify. Furthermore, the theorems assume the underlying spaces to be Hausdorff locally convex/normed (respectively).

\newpage{}

\section{Optimal transport for vector measures}

In this chapter, I will present a transport problem for vector measures,
which is a generalization of the scalar case and a special case of the abstract duality theorem.
I will also present some new results concerning the works of Blackwell in \cite{blackwell1951} and use
it to characterize the cases for which there exists a transport
plan. Moreover, I will give conditions to the existence of transport maps in the semi-discrete case/general case.

\bigskip

To the end of this chapter, unless stated otherwise, $X$ and $Y$ are
compact metric spaces equipped with their Borel $\sigma$-algebra.
As described in the introduction, given two Borel vector measures $\left\{ \mu_{i}\right\} _{i=1}^{n}\subset\mathcal{M}_+\left(X\right),\ \left\{ \nu_{i}\right\} _{i=1}^{n}\subset\mathcal{M}_+\left(Y\right)$
and a continuous cost function $c\in C\left(X\times Y\right)$,
we generalize the problem of Monge
\[
\inf_{T}\int_{X\times Y}c\left(x,Tx\right)\d\mu\left(x\right)
\]
by taking the infimum over all measurable maps $T:X\rightarrow Y$
such that $T_{\#}\mu_{i}=\nu_{i}\ \forall i$.

To define the relaxation of this problem we use the
following set of transport plans:
\begin{defn}
\label{def:VectorPlans}Let $\mu=\left\{ \mu_{i}\right\} _{i=1}^{n}\subset\mathcal{M}_+\left(X\right),\ \nu=\left\{ \nu_{i}\right\} _{i=1}^{n}\subset\mathcal{M}_+\left(Y\right)$
two vector measures and $\eta_i:X\x Y\rightarrow\r,\ i=1,\ldots,n$. 
We define $\Pi\left(\mu,\nu,\eta\right)$ to be the set of measures $\pi\in\mathcal{M}_+\left(X\x Y\right)$
such that $\eta_{i}\pi\in\Pi\left(\mu_{i},\nu_{i}\right)\ \forall i$ ($\eta_{i}\pi$ is the measure for which $\frac{\d\eta_{i}\pi}{\d\pi}=\eta_i$).
\end{defn}

If we assume $\mu_{i}<<\left|\mu\right|\ \forall i$
for some measure $\left|\mu\right|$ and $\eta_i=\frac{\d\mu_{i}}{\d\left|\mu\right|}$, the Kantorovich relaxation to the problem is
\[
\inf_{\pi\in\Pi\left(\mu,\nu,\eta\right)}\int_{X\times Y}c\left(x,y\right)\d\pi\left(x,y\right)
\]
since there is a natural
inclusion from the set of transport maps $T:X\rightarrow Y$ satisfying 
$T_{\#}\mu_{i}=\nu_{i}\ \forall i$ into the set of measures $\Pi\left(\mu,\nu,\eta\right)$,
defined by $T\mapsto(\text{Id}_X\x T)_\#|\m|$.
A quick check: if $T_{\#}\mu_{i}=\nu_{i}\ \forall i$ then for every
$i$ and every $\psi\in C\left(X\right),\varphi\in C\left(Y\right)$ 
\[
\begin{array}{ccccc}
\int_{X\x Y}\psi\left(x\right)\frac{\d\mu_{i}}{\d\left|\mu\right|}\left(x\right)\d(\text{Id}_X\x T)_\#|\m|(x,y) & = & \int_{X}\psi\left(x\right)\frac{\d\mu_{i}}{\d\left|\mu\right|}\left(x\right)\d\left|\mu\right|\left(x\right) & = & \int_{X}\psi\left(x\right)\d\mu_{i}\left(x\right)\\
\int_{X\x Y}\varphi\left(y\right)\frac{\d\mu_{i}}{\d\left|\mu\right|}\left(x\right)\d(\text{Id}_X\x T)_\#|\m|\left(x,y\right) & = & \int_{X}\varphi\left(Tx\right)\d\mu_{i}\left(x\right) & = & \int_{Y}\varphi\left(y\right)\d\nu_{i}\left(y\right)
\end{array}
\]
Therefore $(\text{Id}_X\x T)_\#|\m|\in\Pi(\mu,\nu,\eta)$.

\subsection{Infinite-dimensional vector measures}

To apply the optimal transport problem to measures with an infinite-dimensional range, we need to define what those measures are. We assume $(\Omega,\mathcal{A})$ is a measurable space and $S$ is a Banach space.

\begin{defn}
\index{S@$S$-valued measure} A \emph{$S$-valued measure} is a function $\mu:\mathcal{A}\rightarrow S$ that satisfies $\mu(\emptyset)=0$ and  $\mu(\dot\cup_{n=1}^{\infty} A_n)=\sum_{n=1}^{\infty}\mu(A_n)$ (where the series converges in norm) for every sequence of pairwise disjoint measurable subsets $\{A_n\}_{n=1}^{\infty}\subset\mathcal{A}$. 
\end{defn}

We wish to work with vector measures that have vector-valued densities wrt some non-negative measure. There are several ways to define such integrals, one of them is the "Bochner integral"\index{Bochner integral}. The definition of this integral is similar the the definition of the Lebesgue integral. Following is a short summary of  the Bochner integral theory:

\begin{defn}
Let $f:\Omega\rightarrow S$ and $|\mu|\in\mm_+(\Omega)$.
\begin{enumerate}
    \item $f$ is called a \emph{simple function}\index{simple function} if there exists $n\in\mathbb{N},\ \{A_i\}_{i=1}^n\subset\mathcal{A}$ and $\{s_i\}_{i=1}^n\subset S$ such that $$f(\omega)=\sum_{i=1}^n\chi_{A_i}(\omega)s_i\ \forall\omega\in\Omega$$
    \item The integral over $A\in\mathcal{A}$ of a simple functions wrt $|\mu|$ is
    $$
    \int_{A}\sum_{i=1}^n\chi_{A_i}s_i\d|\mu|
    =
    \sum_{i=1}^n|\mu|(A_i\cap A)s_i
    $$
    \item $f$ is called \emph{measurable}\index{measurable} if it is the pointwise limit of a sequence of simple functions.
    \item A measurable function $f$ is said to be \emph{Bochner integrable}\index{Bochner integrable} (wrt $|\mu|$) is there exists a sequence of simple functions $\{f_n:\Omega\rightarrow S\}_{n=1}^{\infty}$ such that
    $\lim_{n\rightarrow\infty}\int_{\Omega}\|f-f_n\|\d|\mu|=0$.
    In this case, its \emph{Bochner integral} over $A\in\mathcal {A}$ is defined by
    $
    \int_Af\d|\mu|=\lim_{n\rightarrow\infty}\int_Af_n\d|\mu|
    $.
    \item The \index{variation}\emph{variation} of a $S$-valued measure is a non-negative measure $V(\mu)\in\mm_+(\Omega)$ defined by
    $$V(\mu)(A):=\sup_{\{A_i\}_i}\sum_{B\in\{A_i\}_i}\|\mu(B)\|_S$$
    where the supremum is taken over all measurable finite partitions of $A$. If $V(\mu)$ is finite, $\mu$ is said to be a measure of \emph{bounded variation}.\index{bounded variation}
\end{enumerate}
\end{defn}

\begin{rem}
For any $|\mu|\in\mm_+(\Omega)$, $V(|\mu|)=|\mu|$.
\end{rem}

\begin{prop}
\label{prop:Bochner}
Let $f:\Omega\rightarrow S$ and $|\mu|\in\mm_+(\Omega)$.
\begin{enumerate}
    \item (Lemma 11.39 in \cite{Aliprantis2006Hitch}) If $f$ is measurable then $\|f\|:\Omega\rightarrow\r$ is measurable.
    \item (Theorem 2 in \cite{Diestel1977} Chapter II.2) Assume $f$ is measurable. Then $f$ is Bochner integrable iff $\int_{\Omega}\|f\|\d|\mu|<\infty$, and in this case (Theorem 4 part $(ii)$ in \cite{Diestel1977} Chapter II.2)
    $$
    \left\|\int_{A}f\d|\mu|\right\|\leq\int_{A}\|f\|\d|\mu|\ \forall A\in\mathcal{A}
    $$
    \item (Lemma 11.41 in \cite{Aliprantis2006Hitch}) The Bochner integral of a Bochner integrable function does not depend on the choice of sequence of simple functions.
    \item (Theorem 4 part $(iii)$ in \cite{Diestel1977} Chapter II.2) If $f$ is Bochner integrable, $\mu:\mathcal{A}\rightarrow S$ defined by $\m(A):=\int_Af\d|\m|\in S$ is a $S$-valued measure, and (Theorem 4 part $(iv)$ in \cite{Diestel1977} Chapter II.2) its variation is $V(\mu)(A)=\int_A\|f\|\d|\m|$ (in particular $V(\mu)<<|\mu|$).
    \item (Corollary 5 in \cite{Diestel1977} Chapter II.2) If $\int_Af\d|\mu|=0$ for all $A\in\mathcal{A}$ then $f=0$ $|\mu|$-almost everywhere.
    \item (Theorem 6 in \cite{Diestel1977} Chapter II.2 or Lemma 11.45 in \cite{Aliprantis2006Hitch}) If $f$ is Bochner integrable, $S'$ is another Banach space and $T:S\rightarrow S'$ is a bounded linear operator then $Tf:\Omega\rightarrow S'$ is Bochner integrable and $\int_{\Omega}Tf\d|\mu|=T\int_{\Omega}f\d|\mu|$
\end{enumerate}
\end{prop}

Using part 4 of the above proposition we define the space of measures that will be relevant to our problem:

\begin{defn}
 $\mm(\Omega,S)$ \index{m@$\mm$} is the space $S$-valued measures $\m$ of the form $\m(A):=\int_A\eta\d|\m|\in S$ for any measurable subset $A\subset \Omega$ where $|\m|\in\mm_+(\Omega)$ and $\eta:\Omega\rightarrow S$ is Bochner integrable. In this case we denote $\mu=(\eta,|\m|)$ and $\eta=\dm$.
\end{defn}

By part 4 of Proposition \ref{prop:Bochner}, any measure in $\mm(\Omega,S)$ is of bounded variation, and moreover we conclude that it is always possible to take $|\mu|=V(\mu)$:
\begin{claim}
\label{claim:VarRep}
Let $\eta$ be Bochner integrable wrt $|\mu|\in\mm_+(\Omega)$ such that $\eta\neq0$ $|\mu|$-almost everywhere, then $\mu:=(\eta,|\mu|)=(\frac{\eta}{\|\eta\|},V(\mu))$, or equivalently $\frac{\dm}{\|\dm\|}=\frac{\d\mu}{\d V(\mu)}$.
\end{claim}
\begin{proof}
Since $\eta$ is Bochner integrable wrt $|\mu|$, there exists a sequence of simple functions $\eta_n$ such that
\begin{equation}
\label{eq:Bochner}
\lim_{n\rightarrow\infty}\int_{\Omega}\|\eta_n-\eta\|\d|\mu|=0
\end{equation}
We may assume $\eta_n\neq0$ everywhere (just choose $\eta_n=\frac{1}{n}s\neq0$ whenever $\eta_n=0$), therefore by part 4 of Proposition \ref{prop:Bochner}

\[
\begin{array}{ccc}
     \int_\Omega\|\frac{\eta_n}{\|\eta_n\|}-\frac{\eta}{\|\eta\|}\|\d V(\mu)
     & = &
     \int_\Omega\|\frac{\eta_n}{\|\eta_n\|}-\frac{\eta}{\|\eta\|}\|\|\eta\|\d|\mu|
     \\ & = &
     \int_\Omega\|\frac{\eta_n}{\|\eta_n\|}\|\eta\|-\eta\|\d|\mu|
     \\ & = &
     \int_\Omega\|\frac{\eta_n}{\|\eta_n\|}\|\eta\|-\eta_n+\eta_n-\eta\|\d|\mu|
     \\ & \leq &
     \int_\Omega\|\frac{\eta_n}{\|\eta_n\|}\|\eta\|-\eta_n\|+\|\eta_n-\eta\|\d|\mu|
     \\ & = &
     \int_\Omega|\frac{\|\eta\|}{\|\eta_n\|}-1|\cdot\|\eta_n\|+\|\eta_n-\eta\|\d|\mu|
     \\ & = &
     \int_\Omega|\|\eta\|-\|\eta_n\||+\|\eta_n-\eta\|\d|\mu|
     \\ & \leq &
     2\int_\Omega\|\eta_n-\eta\|\d|\mu|
     \underset{n\rightarrow\infty}{\longrightarrow}0
\end{array}
\]

Assume $\eta_n(\omega)=\sum_{i=1}^{M_n}\chi_{A_i^n}(\omega)s_i^n$ where $s_i^n\neq0\ \forall i$ and $\{A_i^n\}_{i=1}^{M_n}$ are pairwise disjoint for each $n\in\mathbb{N}$, then $\|\eta_n\|=\sum_{i=1}^{M_n}\chi_{A_i^n}\|s_i^n\|$ and $\frac{\eta_n}{\|\eta_n\|}=\sum_{i=1}^{M_n}\chi_{A_i^n}\frac{s_i^n}{\|s_i^n\|}$ which is simple thus $\frac{\eta}{\|\eta\|}$ is Bochner integrable wrt $V(\mu)$. Again by part 4 of Proposition \ref{prop:Bochner}, for any $A\in\mathcal{A}$ we have
\[
\begin{array}{ccc}
     \int_A\|\eta_n-\eta\|\d|\mu| & = & \sum_{i=1}^{M_n}\int_{A_i^n\cap A}\|\eta_n-\eta\|\d|\mu|
     \\ & \geq &
     \sum_{i=1}^{M_n}\left|\int_{A_i^n\cap A}\|\eta_n\|-\|\eta\|\d|\mu|\right|
     \\ & = &
     \sum_{i=1}^{M_n}\left|\int_{A_i^n\cap A}\|s_i^n\|-\|\eta\|\d|\mu|\right|
     \\ & \geq &
     \sum_{i=1}^{M_n}\left|\left(|\mu|(A_i^n\cap A)\|s_i^n\|-V(\mu)(A_i^n\cap A)\right)\right|
     \\ & = &
     \sum_{i=1}^{M_n}\left|\left(|\mu|(A_i^n\cap A)\|s_i^n\|-V(\mu)(A_i^n\cap A)\right)\right|\|\frac{s_i^n}{\|s_i^n\|}\|
     \\ & \geq &
     \left\|\sum_{i=1}^{M_n}|\mu|(A_i^n\cap A)s_i^n-\sum_{i=1}^{M_n}V(\mu)(A_i^n\cap A)\frac{s_i^n}{\|s_i^n\|}\right\|
     \\& = &
     \left\|\int_A\eta_n\d|\mu|-\int_A\frac{\eta_n}{\|\eta_n\|}\d V(\mu)\right\|
\end{array}
\]
Thus $\underset{n\rightarrow\infty}{\lim}\int_A\eta_n\d|\mu|-\int_A\frac{\eta_n}{\|\eta_n\|}\d V(\mu)=0$ and
$$
\int_A\frac{\eta}{\|\eta\|}\d V(\mu)
=
\underset{n\rightarrow\infty}{\lim}\int_A\frac{\eta_n}{\|\eta_n\|}\d V(\mu)
=
\underset{n\rightarrow\infty}{\lim}\int_A\eta_n\d |\mu|
=
\int_A\eta\d |\mu|
$$
\end{proof}

\begin{cor}
\label{cor:NorDen}
If $\mu=(\eta_1,|\mu_1|)=(\eta_2,|\mu_2|)$ then $\frac{\eta_1}{\|\eta_1\|}=\frac{\eta_2}{\|\eta_2\|}$ $V(\mu)$-almost everywhere.
\end{cor}

\begin{proof}
By Claim \ref{claim:VarRep} $(\frac{\eta_1}{\|\eta_1\|},V(\mu))=(\frac{\eta_2}{\|\eta_2\|},V(\mu))$ so by part 5 of Proposition \ref{prop:Bochner} we are done.
\end{proof}

\begin{cor}

\label{cor:not in M}
If $\frac{\d\mu}{\d V(\mu)}$ does not exists, then $\dm$ does not exists for any $|\mu|\in\mm_+(\Omega)$, so $\mu\notin\mm(\Omega,S)$.
\end{cor}

As opposed to the scalar case ($S=\r$),  $\mu<<|\mu|$ does not  guarantee the existence of $\dm$, even if $\mu$ is of bounded variation.
\begin{defn}
\label{Defn:RNProp}
A Banach space $S$ said to have the "Radon-Nikodym property"\index{Radom-Nikodym property} if $\dm$ exists and is Bochner integrable whenever the $S$-valued measure $\mu$ is of bounded variation and $\mu<<|\mu|$.
\end{defn}
 The classical Radon-Nikodym theorem states that $\r$ has the Radon-Nikodym property. By Theorem 1 in \cite{Diestel1977} Chapter III.3, any separable dual space has the Radon-Nikodym property. The following example demonstrates a case where $\mu<<|\mu|$ but $\dm$ does not exist.

\bigskip

\fbox{
\parbox{16cm}{
\begin{example}
Let $\Omega=[0,1]$ equipped with the Borel $\sigma$-algebra, $\ S=L^p[0,1]$. $\mu$ defined by $\mu(A):=\chi_A$ is absolutely continuous wrt the Lebesgue measure $\lambda\in\mathcal{P}(\Omega)$. But it is possible to show that there is no Bochner integrable $\eta:\Omega\rightarrow S$ such that  $\mu=(\eta,\lambda)$ thus by Corollary \ref{cor:not in M} $\mu\notin\mm(\Omega,S)$. For $p=1$, $\mu$ is of bounded variation $V(\mu)=\lambda$ but $S$ is not a dual space. For $p>1$, $S$ is a separable dual space but $\mu$ is not of bounded variation (For $p=\infty$, $\mu$ is not countably-additive so it is not a measure).
\end{example}
}
}

\bigskip

It is also possible to integrate scalar and $S^*$-valued functions wrt measures in $\mm(\Omega,S)$:

\begin{defn}
\label{def:integral S^*}
Let $f:\Omega\rightarrow S^*$, $g:\Omega\rightarrow\r$ and $\mu=(\eta,|\mu|)\in\mm(\Omega,S)$ then 
$\int_{\Omega} g\d\mu=\int_{\Omega}g\eta\d|\mu|$ 
and
$$
\int_{\Omega}f(\omega)\d\mu(\omega)
=
\int_{\Omega}\left\langle f(\omega),\eta(\omega)\right\rangle\d|\mu|(\omega)
$$
$f$ or $g$ are said to be $\mu$-integrable if their integral converges.
\end{defn}

When $S=\r$ and $\Omega$ is a compact metric space, the space of Borel measures on $\Omega$ can be identified with the space of linear functionals on $C(\Omega)$ (see for example Theorem 14.12 in \cite{Aliprantis2006Hitch}). Therefore these measures are uniquely determined by their values on $C(\Omega)$. In optimal transportation, we use this property by assuming the constraint
$$\int_{X\x Y}\psi(x)+\varphi(x)\d\pi(x,y)=\int_X\psi\d\mu+\int_Y\varphi\d\nu$$
characterizes the set of measures $\pi$ with marginals $\mu$ and $\nu$. To use the same characterization in the vector case we will show the following:
\begin{claim}
\label{claim:cont is enough}
Let $\Omega$ be a compact metric space equipped with its Borel $\sigma$-algebra and $\mu_i=(\eta_i,|\mu_i|)\in\mm(\Omega,S),\ i=1,2$ with $\eta_i\neq0$ $|\mu_i|$-almost everywhere, then the following are equivalent:
\begin{enumerate}
    \item $\mu_1=\mu_2$.
    \item $\int_{\Omega}f\d\mu_1=\int_{\Omega}f\d\mu_2$ for all $|\mu_1|$ and $|\mu_2|$ inegrable $f:\Omega\rightarrow S^*$.
    \item $\int_{\Omega}f\d\mu_1=\int_{\Omega}f\d\mu_2$ for all $f\in C(\Omega,S^*)$.
    \item $\int_{\Omega}f\d\mu_1=\int_{\Omega}f\d\mu_2$ for all $f\in C(\Omega,\f)$. Where $\f\subset S^*$ is some family of functionals separating the points of $S$.
    \item $\int_{\Omega}f\d\mu_1=\int_{\Omega}f\d\mu_2$ for all $f\in C(\Omega)$.
\end{enumerate}
\end{claim}

\begin{proof}
$\ $
\begin{itemize}
    \item $1\Rightarrow2$: By Corollary \ref{cor:NorDen}
    $$
    \int_{\Omega}f\d\mu_1
    =
    \int_{\Omega}\left\langle f,\eta_1\right\rangle\d|\mu_1|
    =
    \int_{\Omega}\left\langle f,\frac{\eta_1}{\|\eta_1\|}\right\rangle\d V(\mu_1)
    =
    \int_{\Omega}\left\langle f,\frac{\eta_2}{\|\eta_2\|}\right\rangle\d V(\mu_2)
    =
    \int_{\Omega}\left\langle f,\eta_2\right\rangle\d|\mu_2|
    =
    \int_{\Omega}f\d\mu_2$$
    \item $2\Rightarrow3$: Immediate since any continuous $f$ can be uniformly approximated by simple functions thus it is Bochner integrable.
     \item $3\Rightarrow4$: Immediate since $C(\Omega,\f)\subset C(\Omega,S^*)$.
    \item $4\Rightarrow5$: Let $s^*\in\f$ and $f\in C(\Omega)$, then $s^*\cdot f\in C(\Omega,\f)$. By $4$ and part 6 of Proposition \ref{prop:Bochner}
    $$\left\langle s^*,\int_{\Omega}f\d\mu_1\right\rangle=\int_{\Omega}s^*f\d\mu_1=\int_{\Omega}s^*f\d\mu_2=\left\langle s^*,\int_{\Omega}f\d\mu_2\right\rangle$$
    Since $s^*$ is arbitrary and $\f$ separates the points of $S$ we are done.
    \item $5\Rightarrow1$: Let $A\in\mathcal{A}$ an open subset. There exists a uniformly bounded sequence of continuous functions $f_n\in C(\Omega)$ converging pointwise to $\chi_A$. Since $\|f_n\eta_i\|\leq \sup_n\|f_n\|_{\infty}\|\eta_i\|\in L^1(|\mu_i|)$, by the vector dominated convergence theorem (Theorem 11.46 in \cite{Aliprantis2006Hitch} or Theorem 3 in \cite{Diestel1977} Chapter II.3)  we conclude
    $$
    \int_{\Omega}f_n\d\mu_i
    =
    \int_\Omega f_n\eta_i\d|\m_i|
    \underset{n\rightarrow\infty}{\longrightarrow}
    \int_A\eta_i\d|\m_i|
    =
    \mu_i(A)
    $$
    Since $ \int_\Omega f_n\d\m_1= \int_\Omega f_n\d\m_2\ \forall n$ its limit is $\mu_1(A)=\mu_2(A)$.
\end{itemize}
\end{proof}

\begin{rem}
From now on, we will always assume $\f$ is some given family (or in particular a subspace) of functionals in $S^*$ separating the points of $S$.
\end{rem}

In order to use $S$-valued measures in transport problems, we  define the marginals of these measures as in the scalar case:
\begin{defn}  
\label{def:marginals} For compact metric spaces $X,Y$, we say $(\eta,\pi)\in\mm\left(X\times Y,S\right)$ has \index{marginal}\emph{marginal}
$\mu\in\mm\left(X,S\right)$ on $X$ if any of the following hold:
\begin{itemize}
    \item $(\eta,\pi)(A\x Y)=\mu(A)$ for every $A\in\mathcal{A}$.
    \item $
            \int_{X\times Y} g\left(x\right)\eta(x,y)\d\pi\left(x,y\right)
                 =
            \int_{X}g\left(x\right)\d\mu\left(x\right)
            \ \forall g\in C(X)
          $
    \item $
            \int_{X\times Y} \left\langle f\left(x\right),\eta(x,y)\right\rangle\d\pi\left(x,y\right)
                 =
            \int_{X}f\left(x\right)\d\mu\left(x\right)
            \ \forall f\in C(X,\f)
          $
          where $\f$ is some family of functionals separating the points of $S$.
\end{itemize}
        
The marginal on $Y$ is defined similarly.
\end{defn}

\subsection{Transport problems in the infinite-dimensional case}

We can now formulate the optimal transport problem for $S$-valued
measures:
\begin{itemize}
\item $X,Y$ are two compact metric spaces.
\item $S$ is a Banach space. $\mathcal{F}\subset S^*$ is a subspace of functionals separating the points of $S$ (For example $\mathcal{F}=S^*$ or $\mathcal{F}^*=S$). We denote by $w(\f)$\index{w@$w(\f)$} the \index{weak topology} weak topology on $S$ generated by $\f$.
\item $\mu\in\mm\left(X,S\right),\nu\in\mm\left(Y,S\right)$ are
two $S$-valued measures.
\item $c\in C(X\x Y)$.
\item $\eta\in C(X\times Y,(S,w(\f)))$ and norm bounded.
\item $\Pi\left(\mu,\nu,\eta\right)$ \index{pimn@$\Pi\left(\mu,\nu,\eta\right)$}
is the set of scalar measures $\pi\in\mm_+\left(X\times Y\right)$
such that $(\eta,\pi)$ has marginals $\mu,\nu$: for every $\psi\in C\left(X,\mathcal{F}\right),\varphi\in C\left(Y,\mathcal{F}\right)$
\[
\int_{X}\psi\left(x\right)\d\mu\left(x\right)+\int_{Y}\varphi\left(y\right)\d\nu\left(y\right)=\int_{X\times Y}\left\langle \psi\left(x\right)+\varphi\left(y\right),\eta\left(x,y\right)\right\rangle \d\pi\left(x,y\right)
\]
\item $\Phi\left(c,\eta\right):=\left\{ \left(\psi,\varphi\right)\in C\left(X,\mathcal{F}\right)\times C\left(Y,\mathcal{F}\right):\left\langle \psi\left(x\right)+\varphi\left(y\right),\eta\left(x,y\right)\right\rangle \leq c\left(x,y\right)\forall x,y\right\} $\index{f1ch@$\Phi\left(c,\eta\right)$}.
\end{itemize}
\begin{assumption}
\label{ass:Vec}
There exists $f\in C(X,\mathcal{F})$ such that $\left\langle f,\eta\right\rangle\equiv1$.
\end{assumption}

\begin{rem}
The above assumption may be replaced with the existence of some $\tilde{f}\in C(X,\mathcal{F})$ such that $\left\langle \tilde{f},\eta\right\rangle^{-1}\in C(X)$, to see that just take $f=\frac{\tilde{f}}{\left\langle \tilde{f},\eta\right\rangle}$.
\end{rem}

\begin{rem}
For $\eta=\dm$, if $\left\langle f,\eta\right\rangle\equiv1$ then $\int_Af\d\mu=\int_A \left\langle f,\eta\right\rangle\d|\mu|=|\mu|(A)$ for any $A\subset X$. On other hand, if $\int_Af\d\mu=|\mu|(A)$ for all $A\subset X$ then
$|\mu|(A)=\int_Af\d\mu=\int_A \left\langle f,\eta\right\rangle\d|\mu|$
for all $A\subset X$ thus $\left\langle f,\eta\right\rangle\equiv1$ $|\mu|$-almost everywhere.
\end{rem}

\begin{rem}
\label{rem:Exb1}
The above assumption hold true in the finite-dimensional case $S=\rn$ when $\mu=(\m_1,\ldots,\m_n)\in\mm_+(X)^n$ and $\left|\mu\right|=\sum_{i=1}^{n}\mu_{i}$:
Just choose the constant function $f=\left(1,\ldots,1\right)\in\rn$
and then
\[
\left\langle \frac{\d\mu}{\d\left|\mu\right|},f\right\rangle =\sum_{i=1}^{n}\frac{\d\mu_{i}}{\d\left|\mu\right|}=\frac{\d\left|\mu\right|}{\d\left|\mu\right|}\equiv1
\]
\end{rem}

\begin{claim}
\label{claim:pi has marginal |mu|}
Under Assumption \ref{ass:Vec}, any $\pi\in\Pi\left(\mu,\nu,\dm\right)$ has marginal $|\mu|$ on $X$
\end{claim}

\begin{proof}
Let $\pi\in\Pi\left(\mu,\nu,\dm\right)$. By definition, $\left(\dm,\pi\right)$
has marginal $\mu$ on $X$ so for any $\psi\in C\left(X\right)$ 
\begin{multline*}
\int_{X\x Y}\psi\left(x\right)\d\pi\left(x,y\right)
 = 
\int_{X\x Y}\psi\left(x\right)\left\langle f\left(x\right),\dm\left(x\right)\right\rangle \d\pi\left(x,y\right)
= \\
 \int_{X}\psi\left(x\right)f\left(x\right)\d\mu\left(x\right)
  = 
 \int_{X}\left\langle \psi\left(x\right)f\left(x\right),\dm\left(x\right)\right\rangle \d\left|\mu\right|\left(x\right)
  = 
 \int_{X}\psi\left(x\right)\d\left|\mu\right|\left(x\right)
\end{multline*}
\end{proof}

The optimal transport problem for $S$-valued measures and its dual
problem are described in the next duality theorem:

\bigskip

\colorbox{thmcolor}{
\parbox{\linewidth}{
\begin{thm}
\label{thm:GMK}(Generalized Kantorovich Duality\index{generalized Kantorovich duality}) Under Assumption \ref{ass:Vec}:

\[
\min_{\pi\in\Pi\left(\mu,\nu,\eta\right)}\int_{X\times Y}c\left(x,y\right)\d\pi\left(x,y\right)
=
\sup_{\left(\psi,\varphi\right)\in\Phi\left(c,\eta\right)}\int_{X}\psi\left(x\right)\d\mu\left(x\right)+\int_{Y}\varphi\left(y\right)\d\nu\left(y\right)
\]
\end{thm}
}}

\begin{proof}
Substitute in Theorem \ref{Thm:AbsDuality}:

\begin{itemize}
\item $V=C(X\x Y)$, $V^*=\mm(X\x Y)$
\item $K=C\left(X,\mathcal{F}\right)\times C\left(Y,\mathcal{F}\right)$
\item $u=c\left(x,y\right)$
\item $\q=$ the non-negative functions in $V$.
\item $k^{*}\left(\psi,\varphi\right)=\int_{X}\psi\left(x\right)\d\mu\left(x\right)+\int_{Y}\varphi\left(y\right)\d\nu\left(y\right)$
\item $h\left(\psi,\varphi\right)(x,y)=\left\langle \psi\left(x\right)+\varphi\left(y\right),\eta\left(x,y\right)\right\rangle$. $h\left(\psi,\varphi\right)$ is continuous since $\eta$ is norm bounded and $w(\f)$-continuous.
\end{itemize}
To verify Assumption \ref{assu:Sublinear}, just choose $s(g):=(f\cdot\|g\|_{\infty},0)$ (where $f$ is the function from Assumption \ref{ass:Vec}).
\end{proof}
By Lemma \ref{lem:EqCon}, we get equivalent conditions to the existence
of a transport plan:
\begin{cor}
\label{cor:of EqCon}Under Assumption \ref{ass:Vec} the following are equivalent
\begin{enumerate}
\item $\Pi\left(\mu,\nu,\eta\right)\neq\emptyset$ 
\item $\inf_{\left(\psi,\varphi\right)\in\Phi\left(c,\eta\right)}\int_{X}\psi\left(x\right)\d\mu\left(x\right)+\int_{Y}\varphi\left(y\right)\d\nu\left(y\right)>-\infty$
for some $c\in C\left(X\x Y\right)$.
\item $\int_{X}\psi\left(x\right)\d\mu\left(x\right)+\int_{Y}\varphi\left(y\right)\d\nu\left(y\right)\geq0$
for every $\left(\psi,\varphi\right)\in C\left(X,\mathcal{F}\right)\times C\left(Y,\mathcal{F}\right)$
such that $\left\langle \eta\left(x,y\right),\psi\left(x\right)+\varphi\left(y\right)\right\rangle \geq0\ \forall x,y$
\end{enumerate}

\end{cor}

Before discussing the properties of this problem, we give an application of the vector-valued problem in the semi-discrete case.

\fbox{
\parbox{16cm}{
\begin{example}
\label{Ex:InfMarket} In this example we model
a market of two power plants among a set of consumers:
\begin{itemize}
\item $|\mu|\in\mathcal{P}(X)$ is the distribution of consumers over a set $X$. $Y=\left\{ y_{1},y_{2}\right\}$
is a set of two power plants. 
\item Each plant $y_{i}$ can supply an amount of power $\nu_{i}\left(t\right)\in\r$
at time $t\in\left[0,1\right]$. $\nu$ is a $S$-valued
measure on $Y$, where $S$ is some space of functions with variable
$t$, assume for simplicity $S=S^{*}=\mathcal{F}=L^{2}\left[0,1\right]$.
\item Each consumer at $x\in X$ demands $\eta\left(x,t\right)$ amount of
power at time $t$ (we consider $\eta:X\times\left[0,1\right]\rightarrow\r$
instead of $\eta:X\rightarrow S$), $\mu=(\eta,|\mu|)$.
\item $u=\left(c_{1},c_{2}\right)$ represents the cost of supplying power
to $x\in X$ by the two plants.
\item A transport plan is a measure in $\mm_+(X\x Y)=\mm_+(X)^2$. By Claim \ref{claim:pi has marginal |mu|} such transport plan $(\mu_1,\mu_2)$ has marginal $|\mu|$ on $X$ thus $\mu_{1}+\mu_{2}=|\mu|$. Moreover the transport plan should allow each consumer to purchase from both plants in a way which is consistent
with the supply constraint: For every $i=1,2$ and almost every $t\in\left[0,1\right]$:
\[
\begin{array}{c}
\int_{X}\eta\left(x,t\right)\d\mu_{i}\left(x\right)=\nu_{i}\left(t\right)\end{array}
\]
\item We look for the minimal total cost 
\[
\min_{(\mu_{1},\mu_{2})}\int_{X}c_{1}\left(x\right)\d\mu_{1}\left(x\right)+\int_{X}c_{2}\left(x\right)\d\mu_{2}\left(x\right)
\]
\item The dual problem is
\[
\max\int_{X}\int_{0}^{1}\psi\left(x,t\right)\eta\left(x,t\right)\d t\d|\mu|+\int_{0}^{1}\varphi_{1}\left(t\right)\nu_{1}\left(t\right)\d t+\int_{0}^{1}\varphi_{2}\left(t\right)\nu_{2}\left(t\right)\d t
\]
where the maximum is taken over all triplets $\left(\psi,\varphi_{1},\varphi_{2}\right)\in C\left(X,S\right)\x S\x S$
satisfying
$$\int_{0}^{1}\left[\psi\left(x,t\right)+\varphi_{i}\left(t\right)\right]\eta\left(x,t\right)\d t\leq c_{i}\left(x\right)\ \forall x,i$$
\end{itemize}
\end{example}
}}
\newpage
\subsection{Kernels of measures}

A transport map sends each $x\in X$ to a unique $y\in Y$, whereas a transport
plan \textquotedbl spreads\textquotedbl{} any $x\in X$ over $Y$. In the scalar case, we identify any transport map $T:X\rightarrow Y$ with the transport plan $(\text{Id}_X\x T)_\#$ which has an integration element $\d\pi(x,y)=\d\delta_{Tx}(y)\d\mu(x)$. Any transport plan can be represented  similarly as $\d\pi(x,y)=\d P_x(y)\d\mu(x)$ where $\delta_{Tx}$ is replaced with some family of probability measures, corresponding to the non-deterministic distribution on $Y$ to which $x$ is sent. We will establish a similar representation for transport plans of the vector-valued problem. We recall that $X,Y$ are two compact metric spaces.

\begin{defn}
$\ $
\begin{enumerate}
\item \index{kernel}A\emph{ kernel} (also called a Markov kernel)
is a family of Borel probability measures $P=\left\{ P_{x}\right\} _{x\in X}\subset\mathcal{P}\left(Y\right)$
such that $x\mapsto P_{x}\left(B\right)$ is a real Borel measurable function
for every Borel measurable subset $B\subset Y$. We denote the set
of such kernels by\index{ker@$\text{ker}(X,Y)$} $\text{ker}\left(X,Y\right)$.
\item For any $\mu\in\mm\left(X,S\right)$ and $P\in\text{ker}\left(X,Y\right)$ we define $P\otimes|\mu|\in\mm(X\x Y)$\index{1@$\otimes$} by 
\[
\int_{X\x Y}f\left(x,y\right)\d P\otimes|\mu|\left(x,y\right)
:=
\int_{X}\left(\int_{Y} f\left(x,y\right)\d P_{x}\left(y\right)\right)\d\left|\mu\right|\left(x\right)
\ \forall f\in C\left(X\x Y\right)
\]
 and $P\otimes\mu:=(\dm,P\otimes|\m|)\in\mm\left(X\x Y,S\right)$
   (which is indeed a measure by part 4 of Proposition \ref{prop:Bochner}), explicitly: $P\otimes \mu(A)=\int_A\dm(x)\d P\otimes|\mu|(x,y)$.
We denote by $P\mu\in\mm(Y,S)$ and $P|\m|\in\mm_+(Y)$ the marginals on $Y$ of $P\otimes\mu$ and $P\otimes|\mu$| respectively.
\end{enumerate}
\end{defn}

\begin{rem}
By Claim \ref{claim:cont is enough}, we may define $P\otimes\mu$ using its action on vector or real functions:
\begin{itemize}
    \item For any $f:X\x Y\rightarrow S^*$
    \begin{equation}
    \label{eq:action on vector function}
    \int_{X\x Y}f\d P\otimes \mu
    =
    \int_X\int_Y\left\langle f(x,y),\dm(x)\right\rangle\d P_x(y)\d|\m|(x)
    =
    \int_X\int_Yf(x,y)\d P_x(y)\d\m(x)
    \end{equation}
    Where the last equality is implied by part 6 of Proposition \ref{prop:Bochner}.
    \item For any $g:X\x Y\rightarrow\r$
    $$
    \int_{X\x Y}g\d P\otimes\mu
    =
    \int_X\int_Yg(x,y)\dm(x)\d P_x(y)\d|\mu|(x)
    =
    \int_X\int_Yg(x,y)\d P_x(y)\d\mu(x)
    $$
    
\end{itemize}
\end{rem}

\fbox{
\parbox{15cm}{
\begin{example}
Kernels can be used to describe random processes, for example the
simple random walk on $\mathbb{Z}$ with probabilities $l$ to go
left, $m$ to stay put and $r$ to go right ($l+m+r=1$) is described
by the kernel $P\in\text{ker}\left(\mathbb{Z},\mathbb{Z}\right)$
defined by
\[
P_{x}=l\cdot\delta_{x-1}+m\cdot\delta_{x}+r\cdot\delta_{x+1}
\]
\end{example}
}
}

\bigskip

For two measures $\m\in\mm(X,S),\n\in\mm(Y,S)$ there might be no kernel $P\in\text{ker}(X,Y)$ such that $P\m=\n$. The existence of such a kernel defines a preorder on the family of measure spaces:
\begin{defn}
\index{1@$\succ$}
\label{def:relation}
For any pair of $S$-valued measures  $(\m,\n)\in\mm(X,S)\x\mm(Y,S)$ we define the relation $\mu\succ\nu\Leftrightarrow\exists P\in\text{ker}(X,Y)$ such that $P\mu=\nu$, explicitly:
$$
\nu(B)
=
\int_X P_x(B)\d\m(x)\ \forall B\subset Y
$$
\end{defn}

\begin{rem}
\label{rem:RelationScalar}
For non-negative measures, $\mu\succ\nu$ is equivalent to $\mu(X)=\nu(Y)$, since we may choose $P_x=\frac{1}{\nu(Y)}\nu\ \forall x$.
\end{rem}

\begin{prop}
\label{prop:Balyage Convex}
Let $\mu\in\mm(X,S)$, then $N_{\mu}:=\{\nu\in\mm(Y,S);\mu\succ\nu\}$ is a convex subset of $\mm(Y,S)$.
\end{prop}

\begin{proof}
    \item Let $\nu_0,\nu_1\in N_{\mu}$ and $t\in[0,1]$. There exists $P^0,P^1\in\text{ker}(X,Y)$ such that $P^0\mu=\nu_0,\ P^1\mu=\nu_1$. For $P:=tP^1+(1-t)P^0\in\text{ker}(X,Y)$ we get $P\mu=t\nu_1+(1-t)\nu_0$ thus $\mu\succ t\nu_1+(1-t)\nu_0$ and $t\nu_1+(1-t)\nu_0\in N_{\mu}$. Therefore $N_{\mu}$ is convex.
\end{proof}

\fbox{
\parbox{15cm}{
\begin{example}
\label{ex:discretedominance}
(Discrete case) Let $X=Y=\{0,1\},\ S=\r$, then $\mm(X,S)=\r^2$. Take some $\mu=(\mu_0,\mu_1)\in\mm(X,S),\ \nu=(\nu_0,\nu_1)\in\mm(Y,S)$. In this case $\mu\succ\nu$ is equivalent to the existence of $p,q\in\left[0,1\right]$ such that
$$
\begin{array}{ccc}
     \nu_0&=&p\mu_0+q\mu_1
     \\
     \nu_1&=&(1-p)\mu_0+(1-q)\mu_1
\end{array}
$$
which implies $\mu_0+\mu_1=\nu_0+\nu_1$. The kernel is $P=(P_0,P_1)=((p,1-p),(q,1-q))$ and $P\mu$ can be considered as a multiplication  of a matrix with a vector. When $S=\r^2$ and $\mu=((\mu_1,\mu_2),(\mu_3,\mu_4)),\ \nu=((\nu_1,\nu_2),(\nu_3,\nu_4))$ then $\mu\succ\nu$ is equivalent to the existence of $p,q\in\left[0,1\right]$ such that 
$$
\begin{pmatrix}p&q\\1-p&1-q\end{pmatrix}\begin{pmatrix}\mu_1&\mu_3\\\mu_2&\mu_4\end{pmatrix}=\begin{pmatrix}\nu_1&\nu_3\\\nu_2&\nu_4\end{pmatrix}
$$
The left matrix is called a (left) stochastic/Markov matrix. For this reason this relation is sometimes referred as "stochastic dominance".

For more concrete examples we may conclude that
$$
\begin{array}{cc}
     ((1,0),(0,1))\succ((a,1-a),(b,1-b))\iff a,b\in[0,1]
     \\
     \left(\left(1,0\right),\left(\frac{1}{2},\frac{1}{2}\right)\right)\succ((a,1-a),(b,1-b))\iff a\in[0,1],\ b\in\left[\frac{a}{2},\frac{a+1}{2}\right]
     \\
     \left(\left(a,1-a\right),\left(b,1-b\right)\right)\succ((c,1-c),(c,1-c))\iff a,b\in\r,\ c\in\left[0,1\right]\ (take\ p=q=c)
\end{array}
$$
\end{example}
}
}

\bigskip

\fbox{
\parbox{15cm}{
\begin{example}
\label{ex:semidiscretedominance}
(Semi-discrete case) Let $X=[0,1],\ Y=\{0,1\},\ S=\r^2$ and $\mu=(\frac{\d\mu}{\d\lambda},\lambda)$ where $\lambda$ is the Lebesgue measure on $[0,1]$ and $\frac{\d\mu}{\d\lambda}=(1,2x)$. In this case $\mu\succ((a,1-a),(b,1-b)),\ a,b\in[0,1]$ is equivalent to the existence of some measurable $P_x(0)=1-P_x(1):=f(x)$ with values in $[0,1]$ such that
\begin{equation}
\label{eq:semi dis example}
\int_0^1f(x)=a,\ \int_0^12xf(x)=b 
\end{equation}
An immediate necessary condition is $2a>b$, but a stronger equivalent condition may be found: Assume $\int_0^1f(x)=a$, the minimal value of $\int_0^12xf(x)$ is $a^2$ (for $f(x)=\chi_{[0,a]}$) and the maximal value is $2a-a^2$ (for $f(x)=\chi_{[1-a,1]}$) therefore
$$\mu\succ((a,1-a),(b,1-b))\iff a\in[0,1],\ b\in[a^2,2a-a^2]$$
In particular for $a=b$: $\mu\succ((a,1-a),(a,1-a))\ \forall a\in[0,1]$.

Moreover, $P$ is unique (up to almost everywhere equivalence) when $b=a^2$ or $b=2a-a^2$, since in these cases $f(x)=\chi_{[0,a]}$ or $f(x)=\chi_{[1-a,a]}$ respectively is the unique (up to almost everywhere equivalence) solution to (\ref{eq:semi dis example}) (To minimize/maximize the integral $\int_0^1 xf(x)\d x$ the mass of $f$ must be concentrated to the left/right as much as possible).
\end{example}
}
}

\bigskip

As mentioned earlier for scalar measures, any measurable $T:X\rightarrow Y$ induces a deterministic kernel sending $\m$ to $T_\#\m$. This is also true for infinite-dimensional measures:

\begin{claim}
\label{claim:map induces kernel}
Let $\m\in\mm(X,S)$ and $T:X\rightarrow Y$ a measurable map. Then $P\in\text{ker}(X,Y)$ defined by $P_x=\delta_{Tx}$ satisfies $P\m=T_\#\m$. In particular $\m\succ T_\#\m$.
\end{claim}
\begin{proof}
For any measurable $B\subset Y$
$$
    P\m(B)
    =
    \int_XP_x(B)\d\m(x)
    =
     \int_{T^{-1}(B)}\d\m(x)
    = 
     \m(T^{-1}(B))
     =
     T_\#\m(B)
$$
Therefore $P\mu=T_{\#}\m$,
\end{proof}

\begin{rem}
\label{ex:pushforward}
An immediate corollary is
$$
\int_Yf\d T_\#\mu=\int_X\int_Yf(y)\d \delta_{Tx}(y)\d\mu(x)=\int_Xf(Tx)\d\mu(x)
\ \forall f\in C(Y,S^*)
$$
\end{rem}

\begin{claim}
\label{lem:relation}The relation $\succ$ is a preorder.
\end{claim}

\begin{proof}
$\ $
\begin{itemize}
\item Reflexivity - For any $\mu$, the identity map sends $\mu$ to itself to by Claim \ref{claim:map induces kernel} $\mu\succ\mu$.
\item Transitivity - Let $\mu\in\mm(X,S), \nu\in\mm(Y,S),\ \lambda\in\mm(Z,S)$ and assume $\mu\succ\nu\succ\lambda$.
By definition there exist two kernels $P\in\text{Ker}\left(X,Y\right),\ Q\in\text{ker}\left(Y,Z\right)$
such that $P\mu=\nu$ and $Q\nu=\lambda$. The kernel  $K\in\text{Ker}(X,Z)$ defined by $K_{x}:=\int_{Y}Q_{y}\left(\cdot\right)\d P_{x}\left(y\right)$ satisfies $K\mu=\lambda$
hence $\mu\succ\lambda$.
\end{itemize}
\end{proof}

\subsection{A generalization of Blackwell's theorem}

In a series of papers from the 1950's \cite{blackwell1951,Blackwell_Range,blackwell1953}, David Blackwell described a method to compare finite-dimensional vector measures, which he referred to as "experiments". By this method, one experiment is more "informative" than
another if the image of the first experiment contains the image of the second. His motivation was sampling
procedures in statistics and his tools were probabilistic
(conditional expectation etc.). The author also formulated equivalent conditions for one experiment to be more informative than another, one of them is the relation described in Definition \ref{def:relation}. I will show a connection between this relation and the optimal transport problem for vector measures (Claim \ref{claim:char of trans plans}), suggest a generalization
of this result for infinite-dimensional vector measures,
and give a measure-theoretic proof based on the disintegration theorem, different from the original
probabilistic proof. 

\b

The disintegration
theorem is a known tool in measure and probability theory. Intuitively, a disintegration of a measure
is its representation as a collection of measures, each lives on a
\textquotedbl small\textquotedbl{} (usually null) set. Formally:
\begin{defn}
\index{disintegration}Let $X,Y$ be two measurable spaces, $\mu\in\mm_+\left(X\right)$ and $T:X\rightarrow Y$ a measurable map.
A \emph{$T$-disintegration} of $\mu$ is a kernel $Q\in\text{Ker}\left(Y,X\right)$
such that
\begin{enumerate}
\item $QT_{\#}\mu=\mu$ 
\item $T_\#Q_y(\{y\})=Q_{y}\left(T^{-1}(\{y\})\right)=1$ for $T_{\#}\mu$-almost every $y\in Y$.
\end{enumerate}
\end{defn}

The disintegration theorem states the existence and uniqueness of
such disintegration:
\begin{thm}
\index{disintegration theorem}\label{thm:DisInt}(Theorem 1 in \cite{Chang1997})
Let $\mu\in\mm_+\left(X\right)$ and $T:X\rightarrow Y$ a measurable map.  There exists
a $T$-disintegration of $\mu$ which is unique almost surely: If
$Q,R\in\text{ker}\left(Y,X\right)$ are two such disintegrations then
$T_{\#}\mu\left(\left\{ y\in Y;\ Q_{y}\neq R_{y}\right\} \right)=0$.
\end{thm}

A common application of disintegration is for product spaces:
\begin{cor}
\label{cor:DisInt product measures}If $\pi\in\mm_+\left(X\times Y\right)$
has marginal $\mu$ on $X$, then there exists $Q\in\text{ker}(X,Y)$ unique $\m$-almost everywhere such that $\pi=Q\otimes\mu$.
\end{cor}

\begin{proof}
Apply the disintegration theorem to $\mu=T_{\#}\pi$, where $T:X\times Y\rightarrow X$ is the projection onto $X$ (defined by $T(x,y):=x$). There exists $\tilde{Q}\in\text{ker}\left(X,X\x Y\right)$ unique $\m$-almost everywhere,
such that $\pi=\tilde{Q}\mu$. Since $\tilde{Q}_x(T^{-1}(\{x\}))=\tilde{Q}_x(\{x\}\x Y)=1$ $\mu$-almost everywhere we may define $Q\in\text{ker}(X,Y)$ by $Q_x(B):=\tilde{Q}_x(\{x\}\x B)$ for every measurable $B\subset Y$ and then $\pi=\tilde{Q}\mu=Q\otimes\mu$
\end{proof}
\begin{example}
In the above corollary, when $\pi=\mu\times\nu$ is a product
measure then $Q_{x}=\nu\ \forall x$.
\end{example}

\begin{rem}
\label{rem:DisIntEq}
Corollary \ref{cor:DisInt product measures} is actually equivalent to Theorem \ref{thm:DisInt}: To see that we take $\pi=(\text{Id}_X\x T)_\#\m$, then $\pi$ has marginal $T_\#\m$ on $Y$ so by Corollary \ref{cor:DisInt product measures} there exists some kernel $Q\in\text{ker}(Y,X)$ such that $\pi=Q\otimes T_\#\m$. Since the marginal of $\pi$ on $X$ is $\mu$ we conclude $\mu=QT_\#\m$.
\end{rem}
\b

Using the disintegration theorem we can characterize transport plans in $\Pi\left(\m,\n,\dm\right)$:
\begin{claim}
\label{claim:char of trans plans}Let $\mu\in\mm\left(X,S\right)$
, $\nu\in\mm\left(Y,S\right)$ and $\pi\in\mm_+(X\x Y)$ with marginal $\pi_X$ on $X$\index{Pix@$\pi_X$}. Then $\pi\in\Pi\left(\mu,\nu,\dm\right)$ 
iff $\left(\dm,\pi_X\right)=\mu$ and there exists a kernel $P\in\text{ker}\left(X,Y\right)$ such that $\pi=P\otimes\pi_X$  and $P\mu=\nu$.
\end{claim}

\begin{proof}
\begin{itemize}
    \item Assume $\left(\dm,\pi_X\right)=\mu$ and $\exists P\in\text{Ker}\left(X,Y\right)$ such that $\pi=P\otimes\pi_X$
and $P\mu=\nu$. For every $\left(\psi,\varphi\right)\in C\left(X,\mathcal{F}\right)\times C\left(Y,\mathcal{F}\right)$ 

\[
\begin{array}{cc}
\int_{X\times Y}\left\langle \psi\left(x\right)+\varphi\left(y\right),\dm\left(x\right)\right\rangle \d\pi\left(x,y\right)
& =\\
\\
\int_{X}\left\langle \psi\left(x\right),\dm\left(x\right)\right\rangle\d \pi_X(x)+\int_{X}\int_{Y}\left\langle \varphi\left(y\right),\dm\left(x\right)\right\rangle \d P_{x}\left(y\right)\d\pi_X\left(x\right)
& =\\
\\
\int_{X}\psi\left(x\right)\d\mu\left(x\right)+\int_{X}\left\langle\int_{Y} \varphi\left(y\right)\d P_{x}\left(y\right),\dm\left(x\right)\right\rangle \d\pi_X\left(x\right)
& =\\
\\
\int_{X}\psi\left(x\right)\d\mu\left(x\right)+\int_{X}\int_{Y} \varphi\left(y\right)\d P_{x}\left(y\right)\d\mu\left(x\right)
& =\\
\\
\int_{X}\psi\d\mu+\int_{Y} \varphi\d\nu
\end{array}
\]

therefore $\pi\in\Pi\left(\mu,\nu,\dm\right)$. 
\item Assume $\pi\in\Pi\left(\mu,\nu,\dm\right)$. Since  $\left(\dm,\pi_X\right)$ is the marginal of $\left(\dm,\pi\right)$ on $X$ it is equal to $\mu$ by definition of $\Pi\left(\mu,\nu,\dm\right)$. Let $P$ be the kernel so that $\pi=P\otimes\pi_X$ (exists and unique by Corollary \ref{cor:DisInt product measures}). For every $\varphi\in C\left(Y,\f\right)$ 
\begin{multline*}
\int_{Y}\varphi\left(y\right)\d P\mu\left(y\right)
 = 
\int_{X\x Y}\varphi\left(y\right)\d P\otimes\mu\left(x,y\right)
 = 
\int_{X}\int_{Y}\varphi\left(y\right)\d P_{x}\left(y\right) \d\mu\left(x\right) 
 =\\=
 \int_{X}\left\langle \int_{Y}\varphi\left(y\right)\d P_{x}\left(y\right),\dm\left(x\right)\right\rangle \d\pi_X\left(x\right)
 =
\int_{X\times Y}\left\langle \varphi\left(y\right),\dm\left(x\right)\right\rangle \d\pi\left(x,y\right)
 = 
\int_{Y}\varphi\left(y\right)\d\nu\left(y\right)
\end{multline*}
so by Claim \ref{claim:cont is enough} $P\mu=\nu$.
\end{itemize}
\end{proof}

\begin{cor}
\label{cor:succ iff Pi not empty}
$\Pi\left(\mu,\nu,\dm\right)\neq\emptyset$ iff $\mu\succ\nu$.
\end{cor}
\begin{proof}
If $\Pi\left(\mu,\nu,\dm\right)\neq\emptyset$ then by Claim \ref{claim:char of trans plans} there exists $P\in\text{ker}(X,Y)$ such that $P\mu=\nu$ so by definition $\mu\succ\nu$. If $\mu\succ\nu$ there exists $P\in\text{ker}(X,Y)$ such that $P\mu=\nu$, and then $P\otimes|\mu|\in\Pi\left(\mu,\nu,\dm\right)$.
\end{proof}

The kernel $P_{x}=\delta_{Tx}$ sends $\mu$ into
its pushforward $P\mu=T_{\#}\mu$. Such a kernel
can be considered as a \textquotedbl deterministic kernel\textquotedbl{}
which sends any element of $X$ to a Dirac measure or an element
in $Y$. A $T$-disintegration is a kernel $Q$ that sends the pushforward measure $T_{\#}\mu$ back to $\mu$.
So one can refer the disintegration as an \textquotedbl inverse
kernel\textquotedbl{} of the deterministic kernel induced by $T$.
A reasonable question to ask is whether such \textquotedbl inverse
kernel\textquotedbl{} exists for non-deterministic
kernels. In the following corollary we answer this question:

\begin{cor}
\label{cor:GDT} Let $\mu\in\mm_+\left(X\right)$, $P\in\text{ker}\left(X,Y\right)$ and $\nu:=P\mu$. Then there exists a kernel
$Q\in\text{ker}\left(Y,X\right)$ such that $Q\otimes\nu=P\otimes\mu$ (and in particular $Q\nu=\mu$).
\end{cor}

\begin{proof}
Denote $\pi:=P\otimes\mu$.
By definition of $\nu$, it is the marginal of $\pi$ on $Y$, so
by Corollary \ref{cor:DisInt product measures} there exists $Q\in\text{ker}\left(Y,X\right)$
such that $Q\otimes\nu=\pi$.
\end{proof}

Under some normalization condition, we may interchange $P$ and $|\cdot|$:
\begin{lem}
\label{lem:kernel maps densities}
Let $\mu\in\mm\left(X,S\right),\nu\in\mm\left(Y,S\right)$ and $P\in\text{ker}(X,Y)$ such that $P\mu=\nu$ and assume
\begin{equation}
\label{eq:CondDens}
  \exists s^*\in S^*,\ 
  \left\langle s^*,\frac{\d\mu}{\d\left|\mu\right|}\right\rangle \equiv\left\langle s^*,\frac{\d\nu}{\d\left|\nu\right|}\right\rangle \equiv1  
\end{equation}
Then $P|\mu|=|\nu|$.
\end{lem}
\begin{proof}
For the $s^*$ satisfying (\ref{eq:CondDens}) and any measurable $B\subset Y$
\begin{multline*}
|\nu|(B)
=
\int_B\left\langle s^*,\dn\left(y\right)\right\rangle \d|\nu|(y)
=
 \left\langle s^*,\nu(B)\right\rangle
 =
 \left\langle s^*,P\mu(B)\right\rangle
 =\\
 \int_X\int_B\left\langle s^*,\dm\left(x\right)\right\rangle\d P_{x}\left(y\right)\d\left|\mu\right|(x)
 =
 \int_X P_{x}\left(B\right)\d\left|\mu\right|(x)
 =
 P|\m|(B)
\end{multline*}
\end{proof}

To prove one of our main results, which is a generalization of a result by Blackwell, we will need the following property of reversed kernels:
\begin{lem}
\label{lem:rev ker prop}
Let $\mu\in\mm\left(X,S\right),\nu\in\mm\left(Y,S\right)$. If $P\otimes|\m|=Q\otimes|\n|$ and $P\m=\n$ then  $|\nu|$-almost everywhere
\[
\int_{X}\dm(x)\d Q_{y}\left(x\right)=\dn\left(y\right)
\] 
\end{lem}
\begin{proof}
For every Borel measurable $B\subset Y$:
\begin{multline*}
\int_B\dn\d\left|\nu\right|
=
\nu(B)
=
P\mu(B)
=
P\otimes\mu(X\x B)
=\\=
\int_{X\x B}\dm(x)\d P\otimes\left|\mu\right|
=
\int_{X\x B}\dm(x)\d Q\otimes\left|\nu\right|
=
\int_{B}\int_{X}\dm\left(x\right)\d Q_{y}\left(x\right)\d\left|\nu\right|\left(y\right)
\end{multline*}

hence
$$\int_{B}\left[\dn(y)-\int_{X}\dm\left(x\right)\d Q_{y}\left(x\right)\right]\d\left|\nu\right|\left(y\right)
\ \forall B\subset Y$$
and by part 5 of Proposition \ref{prop:Bochner} we are done.
\end{proof}

We will also need Jensen's inequality for the Bochner integral: 

\begin{prop}
\label{prop:Jensen's inequality}(Jensen's inequality)\index{Jensen's inequality}
Let
\begin{itemize}
    \item $\f\subset S^*$ a family of functionals separating the points of $S$.
    \item $g:S\rightarrow\r$ a convex and $w(\f)$-lower semi-continuous function.
    \item $\mu\in\mathcal{P}\left(X\right)$.
    \item $f:X\rightarrow S$ a Bochner integrable function wrt $\mu$.
\end{itemize}
Then
$
g\left(\int_{X}f\\d\mu\right)
\leq
\int_{X}g\circ f\d\mu
$.
\end{prop}

\begin{proof}
Denote $s_{0}:=\int_{X}f\d\mu$. The $w(\f)$ topology on $S$ is Haussdorf and locally convex, so by
Lemma 7.5 in \cite{Aliprantis2006Hitch}, for any $\varepsilon>0$ there exists $s^*\in \f$ such that 
\[
\
\left\langle s^*,s-s_{0}\right\rangle +g\left(s_{0}\right)  \leq  g\left(s\right)+\varepsilon\  \forall s\in S
\]
Take $s=f(x)$ and integrate wrt $\m$:
\[
\varepsilon+\int_{X}g\circ f\left(x\right)\d\mu
\geq
\int_{X}\left\langle s^*,f\left(x\right)-s_0\right\rangle +g\left(s_0\right)\d\mu
=
\int_{X}\left\langle s^*,f\left(x\right)\right\rangle \d\mu-\left\langle s^*,s_0\right\rangle+g\left(s_0\right)
=
g\left(\int_{X}f\left(x\right)\d\mu\right)
\]
and since $\varepsilon$ is arbitrary small we are done.
\end{proof}

We are now ready to present and prove the generalization of Blackwell's theorem:

\b

\colorbox{thmcolor}{
\parbox{\linewidth}{
\begin{thm}
\label{thm:Blackwell}Let $\mu\in\mm\left(X,S\right),\nu\in\mm\left(Y,S\right)$ and $\f\subset S^*$ a subspace of functionals that separates the points of $S$.
Denote by $G$ the set of all convex and $w(\f)$-lower semi-continuous
functions $g:S\rightarrow\mathbb{R}$. Assuming (\ref{eq:CondDens}), the following are equivalent:
\begin{enumerate}
\item There exists a kernel $Q\in\text{ker}\left(Y,X\right)$ such that $Q\left|\nu\right|=\left|\mu\right|$
and $|\nu|$-almost everywhere
\[
\int_{X}\dm(x)\d Q_{y}\left(x\right)=\dn\left(y\right)
\] 

\item $\int_{X}g\left(\dm\right)\d\left|\mu\right|\geq\int_{Y}g\left(\dn\right)\d\left|\nu\right|$
for every $g\in G$.
\item $\Pi\left(\m,\n,\dm\right)\neq\emptyset$.
\item There exists a kernel $P\in\text{ker}\left(X,Y\right)$ such that $P\mu=\nu$.
\end{enumerate}
While assuming \ref{ass:Vec} instead of the stronger assumption (\ref{eq:CondDens}), we only have $1\Rightarrow2\Rightarrow3\Leftrightarrow4$.
\end{thm}
}}

\begin{proof}
\begin{itemize}
\item $1\Rightarrow2$: By Jensen's inequality
(Proposition \ref{prop:Jensen's inequality}) for any $g\in G$
\[
\begin{array}{ccc}
\int_{X}g\left(\dm\right)\d\left|\mu\right|
& = &
\int_{Y}\int_{X}g\left(\dm\left(x\right)\right)\d Q_{y}\left(x\right)\d\left|\nu\right|(y)
\\\\ & \geq &
\int_{Y}g\left(\int_{X}\dm\left(x\right)\d Q_{y}\left(x\right)\right)\d\left|\nu\right|(y)
\\\\& = &
\int_{Y}g\left(\dn\right)\d\left|\nu\right|
\end{array}
\]
\item $2\Rightarrow3$: Let $\left(\psi,\varphi\right)\in C\left(X,\f\right)\x C\left(Y,\f\right)$
satisfying $\left\langle \dm\left(x\right),\psi\left(x\right)+\varphi\left(y\right)\right\rangle\geq0\ \forall x,y$
and define $g_{\varphi}\in G$ by $g_{\varphi}\left(z\right):=\sup_{y\in Y}\left\langle -\varphi\left(y\right),z\right\rangle$.
By $2$:
\begin{multline*}
\int_X\psi\d\mu 
 = 
\int_{X}\left\langle \psi\left(x\right),\dm\left(x\right)\right\rangle \d\left|\mu\right|\left(x\right) 
\geq
 \int_{X}g_{\varphi}\left(\dm\left(x\right)\right)\d\left|\mu\right|\left(x\right)
 \geq
 \int_{Y}g_{\varphi}\left(\dn\left(y\right)\right)\d\left|\nu\right|\left(y\right)   
 \\=
 \int_{Y}\sup_{z\in Y}\left\langle -\varphi\left(z\right),\dn\left(y\right)\right\rangle \d\left|\nu\right|\left(y\right)
 \geq
 -\int_{Y}\left\langle \varphi\left(y\right),\dn\left(y\right)\right\rangle \d\left|\nu\right|\left(y\right)
 =
-\int_Y\varphi\d\nu
\end{multline*}

thus $\int\psi\d\mu+\int\varphi\d\nu\geq0$. By Corollary \ref{cor:of EqCon} $\Pi\left(\mu,\nu,\dm\right)\neq\emptyset$.

\item $3\Leftrightarrow4$: By Corollary \ref{cor:succ iff Pi not empty}.
\item $4\Rightarrow1$: By Lemma \ref{lem:kernel maps densities} (assuming (\ref{eq:CondDens})) $P\left|\mu\right|=\left|\nu\right|$, and by Corollary \ref{cor:GDT}
there exists a reversed kernel $Q$ such that $P\otimes\left|\mu\right|=Q\otimes\left|\nu\right|$. By Lemma \ref{lem:rev ker prop} we are done.
\end{itemize}
\end{proof}

\subsection{The Martingale formulation}
Let $|\m|\in\mathcal{P}(X),\ |\n|\in\mathcal{P}(Y)$ and $\pi\in\Pi(|\m|,|\n|)$. A pair of functions $f,g:X\x Y\rightarrow\r$ is called a martingale\index{martingale} wrt $\pi$ if $\mathbb{E}_{\pi}[f|g]=g$, meaning
\begin{equation}
    \label{eq:mart cond}
    \int_{X\x Y}(\zeta\circ g)(f-g)\d\pi=0\ \forall \zeta\in C_c(\r)
\end{equation}

By the disintegration theorem (Theorem \ref{thm:DisInt}), we can find some $Q\in\text{ker}(Y,X)$ so that $\pi=Q\otimes|\n|$. If $X=Y=\r$ and $f(x)=x,\ g(y)=y$, equation (\ref{eq:mart cond}) translates into
$$
    \int_Y \zeta(y)\int_X(x-y)\d Q_y(x)\d|\n|(y)=0\ \forall \zeta\in C_c(\r)
$$
which is equivalent to 
$$
\int_Xx\d Q_y(x)=y\ \ |\n|-almost\ everywhere
$$
The problem of minimizing the total cost $\int_{X\x Y}c\d\pi$ among all transport plans satisfy the above constraint is known as the "martingale optimal transport problem", discussed for example in \cite{beiglbockMartingale,guo2017computational}. More generally, we may consider the constraint
$$
\int_Xf(x)\d Q_y(x)=g(y)\ \ |\n|-almost\ everywhere
$$
for some measurable $f:X\rightarrow S,\ g:Y\rightarrow S$.
We consider the set of transport plans in $\Pi(|\mu|,|\nu|)$ satisfying the above constraint:
\[
\Pi_{Mar}(|\m|,|\n|,f,g):= \left\{
\begin{array}{cc}
     \pi\in\mm_+\left(X\x Y\right);\ \forall (\Psi,\Phi,\zeta)\in C(X)\x C(Y)\x C(Y,\f)  \\\\\ 
      \int_{X\x Y}\Psi(x)+\Phi(y)+\left< \zeta(y),f(x)-g(y)\right>\d\pi=\int_{X}\Psi\d|\m|+\int_{X}\Phi\d|\n| 
\end{array}
\right\}
\]\index{pimar@$\Pi_{Mar}$}

The Martingale formulation of Theorem \ref{thm:GMK} is the following:
 
 \b
 
 \colorbox{thmcolor}{
\parbox{\linewidth}{
 \begin{thm}
 \label{thm:MartingaleOT}
 Let $c\in C(X\x Y)$ and $f\in C(X,(S,w(\f))),\ g\in C(Y,(S,w(\f)))$ both norm bounded then
 $$
 \min_{\pi\in\Pi_{Mar}(|\m|,|\n|,f,g)}\int_{X\x Y}c\d\pi=\sup_{\Psi,\Phi} \int_{X}\Psi\d|\m|+\int_{X}\Phi\d|\n|
 $$
 Where the supremum is taken over all $(\Psi,\Phi)\in C(X)\x C(Y)$ satisfying
\[
\Psi(x)+\Phi(y)+\left<\zeta(y),f(x)-g(y)\right>\leq c(x,y)\ \forall x,y\in X\x Y
\]
for some $\zeta\in C(Y,\f)$.
 \end{thm}
 }}
 \b

\begin{proof}
Take in Theorem \ref{Thm:AbsDuality}:
\begin{itemize}
    \item $V=C(X\x Y)$, $\q$ is the cone of non-negative functions in $V$.
    \item $K=C(X)\x C(Y)\x C(Y,\f)$
    \item $h(\psi,\varphi,\zeta)(x,y)=\Psi(x)+\Phi(y)+\left<\zeta(y),f(x)-g(y)\right>$, continuous since $f,g$ are $w(\f)$-continuous and norm bounded.
    \item $k^*(\psi,\varphi,\zeta)=\int_{X}\Psi\d|\m|+\int_{X}\Phi\d|\n|$
    \item Assumption \ref{assu:Sublinear} holds by taking $s(v)=\left(\|v\|_{\infty},0,0\right)$.
\end{itemize}

\end{proof}

Assuming (\ref{eq:CondDens}) and $f=\dm,g=\dn$, the primal and dual problems in the theorem above and in  the problem of vector measures (Theorem \ref{thm:GMK}) coincide. In the following claim we prove this for the two primal problems:

\begin{claim}  
\label{claim:Mar=GMK}
Let $\m\in\mm(X,S),\ \n\in\mm(Y,S)$. Assume (\ref{eq:CondDens}) then $$\Pi_{Mar}\left(|\mu|,|\nu|,\dm,\dn\right)=\Pi\left(\m,\n,\dm\right)$$
\end{claim}
\begin{proof}
$\ $
\begin{itemize}
    \item Let $\pi\in\Pi\left(\m,\n,\dm\right)$, by Claim \ref{claim:pi has marginal |mu|} $\pi$ has marginal $|\mu|$. By Claim \ref{claim:char of trans plans} $\pi=P\otimes|\m|$ for some $P\in\text{ker}(X,Y)$ such that $P\m=\n$, so by Lemma \ref{lem:kernel maps densities} it has marginals $|\m|$ and $|\n|$ and for all $\zeta\in C(Y,\f)$:
    
    $$
    \int_{X\x Y}\left< \zeta(y),\dm(x)\right>\d\pi(x,y)
     = 
    \int_{Y}\zeta\d\n
     = 
    \int_{Y}\left< \zeta(y),\dn\left(y\right)\right>\d|\n|(y)
     = 
    \int_{X\x Y}\left< \zeta(y),\dn\left(y\right)\right>\d\pi(x,y)
    $$
    
     therefore $\pi\in\Pi_{Mar}\left(|\mu|,|\nu|,\dm,\dn\right)$.
    \item Let $\pi\in\Pi_{Mar}\left(|\mu|,|\nu|,\dm,\dn\right)$. For any $(\psi,\varphi)\in C(X,\f)\x C(Y,\f)$ choose $\Psi=\left\langle\psi,\dm\right\rangle,\ \Phi=\left\langle\varphi,\dn\right\rangle$ and $\zeta=\varphi$ in the definition of $\Pi_{Mar}$ to conclude $\pi\in\Pi\left(\m,\n,\dm\right)$.
\end{itemize}
\end{proof}

By the above claim and Theorems \ref{thm:GMK}, \ref{thm:MartingaleOT} we may conclude the two dual problems coincide as well, although it is not too hard to verify this directly.

\subsection{Existence of a solution to the dual problem}
We now wish to prove the existence of solutions for the dual problem in the special case $S=S^*=\f=\rn$ and $\eta=\frac{\d\mu}{\d|\mu|}\in C(X)^n$. Under Assumption \ref{ass:Vec} we may conclude that
$
\left\{\left\langle\psi,\eta\right\rangle;\ \psi\in C(X,\mathcal{F})\right\}=C(X)
$
and therefore the dual problem can be written as
\begin{equation}
\label{eq:dualproblem}
\begin{array}{cc}
     \sup\left\{\int_X\psi\d\m+\int_Y\varphi\d\n;\ \left(\psi,\varphi\right)\in\Phi(c,\eta)\right\}
      =  \\\\
     \sup\left\{\int_X\left\langle\psi,\eta\right\rangle\d|\m|+\int_Y\varphi\d\n;\ \left(\psi,\varphi\right)\in\Phi(c,\eta)\right\}
       =  \\\\
      \sup\left\{\int_X\Psi\d|\m|+\int_Y\varphi\d\n;\ \Psi\in C(X),\ \varphi\in C(Y)^n,\ \Psi(x)+\left\langle\varphi(y),\eta(x)\right\rangle\leq c(x,y)\right\}
\end{array}
\end{equation}

A solution to the dual problem $\Psi,\varphi$ together with an optimal transport plan $\pi\in\Pi\left(\mu,\nu,\dm\right)$ satisfy:

$$
\int_{X\x Y}\Psi(x)+\sum_{i=1}^n\varphi_i(y)\eta_i(x)\d\pi(x,y)
=
\int_X\Psi\d|\m|+\sum_{i=1}^n\int_Y\varphi_i\d\n_i
=
\int_{X\x Y}c(x,y)\d\pi(x,y)
$$
Together with $\Psi+\sum_{i=1}^n\varphi_i\eta_i\leq c$ this is equivalent to
\[
\begin{array}{ccc}
\Psi\left(x\right)+\sum_{i=1}^n\varphi_i(y)\eta_i(x)\leq c\left(x,y\right) & ; & \forall (x,y)\in X\times Y\\
\\
\Psi\left(x\right)+\sum_{i=1}^n\varphi_i(y)\eta_i(x) =c\left(x,y\right) & ; & \forall\left(x,y\right)\in\text{supp}\pi
\end{array}
\]

This dual problem may have no optimal solution, even for non-continuous $\psi,\varphi$, as we demonstrate in the following example:

\begin{example}
\label{ex:nonexist}
As in Example \ref{ex:semidiscretedominance}, let
\begin{itemize}
\item $X=\left[0,1\right],\ Y=\left\{ 0,1\right\},\ S=\r^2$.
\item $\mu_{1}$ the Lebesgue measure on $X$.
\item $\mu_{2}$ the measure on $X$ with density $\frac{\text{d}\mu_{2}}{\text{d}\mu_{1}}=2x$
wrt $\mu_{1}$.
\item $\mu=\left(\mu_{1},\mu_{2}\right)\in\mathcal{P}\left(X\right)^{2}$,
$\eta\left(x\right)=\frac{\text{d}\mu}{\text{d}\mu_{1}}\left(x\right)=\left(1,2x\right)$.
\item $\nu=\left(\nu_{1},\nu_{2}\right)=\left(\left(\frac{1}{2},\frac{1}{2}\right),\left(\frac{1}{4},\frac{3}{4}\right)\right)\in\mathcal{P}\left(Y\right)^{2}$.
\item Assumption \ref{ass:Vec} holds for $f=(1,0)$ ($f\cdot\eta\equiv1$).
\end{itemize}

By Claims \ref{claim:pi has marginal |mu|} and \ref{claim:char of trans plans}, every transport plan in $\Pi\left(\mu,\nu,\eta\right)$
is of the form $\pi=P\otimes\mu_{1}$ where $P\in\text{ker}\left(X,Y\right)$
satisfies $P\mu=\nu$. By Example \ref{ex:semidiscretedominance} such $P$ exists and unique (since $\left(\frac{1}{2}\right)^2=\frac{1}{4}$) so the only transport plan is $\pi=P\otimes\mu_{1}$ where
$P$ is the deterministic kernel defined by $P_{x}=\left(\chi_{\left[0,\frac{1}{2}\right]}\left(x\right),\chi_{\left(\frac{1}{2},1\right]}\left(x\right)\right)$.
This transport plan is induced by the map 
\[
Tx=\begin{cases}
0 & ;x\in\left[0,\frac{1}{2}\right]\\
1 & ;x\in\left(\frac{1}{2},1\right]
\end{cases}
\]
therefore the transport plan is $\pi=\left(\text{Id}\times T\right)_{\#}\mu_{1}$. Since $\pi$ is the only transport plan, it will be the optimal one for any cost function.

We wish to find some $c\in C\left(X\times Y\right)$ for which there
is no $\Psi:X\rightarrow\r$ and $\varphi:Y\rightarrow\mathbb{R}^{2}$
such that 
\[
\begin{array}{ccc}
\Psi\left(x\right)+\left\langle \varphi\left(y\right),\eta\left(x\right)\right\rangle \leq c\left(x,y\right) & ; & \forall (x,y)\in X\times Y\\
\\
\Psi\left(x\right)+\left\langle \varphi\left(y\right),\eta\left(x\right)\right\rangle =c\left(x,y\right) & ; & \forall\left(x,y\right)\in\text{supp}\pi
\end{array}
\]

Assume by contradiction that some $\Psi:X\rightarrow\r$ and $\varphi=\left(\varphi_{1},\varphi_{2}\right)=\left(\varphi_1(0),\varphi_1(1),\varphi_2(0),\varphi_2(1)\right)$ satisfy the above. We denote $p_0=\varphi_1(0),\ p_1=\varphi_1(1),\ q_0=\varphi_2(0),\ q_1=\varphi_2(1)$ then
\[
\begin{array}{ccc}
\Psi\left(x\right)+p_0+2xq_0\leq c\left(x,0\right) & ; & \forall x\in\left[0,1\right]\\
\Psi\left(x\right)+p_1+2xq_1\leq c\left(x,1\right) & ; & \forall x\in\left[0,1\right]\\
\Psi\left(x\right)+p_0+2xq_0=c\left(x,0\right) &  & \forall x\in\left[0,\frac{1}{2}\right]\\
\Psi\left(x\right)+p_1+2xq_1=c\left(x,1\right) &  & \forall x\in\left(\frac{1}{2},1\right]
\end{array}
\]
By the last two equalities:

\[
\Psi\left(x\right)=\begin{cases}
c\left(x,0\right)-p_0-2xq_0 & x\in\left[0,\frac{1}{2}\right]\\
c\left(x,1\right)-p_1-2xq_1 & x\in\left(\frac{1}{2},1\right]
\end{cases}
\]
Substitute this $\Psi$ in the two inequalities leads to
\[
\begin{array}{ccc}
c\left(x,1\right)-p_1-2xq_1+p_0+2xq_0\leq c\left(x,0\right) & ; & \forall x\in\left(\frac{1}{2},1\right]\\
c\left(x,0\right)-p_0-2xq_0+p_1+2xq_1\leq c\left(x,1\right) & ; & \forall x\in\left[0,\frac{1}{2}\right]
\end{array}
\]
Denote $p=p_1-p_0$ and $q=q_1-q_0$ then:
     
\begin{eqnarray}{}
 \label{eq:nonex}
 \begin{split}
c\left(x,1\right)-c\left(x,0\right)\leq p+2xq & ; & \forall x\in\left(\frac{1}{2},1\right]\\
p+2xq\leq c\left(x,1\right)-c\left(x,0\right) & ; & \forall x\in\left[0,\frac{1}{2}\right]
 \end{split}
\end{eqnarray}

Choose $c(\cdot,0)\equiv0$ and
\[
c(x,1)=\begin{cases}
0 & ;x\in\left[0,\frac{1}{2}\right]\\
\sqrt{x-\frac{1}{2}} & ;x\in\left(\frac{1}{2},1\right]
\end{cases}
\]
By taking $x\rightarrow\frac{1}{2}$ in (\ref{eq:nonex}) we get $p+q=0\Rightarrow p=-q$ and again by (\ref{eq:nonex})
 $\sqrt{x-\frac{1}{2}}\geq(2x-1)q\ \forall x\in\left(\frac{1}{2},1\right]$
hence $1\geq2\sqrt{x-\frac{1}{2}}q\ \forall x\in\left(\frac{1}{2},1\right]$.
By taking $x\rightarrow\frac{1}{2}^+$ we get a contradiction.
\end{example}

\begin{rem}
We chose that $c$ since its slope is unbounded near $x=\frac{1}{2}$. One can show, using similar arguments, that for any $c(\cdot,1)-c(\cdot,0)$ that has a bounded slope near $x=\frac{1}{2}$, this particular dual problem described above admits a solution.
\end{rem}

\begin{rem}
Theorem 2 in \cite{Wolansky2020} gives sufficient conditions under which such examples may be obtained.
\end{rem}

For a fixed $\mu\in\mm(X,S)$, we can use Theorem \ref{thm:ExMin} to show the dual problem admits a solution on some dense subset of
$$
N_{\m}:=\left\{\n\in\mm(Y,S);\ \m\succ\n\right\}=\left\{\n\in\mm(Y,S);\ \Pi\left(\m,\n\dm\right)\neq\emptyset\right\}
$$
where the last equality is implied by Corollary \ref{cor:succ iff Pi not empty}. In the particular case where $Y$ is finite and $N_{\mu}$ is a subset of the Euclidean space, we can identify this dense subset as the relative interior of $N_{\m}$. We denote $Y=\{1,\ldots ,m\}$, $C(Y,\f)\cong \f^m\cong\r^{n\x m}$ so the dual problem is reduced to
\begin{multline*}
\sup\left\{\int_X\Psi\d|\m|+\sum_{i=1}^n\int_Y\varphi_i\d\n_i;\ \Psi+\langle \eta,\varphi\rangle\leq c\right\}
=\\=
\sup\left\{\int_X\Psi\d|\m|+\sum_{i=1}^n\sum_{j=1}^m \varphi_i^j\n_i^j;\ \Psi(x)+\sum_{i=1}^n\eta_i(x)\varphi_i^j\leq c_j(x)\ \forall x\in X\ ,j\in Y\right\}
\end{multline*}

\colorbox{thmcolor}{
\parbox{\linewidth}{
\begin{thm}
\label{thm:ExMin Semi-discrete}
When $Y$ is finite and $c,\eta$ are (Lipschitz) continuous, the dual problem (\ref{eq:dualproblem}) admits a (Lipschitz) continuous solution for any $\n$ in the relative interior of $N_{\m}$.
\end{thm}
}}
\begin{proof}
By considering only $\Psi$ of the form  $\Psi(x)=\min_j c_j(x)-\sum_{i=1}^n\varphi_i^j\n_i^j$, the infimum will not change so the dual problem is 
\begin{multline*}
F(\n)
:=
\sup_{\varphi\in \r^{n\times m}}\sum_{i=1}^n\sum_{j=1}^m\varphi_i^j\n_i^j+\int_X\min_j\left( c_j(x)-\sum_{i=1}^n\varphi_i^j\eta_i(x)\right)\d|\m|(x)
=\\=
\sup_{\varphi\in \r^{n\times m}}\sum_{i=1}^n\sum_{j=1}^m\varphi_i^j\n_i^j-\int_X\max_j\left( \sum_{i=1}^n\varphi_i^j\eta_i(x)-c_j(x)\right)\d|\m|(x)
\end{multline*}

By the abstract duality theorem (Theorem \ref{Thm:AbsDuality}), $F$ is finite iff $\Pi(\m,\n,\eta)\neq\emptyset$ which is true iff $\m\succ \n$ by Corollary \ref{cor:succ iff Pi not empty}. Therefore the domain of $F$ (the set on which $F<\infty$) is exactly $\text{dom}(F)=N_{\m}$.
$F:\r^{n\x m}\rightarrow\r\cup\{\infty\}$ is convex, so by Proposition IV.1.2.1 in \cite{hiriart2013convex} it is subdifferentiable on the relative interior of its domain, meaning that for any $\n$ in the relative interior of $N_{\m}$, there exists some $\phi\in\partial F(\n)$ (the subdifferential of $F$ at $\n$). $F$ is also the convex conjugate of $G(\varphi)=\int_X\max_j\left( \sum_{i=1}^n\varphi_i^j\eta_i(x)-c_j(x)\right)\d|\m|(x)$, which is convex and lower semi-continuous, so by Proposition 2.33 in \cite{barbu2012convexity}
$$
\sup_{\varphi\in \r^{n\times m}}\sum_{i=1}^n \sum_{j=1}^m\varphi_i^j\\n_i^j-G(\varphi)
=
F(\n)
=
\sum_{i=1}^n\sum_{j=1}^m \phi_i^j\n_i^j-G(\phi)
$$
(where we used the equivalence of part (iii) and (iv) in that proposition). Therefore
$$
\sup\left\{\int_X\Psi\d|\m|+\sum_{i=1}^n\sum_{j=1}^m \varphi_i^j\n_i^j;\ \Psi(x)+\sum_{i=1}^n\eta_i(x)\varphi_i^j\leq c_j(x)\right\}
$$
is attained for $\phi$ and $\Psi(x)=\min_j c_j(x)-\sum_{i=1}^n\eta_i(x)\phi_i^j$. So if $c,\eta$ are Lipschitz continuous so is $\psi$.
\end{proof}

For a compact metric space $Y$ (not necessarily a finite set) we prove the existence of solutions for a weak* dense subset of $N_\mu$:

\b

\colorbox{thmcolor}{
\parbox{\linewidth}{
\begin{thm}
\label{thm:ExMinVec}

 If $c,\eta$ are Lipschitz continuous, the dual problem (\ref{eq:dualproblem}) admits a Lipschitz continuous solution on a weak*-dense subset of $N_{\m}$.

\end{thm}
}}

\begin{proof}
Fix $\m\in\mm(X,S)$ and $\n\in N_{\m}$.
Substitute in Theorem \ref{thm:ExMin}:
        \begin{itemize}
            \item $K=C(X)\x C(Y)^n$
            \item $D=\{(\Psi,\varphi)\in K;\ |\varphi_i(y)|\leq 1,\ |\varphi_i(y)-\varphi_i(y')|\leq d(y,y')\ \forall y,y',i\}$.
            \item $V=C(X\x Y)$, $V^*=\mm(X\x Y)$
            \item $h(\Psi,\varphi)=\Psi+\left\langle\varphi,\eta\right\rangle$
            \item $\Phi(c,\eta):=\{(\Psi,\varphi)\in K;\ h(\Psi,\varphi)\geq c\}$
            \item $k^*(\Psi,\varphi):=\int_X \Psi\d|\m|+\int_Y\varphi\d\n$.
            \item Take $s(v):=(||v||_{\infty},0)\in\Phi(v,\eta)\cap n\cdot D$ which satisfies Assumption \ref{assu:Sublinear}.
        \end{itemize}
    A minimizing sequence $(\Psi_n,\varphi_n)\in\Phi(c,\eta)\cap (n\cdot D)$ of $k^*$ can be replaced with $(\max_y\{c(x,y)-\langle \eta(x),\varphi_n(y)\rangle\},\varphi_n)$ which is uniformly bounded and Lipschitz so by the Arzela-Ascoli theorem\index{Arzela-Ascoli theorem} (Theorem A.4.6 in \cite{shalit2017first}) it has a converging subsequence (in the supremum norm) and therefore $\inf_{(n\cdot D)\cap\Phi(c,\eta)}k^*(k)$ is attained (at the limit of the subsequence).
    By Theorem \ref{thm:ExMin}, for any $\varepsilon>0$ there exists some $\pi\in\mm(X\x Y)=V^*$ so that
    $$
    \inf\left\{\int_{X\x Y}\Psi(x)+\left\langle\varphi(y),\eta(x)\right\rangle\d\pi(x,y);\ \Psi(x)+\left\langle\varphi(y),\eta(x)\right\rangle\geq c(x,y)\right\}
    $$
    is attained and
    $$\sup_{(\Psi,\varphi)\in D}\left\{\int_{X\x Y}\Psi(x)+\left\langle\varphi(y),\eta(x)\right\rangle\d\pi(x,y)-\int_X \Psi\d|\m|-\sum_{i=1}^n\int_Y\varphi_i\d\n_i\right\}\leq\varepsilon$$
    
    Since the supremum is finite and $\Psi\in C(X)$ is arbitrary, we conclude the marginal of $\pi$ on $X$ is $|\mu|$ and therefore 
    \begin{multline*}
    \int_{X\x Y}\left[\Psi(x)+\left\langle\varphi(y),\eta(x)\right\rangle\right]\d\pi-\int_X \Psi\d|\m|-\sum_{i=1}^n\int_Y\varphi_i\d\n_i
    = 
    \int_{X\x Y}\left\langle\varphi(y),\eta(x)\right\rangle\d\pi-\sum_{i=1}^n\int_Y\varphi_i\d\n_i
    \end{multline*}
    Denote by $\nu'_i$ the marginal of $(\eta_i,\pi)$ on $Y$,  then
    $
    \int_{X\x Y}\left\langle\varphi(y),\eta(x)\right\rangle \d\pi
    =
    \sum_{i=1}^n\int_{Y}\varphi_i\d\nu'_i
    $
    so 
    
    $$
    \inf\left\{\int_X\Psi\d|\m|+ \sum_{i=1}^n\int_{Y}\varphi_i\d\nu'_i;\ \Psi(x)+\left\langle\varphi(y),\eta(x)\right\rangle\geq c(x,y)\right\}
    $$
    
    is attained and 
    
    $$
    \sup_{\varphi}\sum_{i=1}^n\int_{Y}\varphi_i\d\nu'_i-\int_Y\varphi_i\d\nu_i
    =
    \sum_{i=1}^n||\n_i-\n'_i||_{KR}\leq\varepsilon
    $$
    where the supremum is taken over all 1-Lipschitz $\varphi$ that are bounded between -1 and 1. As already mentioned in Example \ref{ex:ExMin}, the norm $\|\cdot\|_{KR}$ is sometimes called the "Kantorovich-Rubinstein" norm\index{Kantorovich-Rubinstein norm} and is known to metricize the weak convergence for probability measures (see Proposition 7.1.5 in \cite{Ambrosio2005}, Theorem 8.3.2 in \cite{bogachev2007measure}, Theorem 5.9 in \cite{santambrogio2015optimal} or Theorem 7.12 in \cite{Villani2003}).
\end{proof}

We will later get a slightly stronger result using a different proof (Theorem \ref{thm:ExMinCont}).

\subsection{Existence of an optimal transport map in the semi-discrete case}

Let $\mu\in\mathcal{M}\left(X,S\right),\ \nu\in\mathcal{M}\left(Y,S\right)$. We recall that $\Pi\left(\mu,\nu,\dm\right)$
is the set of all $\pi\in\mm_+(X\x Y)$
such that
\[
\int_{X\x Y}\left\langle \psi\left(x\right)+\varphi\left(y\right),\dm\left(x\right)\right\rangle \d\pi\left(x,y\right)=\int_{X}\psi\d\mu+\int_{Y}\varphi\d\nu
\]
 for all $\left(\psi,\varphi\right)\in C\left(X,\f\right)\x C\left(Y,\f\right)$.
Such measures are called transport plans.

\begin{defn}
A \emph{transport map} sending $\mu\in\mm(X,S)$ to $\nu\in\mm(Y,S)$ is a measurable map $T:X\rightarrow Y$ such that $T_\#\mu=\nu$, namely $T_\#\mu(B):=\mu(T^{-1}(B))=\nu(B)$ for all measurable $B\subset Y$. Equivalently (by Claim \ref{claim:cont is enough} for $\mu_1=T_\#\mu$ and $\mu_2=\nu$, see also Remark \ref{ex:pushforward}):
$$\int_Xf(Tx)\d\mu(x)=\int_Yf(y)\d\nu(y)\ \forall f\in C(Y,\f)$$
We denote the family of these transport maps
as $\mathcal{T}\left(\mu,\nu\right)$\index{Tmn@$\mathcal{T}\left(\mu,\nu\right)$}.
\end{defn}

\begin{claim}
$\ $
\begin{enumerate}
    \item If $T\in\mathcal{T}\left(\mu,\nu\right)$ then $(\text{Id}_X\x T)_\#|\mu|\in\Pi\left(\mu,\nu,\dm\right)$.
    \item If $\pi\in\Pi\left(\mu,\nu,\dm\right)$ is supported on a graph of a Borel measurable map $T:X\rightarrow Y$ then $T\in\mathcal{T}\left(\mu,\nu\right)$.
\end{enumerate}

\end{claim}

\begin{proof}
$\ $
\begin{enumerate}
    \item Let $T\in\mathcal{T}\left(\mu,\nu\right)$. For any $(\psi,\varphi)\in C(X,\f)\x  C(Y,\f)$
    \begin{multline*}
    \int_{X\x Y}\left\langle \psi(x)+\varphi(y),\dm(x)\right\rangle\d(\text{Id}_X\x T)_\#|\mu|(x,y)
    =
    \int_{X}\left\langle \psi(x)+\varphi(Tx),\dm(x)\right\rangle\d|\mu|(x)
    =\\=
    \int_{X}\psi(x)\d\mu+\int_{X}\varphi(Tx)\d\mu(x)
    =
    \int_{X}\psi(x)\d\mu+\int_{Y}\varphi(y)\d\mu(y)
    \end{multline*}
    Thus $(\text{Id}_X\x T)_\#|\mu|\in\Pi\left(\mu,\nu,\dm\right)$.
    \item Let $\pi\in\Pi\left(\mu,\nu,\dm\right)$ supported on the graph of a Borel measurable $T:X\rightarrow Y$. Denote its marginal on $X$ by $\pi_X$ and the graph of $T$ by 
$$G(T)=\{(x,y)\in X\x Y;\ Tx=y\}$$
By Corollary \ref{cor:DisInt product measures}, there exists some $P\in\text{ker}(X,Y)$ such that $\pi=P\otimes\pi_X$. For any $f\in C(X\x Y)$:
\begin{multline*}
    \int_{X\x Y}f\d\pi
    =
    \int_{G(T)}f\d\pi
    =
    \int_X\int_Y\chi_{G(T)}(x,y)f(x,y)\d P_x(y)\d\pi_X(x)
    =\\=
    \int_{X}f(x,Tx)\d\pi_X(x)
    =
    \int_{X\x Y}f\d(\text{Id}_X\x T)_\#\pi_X
\end{multline*}
thus $\pi=(\text{Id}_X\x T)_\#\pi_X$ and for any $f\in C(Y,\f)$ we have 
\begin{multline*}
    \int_Xf(Tx)\d\mu(x)
    =
    \int_{X\x Y}\left\langle f(Tx),\dm(x)\right\rangle\d\pi(x,y)
    =\\=
    \int_X\left\langle f(Tx),\dm(x)\right\rangle\d\pi_X(x)
    =
    \int_{X\x Y}\left\langle f(y),\dm(x)\right\rangle\d\pi(x,y)
    =
    \int_Yf(y)\d\nu(y)
\end{multline*}

and therefore $T_\#\mu=\nu$.
\end{enumerate}
\end{proof}

\begin{rem}
Due to the above claim, we allow ourselves to identify $\mathcal{T}(\mu,\nu)$ with the subset of $\Pi\left(\mu,\nu,\dm\right)$ of measures supported on a graph of some measurable $T:X\rightarrow Y$.
\end{rem}

When $Y=\{1,\ldots,n\}$ is finite, $\mm(Y,S)=S^n$ thus $\nu$ is just a vector with values in $S$ and $\mm_+(X\x Y)=\mm_+(X)^n$ so transport plans are just vectors of non-negative measures. In this case, the transport plans/maps can be characterized as following:

\begin{claim}
\label{claim:SemiDisTransChar}
Assume $Y=\{1,\ldots,n\}$, let $\pi=(\pi_1,\ldots,\pi_n)\in\mm_+(X)^n$ and denote $\pi_X=\sum_{i=1}^n\pi_i$, then
\begin{enumerate}
    \item $\pi\in\Pi\left(\m,\n,\dm\right)$ iff $\left(\dm,\pi_X\right)=\mu$ and $\int_X\dm\d\pi_i=\n_i\ \forall i$.
    \item Assume $\pi\in\Pi\left(\m,\n,\dm\right)$ then $\pi\in\mathcal{T}\left(\mu,\nu\right)$ iff for $\pi_X$-almost every $x\in X$, $(\frac{\d\pi_1}{\d\pi_X}(x),\ldots,\frac{\d\pi_n}{\d\pi_X}(x))$ is in the standard basis of $\rn$.
\end{enumerate}
\end{claim}
\begin{proof}
Let $\pi=(\pi_1,\ldots,\pi_n)\in\mm_+(X)^n$.
\begin{enumerate}
    \item By definition, $\pi\in\Pi\left(\m,\n,\dm\right)$ iff for any $\psi\in C(X,\f),\varphi\in C(Y,\f)=\f^n$
    \[
    \label{eq:consYfinite}
    \sum_{i=1}^n\int_X\left\langle\psi(x)+\varphi_i,\dm(x)\right\rangle\d\pi_i(x)=\int_X\psi\d\m+\sum_{i=1}^n\left\langle\varphi_i,\n_i\right\rangle
    \]
    which is equivalent to
    \begin{itemize}
        \item $\int_X\dm\d\pi_i=\n_i\ \forall i$ (by taking $\psi\equiv0$) and
        \item $\int_X\left\langle\psi(x),\dm(x)     \right\rangle\d\pi_X(x)=\sum_{i=1}^n\int_X\left\langle\psi(x),\dm(x)     \right\rangle\d\pi_i(x)=\int_X\psi\d\m$
             which is equivalent to $\left(\dm,\pi_X\right)=\mu$ by Claim \ref{claim:cont is enough}.
    \end{itemize}
    
    \item Assume the support of $\pi$ is contained in the graph of $T:X\rightarrow Y$ which is the set $\{(x,i);\ Tx=i\}$. For every $i\neq j$ and $A\subset T^{-1}(j)$, $A\x\{i\}$ is in the complement of the support of $\pi$ thus
    $$
    0=\pi(A\x\{i\})=\pi_i(A)=\int_A \frac{\d\pi_i}{\d\pi_X}\d\pi_X
    $$
    and therefore $\frac{\d\pi_i}{\d\pi_X}(x)=0$ for $\pi_X$-almost every $x\notin T^{-1}(i)$. We conclude that for $\pi_X$-almost every $x\in X$, $\frac{\d\pi_i}{\d\pi_X}(x)=0$ for all $i\neq Tx$. On the other hand, If $\frac{\d\pi}{\d\pi_X}(x)$ is a standard basis vector of $\rn$, then $\pi$ is supported on the graph of $T$ defined by $\frac{\d\pi_i}{\d\pi_X}(x)\neq0\Leftrightarrow Tx=i$.
\end{enumerate}
\end{proof}

To prove existence of transport maps we will use the notion of extreme points:
\begin{defn}
Let $C$ be a convex subset of a vector space. We say $c\in C$ is an \emph{extreme point}\index{extreme point} of $C$ if for any $x,y\in C$:
$$\frac{x+y}{2}=c\Rightarrow x=y=c$$
\end{defn}

We want to identify the extreme points of $\Pi\left(\m,\n,\dm\right)$ as transport maps. In the finite-dimensional case $(S=\r^m)$, this identification is possible  when $\mu$ is assumed to be non-atomic, due to Lyapunov's convexity theorem\index{Lyapunov's convexity theorem} (Theorem 13.33 in \cite{Aliprantis2006Hitch}) which guarantees that every finite-dimensional vector measure has a convex (and compact) range. In the general case (when $S$ is infinite-dimensional), we will need the following stronger assumption about $\m$:
\begin{assumption}
\label{Ass:Lyap}
For any measurable $A\subset X$ there exists some measurable $B\subset A$ and $\alpha\in(0,1)$ such that $\m(B)=\alpha\m(A)$.
\end{assumption}

\begin{rem}
We recall that a measure is non-atomic if for any measurable $A\subset X$ such that $\mu(A)\neq0$ there exists some measurable $B\subset A$ such that $\m(B)\neq0,\mu(A)$. Any measure with a convex range satisfies Assumption \ref{Ass:Lyap}, and any measure satisfying Assumption \ref{Ass:Lyap} must be non-atomic. If $S$ is finite-dimensional, any non-atomic S-valued measure has a convex range by Lyapunov's convexity theorem and therefore satisfy Assumption \ref{Ass:Lyap} as well.
\end{rem}

\begin{prop}
\label{prop:Ex pt is Map}
Under Assumption \ref{Ass:Lyap}, if $Y$ is finite then every extreme point of $\Pi\left(\m,\n,\dm\right)$ is in $\mathcal{T}\left(\mu,\nu\right)$.
\end{prop}
\begin{proof}

    Assume $Y=\{1,\ldots,n\}$. Let $\pi=(\pi_1,\ldots,\pi_n)\in\Pi\left(\m,\n,\dm\right)\setminus\mathcal{T}\left(\mu,\nu\right)$, we will show $\pi$ is not an extreme point of $\Pi\left(\m,\n,\dm\right)$. Since $\pi$ is not a transport map, by Claim \ref{claim:SemiDisTransChar} there exists some measurable $A\subset X$, $j\neq k$ and $\varepsilon>0$ such that $\pi_X(A):=\sum_{i\in Y}\pi_i(A)>0$ and $\frac{\d\pi_j}{\d\pi_X}(x),\frac{\d\pi_k}{\d\pi_X}(x)\in(\varepsilon,1-\varepsilon)\ \forall x\in A$. By Assumption \ref{Ass:Lyap} there exists some measurable $B\subset A$ and $\alpha\in(0,1)$ such that $\m(B)=\alpha\mu(A)$ and by Claim \ref{claim:SemiDisTransChar} we get $\int_B\dm\d\pi_X=\alpha\int_A\dm\d\pi_X$. Define a pair vector measures $\pi^{\pm}=(\pi_i^{\pm})_{i\in Y}$ by
    \[
    \pi_i^{\pm}(E):=
    \begin{cases}
     \pi_i(E) &;i\notin\{j,k\}
     \\
     \pi_j(E)\pm\varepsilon(\alpha\pi_X(E\cap A)-\pi_X(E\cap B)) &;i=j
     \\
      \pi_k(E)\mp\varepsilon(\alpha\pi_X(E\cap A)-\pi_X(E\cap B)) &;i=k
    \end{cases}
    \]
    The measures $\pi^{\pm}$ satisfy:
    \begin{itemize}
        \item $\sum_{i\in Y}\pi^{\pm}_i=\sum_{i\in Y}\pi_i=\pi_X$.
        
        \item $\pi^{\pm}_i\in\mm_+(X)\ \forall i\in Y$ since $\pi_i^{\pm}(E)\geq\pi_i(E)-\varepsilon\pi_X(E)\geq0$ for all $i\in\{j,k\},E\subset X$, where the last inequality is implied by $\frac{\d\pi_j}{\d\pi_X},\frac{\d\pi_k}{\d\pi_X}>\varepsilon$.
        
        \item $\int_X\dm\d\pi^{\pm}_i=\int_X\dm\d\pi_i=\n_i\ \forall i\in Y$ by Claim \ref{claim:SemiDisTransChar}.
    \end{itemize}
     and therefore by Claim \ref{claim:SemiDisTransChar} $\pi^{\pm}\in\Pi\left(\m,\n,\dm\right)$. Moreover, since $\pi_X(A)>0$, there exists some $E\subset X$ such that $\alpha\pi_X( E\cap A)\neq\pi_X(E\cap B)$ (one of  $B,A\setminus B$ will do) so for that $E$, $\pi^+(E),\pi^-(E),\pi(E)$ are all different, and since $\frac{\pi^++\pi^-}{2}=\pi$ we conclude $\pi$ is not an extreme point of $\Pi\left(\m,\n,\dm\right)$.
\end{proof}

\colorbox{thmcolor}{
\parbox{\linewidth}{
\begin{thm}
\label{thm:ExTransMap}
Assume \ref{ass:Vec} and \ref{Ass:Lyap}. If $Y$ is finite, $\dm\in C(X,(S,w(\f)))$ is norm bounded and $\Pi\left(\m,\n,\dm\right)\neq\emptyset$ then $\mathcal{T}\left(\mu,\nu\right)\neq\emptyset$. Moreover, for every lower semicontinuous $c:X\x Y\rightarrow\r$ there exists an optimal transport map, meaning
$$
\min_{\pi\in\mathcal{T}\left(\mu,\nu\right)}\int_{X\x Y} c\d\pi
=
\min_{\pi\in\Pi\left(\m,\n,\dm\right)}\int_{X\x Y} c\d\pi
$$
\end{thm}
}}

\begin{proof}
Let $c:X\x Y\rightarrow\r$ a lower semicontinuous function and define $F:\Pi\left(\m,\n,\dm\right)\rightarrow\r$ by $F(\pi):=\int_{X\x Y} c\d\pi$. $F$ is a convex and weak*-lower semicontinuous functional (see Lemma 4.3 in \cite{Villani2016}) defined on $\Pi\left(\m,\n,\dm\right)$ which is weak*-compact under Assumption \ref{ass:Vec} (see Claim \ref{claim:compactness}) and convex (immediate by definition), so by the Bauer maximum principle (Theorem 7.69 in \cite{Aliprantis2006Hitch}) it attains its minimum on an extreme point of $\Pi\left(\m,\n,\dm\right)$ which is an element in $\mathcal{T}\left(\mu,\nu\right)$ by Proposition \ref{prop:Ex pt is Map}. So in particular $\mathcal{T}\left(\mu,\nu\right)\neq\emptyset$.
\end{proof}
  
Using the above theorem we can formulate the relation $\m\succ p$ between a continuous measure $\mu$ and a discrete measure $p$ using partitions:
\begin{cor}
\label{cor:Relation discrete Partitions}
Let $\mu\in\mm\left(X,S\right)$ satisfying Assumption \ref{Ass:Lyap}, $\eta=\dm$ satisfying Assumption \ref{ass:Vec} and $p\in M\left(\left\{ 1,\ldots,n\right\} ,S\right)$. Then $\mu\succ p$ iff there exists $\{A_i\}_{i=1}^n$ a Borel measurable partition of $X$ such that $\m(A_i)=p_i\ \forall i$.
\end{cor}
\begin{proof}
\begin{itemize}
    \item Assume $\mu\succ p$: By Corollary \ref{cor:succ iff Pi not empty}, $\Pi\left(\m,p,\dm\right)\neq\emptyset$ thus by Theorem \ref{thm:ExTransMap} there exists $\pi\in\mathcal{T}\left(\mu,p\right)$. Define $A_i=\{x\in X;\ \frac{\d\pi_i}{\d\pi_X}(x)=1\}$ for all $i\in\{1,\ldots,n\}$. By Claim \ref{claim:SemiDisTransChar}, for every $i=1,\ldots,n$ we have $\frac{\d\pi_i}{\d\pi_X}(x)=0$ for $\pi_X$-almost every $x\notin A_i$ therefore by Claim \ref{claim:SemiDisTransChar} again
    $$
    \m(A_i)
    =
    \int_{A_i}\dm\d\pi_X
    =
    \int_{A_i}\dm\frac{\d\pi_i}{\d\pi_X}\d\pi_X
    =
    \int_X\dm\frac{\d\pi_i}{\d\pi_X}\d\pi_X
    =
    \int_X\dm\d\pi_i
    =
    p_i
    $$
    
    \item Assume there exists $\{A_i\}_{i=1}^n\subset 2^X$ a Borel measurable partition of $X$ such that $\m(A_i)=p_i\ \forall i$: Define $T:X\rightarrow\{1,\ldots,n\}$ by $x\in A_i\Leftrightarrow Tx=i$. $T_\#\m=p$ so by Claim \ref{claim:map induces kernel} $\mu\succ p$.
\end{itemize}
\end{proof}

To show the existence of a transport map in the general case, we will approximate the target measure $\nu$ with finitely supported measures.

\subsection{Domination of discrete measures}

The optimal transport problem where one of the spaces is finite (as
in Example \ref{Ex:InfMarket}) is simpler than the general case,
and stronger results may be easier to achieve (see for example Theorem \ref{thm:ExTransMap}). Such results may be applied to the general case using approximation with measures supported on a finite set (see for example the proof of Theorem 5.10 - Kantorovich duality - in \cite{Villani2016}). In this part of the chapter we show that the preorder in definition
\ref{def:relation} \textquotedbl respects\textquotedbl{} such
finite approximations. We begin by defining a weaker sequence of preorders:
\begin{defn}
For any $\mu\in\mm\left(X,S\right),\ \n\in\mm\left(Y,S\right)$
and $n\in\mathbb{N}$ we denote $\mu\succ_{n}\nu$\index{1@$\succ_{n}$}
whenever $\nu\succ p\Rightarrow \mu\succ p$ for every $p\in \mm\left(\left\{ 1,\ldots,n\right\} ,S\right)$.
\end{defn}

The above relation may be formulated in an equivalent way using partitions as described in the following claim. This claim follows directly from Corollary \ref{cor:Relation discrete Partitions}.

\begin{claim}
\label{claim:finite dominance partitions}
Let $\mu\in\mm\left(X,S\right),\n\in\mm\left(Y,S\right)$ both satisfying Assumption \ref{Ass:Lyap} and $\dm,\dn$ satisfying Assumption \ref{ass:Vec}. Then $\mu\succ_{n}\nu$ iff for any $\{B_i\}_{i=1}^n$ a measurable partition of $Y$ there exists $\{A_i\}_{i=1}^n$ a measurable partition of $X$ such that $\m(A_i)=\n(B_i)\ \forall i$.
\end{claim}

By the claim above we immediately get the following:

\begin{cor}
Under the same assumptions of Claim \ref{claim:finite dominance partitions}:
\begin{enumerate}
    \item $\mu\succ_1\nu$ iff $\mu(X)=\nu(Y)$.
    \item $\mu\succ_2\nu$ iff $\{\mu(A);\ A\subset X\}\supset\{\nu(A);\ A\subset Y\}$ and $\mu(X)=\nu(Y)$.
\end{enumerate}
\end{cor}

By the above corollary it is clear that $\mu\succ_2\nu\Rightarrow\mu\succ_1\nu$. The same is true for higher orders:

\begin{claim}
\label{claim:n n+1}
Let $\mu\in\mm\left(X,S\right),\ \n\in\mm\left(Y,S\right)$, $n\in\mathbb{N}$ then 
$
\mu\succ_{n+1}\nu\Rightarrow\mu\succ_n\nu
$.
\end{claim}
\begin{proof}
Assume $\mu\succ_{n+1}\nu$, need to show $\mu\succ_n\nu$. To show that, let $p\in \mm\left(\left\{1,\ldots,n\right\} ,S\right)$ such that $\nu\succ p$, we need to show $\mu\succ p$. $p$ is a $S$-valued measure on $\{1,\ldots,n\}$ so it is of the form $p=(p_1,\ldots,p_n)\in S^n$. There exists $P\in\text{ker}(Y,\{1,\ldots,n\})$ such that $P\nu=p$, namely $\int_YP_y(i)\d\n(y)=p_i$ for all $i\in\{1,\ldots,n\}$.
Denote $\tilde{p}=(p_1,\ldots,p_n,p_{n+1}=0)\in\mm\left(\left\{ 1,\ldots,n+1\right\} ,S\right)$ and $\tilde{P}\in\text{ker}(Y,\{1,\ldots,n+1\})$ by $\tilde{P}_y:=\left(P_y(1),\ldots,P_y(n),0\right)$. $\int_Y\tilde{P}_y(i)\d\n(y)=p_i$ for all $i\in\{1,\ldots,n+1\}$ and therefore $\nu\succ\tilde{p}$. We assumed $\mu\succ_{n+1}\nu$ thus $\mu\succ\tilde{p}$ which implies the existence of some $Q\in\text{ker}(X,\{1,\ldots,n+1\})$ such that $\int_XQ_x(i)\d\m(x)=p_i$ for all $i\in\{1,\ldots,n+1\}$. Define $\tilde{Q}\in\text{ker}(X,\{1,\ldots,n\})$ by $\tilde{Q}_x=(Q_x(1),\ldots,Q_x(n-1),Q_x(n)+Q_x(n
+1))$. $Q\mu=p$ thus $\mu\succ p$.
\end{proof}

\begin{cor}
For every $\mu\in\mm\left(X,S\right)$ and $\nu\in\mm\left(Y,S\right)$,
there exists $N\in\mathbb{N}\cup\left\{ \infty\right\} $ s.t $\mu\succ_{n}\nu\Leftrightarrow n<N$.
\end{cor}
  
\begin{proof}
If $\mu\succ_{n}\nu$ for every $n\in\mathbb{N}$
then $N=\infty$. Otherwise there exists a minimal $N$ for which
$\mu\not\succ_{N}\nu$, by minimality $\mu\succ_{n}\nu$ for every $n<N$.
By Claim \ref{claim:n n+1}, we know
that $\mu\succ_{n+1}\nu\Rightarrow\mu\succ_{n}\nu$ for every $n\in\mathbb{N}$
hence $\mu\not\succ_{n}\nu$ for every $n\geq N$.
\end{proof}

Our goal is to prove the connection between $\succ_{n}$ and $\succ$
given below:

\b

\colorbox{thmcolor}{
\parbox{\linewidth}{
\begin{thm}
\label{thm:order}Let $\mu\in\mm\left(X,S\right)$ with norm bounded $\dm\in C(X,(S,w(\f)))$ satisfying \ref{ass:Vec} and $\nu\in\mm\left(Y,S\right)$,
then $\mu\succ_{n}\nu\ \forall n\in\mathbb{N}\Leftrightarrow\mu\succ\nu$ (meaning $\succ=\cap_{n\in\mathbb{N}}\succ_n$).
\end{thm}
}}

\medskip

Before we prove the above theorem we need a few preliminary results.
\begin{notation}
For $\mu\in\mm\left(X,S\right),N\subset\mm\left(Y,S\right)$,
we denote $\Pi\left(\mu,N\right):=\cup_{\nu\in N}\Pi\left(\mu,\nu,\dm\right)$\index{pimn@$\Pi\left(\mu,N\right)$},
i.e. all  $\pi\in\mm_+(X\x Y)$ such that $\left(\dm,\pi\right)$ has marginal
$\mu$ on $X$ and marginal $\in N$ on $Y$.
\end{notation}

\begin{defn}
Each $g\in C(Y,\f)$ induces a linear functional on $\mm(Y,S)$ by $\nu\mapsto\int g\d\nu$. We denote the topology on $\mm(Y,S)$ generated by those functionals as the weak topology\index{weak topology}. We denote the convergence wrt this topology by $\rightharpoonup$. Explicitly: $\nu_n\rightharpoonup\nu$ iff $\int_Yg\d\nu_n\rightarrow\int_Yg\d\nu$ for all $g\in C(Y,\f)$.
\end{defn}

\begin{claim}
\label{claim:compactness}(Sequentially compactness of $\Pi\left(\mu,N\right)$)
Let $\mu\in\mm\left(X,S\right)$ with norm bounded $\dm\in C(X,(S,w(\f)))$ satisfying \ref{ass:Vec}.
If $N\subset\mm\left(Y,S\right)$ is weak-sequentially compact then $\Pi\left(\mu,N\right)\subset\mm\left(X\times Y\right)$
is weak*-sequentially compact and compact.
\end{claim}

\begin{proof}
By the Banach-Alaoglu\index{Banach-Alaoglu theorem} theorem (Theorem 6.21 in \cite{Aliprantis2006Hitch}) any norm-bounded weak*-closed subset of the dual space $\mm(X\x Y)$ is weak*-compact. Moreover, since $C(X\x Y)$ is separable, such a norm-bounded subset is weak* metrizable by Theorem 6.30 in \cite{Aliprantis2006Hitch}, so the notions of compactness and sequentially-compactness coincide.

All elements of $\Pi\left(\mu,N\right)\subset\mm(X\x Y)$ has total measure of $\pi\left(X\times Y\right)=\left|\mu\right|\left(X\right)$ (by Claim \ref{claim:pi has marginal |mu|})
hence $\Pi(\mu,N)$ is norm-bounded and we only need
to show it is weak*-closed. Take $\Pi\left(\mu,N\right)\ni\pi_n\stackrel[n\rightarrow\infty]{w^{*}}{\longrightarrow}\pi$
then $\left(\frac{\d\mu}{\d\left|\mu\right|},\pi_{n}\right)$ has marginals $\mu$
on $X$ and some marginal $\nu_{n}\in N$ on $Y$. $\nu_{n}$ has
a weak-converging subsequence (without loss of generality also denoted by $\nu_{n}$) $\nu_{n}\underset{n\rightarrow\infty}{\rightharpoonup}\nu\in N$.
For every $\left(\psi,\varphi\right)\in C\left(X,\f\right)\times C\left(Y,\f\right)$, $\left\langle \psi+\varphi,\frac{\d\mu}{\d\left|\mu\right|}\right\rangle\in C(X\x Y)$ thus
\begin{equation}
\label{eq:same arguments}
\begin{array}{ccc}
\int_{X\times Y}\left\langle \psi\left(x\right)+\varphi\left(y\right),\frac{\d\mu}{\d\left|\mu\right|}\left(x\right)\right\rangle \d\pi\left(x,y\right) & = & \lim_{n\rightarrow\infty}\int_{X\times Y}\left\langle \psi\left(x\right)+\varphi\left(y\right),\frac{\d\mu}{\d\left|\mu\right|}\left(x\right)\right\rangle \d\pi_{n}\left(x,y\right)\\
\\
 & = & \lim_{n\rightarrow\infty}\int_{X}\psi\left(x\right)\d\mu\left(x\right)+\int_{Y}\varphi\left(y\right)\d\nu_{n}\left(y\right)\\
\\
 & = & \int_{X}\psi\left(x\right)\d\mu\left(x\right)+\int_{Y}\varphi\left(y\right)\d\nu\left(y\right)
\end{array}
\end{equation}
Therefore $\pi\in\Pi\left(\mu,N\right)$ and $\Pi\left(\mu,N\right)$ is weak*-closed.
\end{proof}

\begin{lem}
\label{lem:ordcont}(Order is preserved by limit) Let $\mu\in\mm\left(X,S\right)$ with norm bounded $\dm\in C(X,(S,w(\f)))$ satisfying \ref{ass:Vec}, and let
$\left\{ \nu_{n}\right\} _{n=1}^{\infty}\subset\mm\left(Y,S\right)$. If $\nu_{n}\underset{n\rightarrow\infty}{\rightharpoonup}\nu$
and $\mu\succ\nu_{n}\ \forall n$ then $\mu\succ\nu$.
\end{lem}

\begin{proof}
For every $n\in\mathbb{N}$, $\mu\succ\nu_{n}$ hence by Claim \ref{claim:char of trans plans} there exists
$\pi_{n}\in\Pi\left(\mu,\nu_{n},\dm\right)$. $N:=\left\{ \nu_{n}\right\} _{n=1}^{\infty}\cup\left\{ \nu\right\} $
is weak-sequentially compact
and $\pi_{n}\in\Pi\left(\mu,\nu_{n},\dm\right)\subset\Pi\left(\mu,N\right)$ hence by Claim \ref{claim:compactness}, $\pi_{n}$ has a subsequence weak*-converging to some $\pi\in\Pi\left(\mu,N\right)$. By exactly the same arguments as in (\ref{eq:same arguments}) we conclude $\pi\in\Pi\left(\mu,\nu,\frac{\d\mu}{\d\left|\mu\right|}\right)$ and by Corollary \ref{cor:succ iff Pi not empty} $\mu\succ\nu$.
\end{proof}

In the next lemma,
we will show that for finitely supported measures, the underlying
space plays no role in the preorder.

\begin{lem}
(Equivalence of discrete and finitely supported measures)\label{lem:equivalence}
Let $\left(b_{1},\ldots,b_{n}\right)\in S^{n}$ and $\left(x_{1},\ldots,x_{n}\right)\in X^{n}$.
Define 
\[
\mu=\sum_{j=1}^{n}b_{j}\delta_{x_{j}}\in\mm\left(X,S\right)\ ,\ \nu=\left(b_{1},\ldots,b_{n}\right)\in\mm\left(\left\{ 1,\ldots,n\right\} ,S\right)
\]
then $\mu\succ\nu$ and $\nu\succ\mu$.
\end{lem}  
\begin{proof}
Define $T_1:\{1,\ldots,n\}\rightarrow X$ by $T_1(j)=x_j$. $T_{1\#}\n=\m$ so by Claim \ref{claim:map induces kernel} $\n\succ\m$. On the other hand for $T_2:X\rightarrow\{1,\ldots,n\}$ defined by $T_2(x_j)=j$ and $T_2(x)=1$ for all $x\notin\{x_j\}_{j=1}^n$ we get $T_{2\#}\m=\n$ so again by Claim \ref{claim:map induces kernel} $\m\succ\n$.
\end{proof}

\begin{claim}
\label{claim:(Discrete-aproximation)} (Discrete approximation) For any
$\mu\in\mm\left(X,S\right)$ there exist a sequence of finitely supported
measures $p_{n}\in\mm\left(X,S\right)$ such that $p_{n}\underset{n\rightarrow\infty}{\rightharpoonup}\mu$
and $\mu\succ p_{n}\ \forall n$ .
\end{claim}

\begin{proof}
Let $\mu\in \mm\left(X,S\right)$, denote the metric of $X$ by $d$. For any $n\in\mathbb{N}$, there
exists a finite measurable partition $\left\{ A_{i}^{n}\right\} _{i=1}^{M_{n}}$
such that $\text{diam}(A_i^n):=\sup\{d(x,y);x,y\in A^n_i\}\leq\frac{1}{n}\ \forall i$.
For every $n\in\mathbb{N}$ and $i=1,\ldots,M_{n}$, choose some $x_{i}^{n}\in A_{i}^{n}$
and define the sequence of finitely supported measures $p_{n}:=\sum_{i=1}^{M_{n}}\mu\left(A_{i}^{n}\right)\delta_{x_{i}^{n}}$. For every $f\in C(X,S^*)$:
\begin{multline*}
    \int_Xf\d p_n
    =
    \sum_{i=1}^{M_{n}}\int_X\left\langle \mu\left(A_{i}^{n}\right),f(x)\right\rangle\d\delta_{x_i^n}
    =
     \sum_{i=1}^{M_{n}}\left\langle \mu\left(A_{i}^{n}\right),f(x_i^n)\right\rangle
     =\\=
     \sum_{i=1}^{M_{n}}\left\langle \int_{A_i^n}\dm\d|\mu|,f(x_i^n)\right\rangle
     =
     \sum_{i=1}^{M_{n}}\int_{A_i^n}\left\langle \dm,f(x_i^n)\right\rangle\d|\mu|
\end{multline*}

\begin{itemize}
\item $p_{n}\underset{n\rightarrow\infty]}{\rightharpoonup}\mu$:
Let $f\in C\left(X,\f\right)$ and $\varepsilon>0$. Since $X$ is
compact, $f$ is also uniformly continuous hence there exist $N>0$
such that $d\left(x,y\right)\leq\frac{1}{N}\Rightarrow\left|\left|f\left(x\right)-f\left(y\right)\right|\right|\leq\varepsilon$
for any $x,y\in X$. So for every
$n>N$, by using Jensen's inequality (the scalar version) we get
\[
\begin{array}{ccccc}
\left|\int_{X}f\d\mu-\int_{X}f\d p_{n}\right|
& = &
\left|\sum_{i=1}^{M_{n}}\int_{A_{i}^{n}}\left\langle f,\dm\right\rangle\d|\mu|-\sum_{i=1}^{M_{n}}\int_{A_i^n}\left\langle f(x_i^n),\dm\right\rangle\d|\mu|\right|\\\\
& = &
 \left|\sum_{i=1}^{M_{n}}\int_{A_{i}^{n}}\left\langle f\left(x\right)-f\left(x_{i}^{n}\right),\dm(x)\right\rangle\d|\mu|(x)\right|\\\\
 & \leq &
\sum_{i=1}^{M_{n}}\int_{A_{i}^{n}}\left|\left\langle f\left(x\right)-f\left(x_{i}^{n}\right),\dm(x)\right\rangle\right|\d|\mu|(x)\\\\
 & \leq &
 \sum_{i=1}^{M_{n}}\int_{A_{i}^{n}}\left|\left|f\left(x\right)-f\left(x_{i}^{n}\right)\right|\right|\cdot\left|\left|\dm(x)\right|\right|\d\left|\mu\right|(x) \\\\
 & \leq &
 \varepsilon\sum_{i=1}^{M_{n}}\int_{A_{i}^{n}}\left|\left|\dm(x)\right|\right|\d\left|\mu\right|(x) \\\\
 & = &
  \varepsilon\int_X\left|\left|\dm(x)\right|\right|\d\left|\mu\right|(x)
 & = &
 \varepsilon V(\mu)(X) 
\end{array}
\]
hence $\int_{X}f\d p_{n}\underset{n\rightarrow\infty}{\longrightarrow}\int_{X}f\d\mu$
for every $f\in C\left(X,\f\right)$ thus $p_{n}\underset{n\rightarrow\infty}{\rightharpoonup}\mu$.
\item $\mu\succ p_{n}\ \forall n$: For any $n\in\mathbb{N}$ define $T_n:X\rightarrow X$ by $T_n(x)=x_i^n\ \forall x\in A_i^n$. $T_{n\#}\m=p_n$ so $\mu\succ p_{n}$ by Claim \ref{claim:map induces kernel}.
\end{itemize}
\end{proof}

We can now prove Theorem \ref{thm:order}:
\begin{proof}
(of Theorem \ref{thm:order})
\begin{itemize}
\item Assume $\mu\succ_{n}\nu\ \forall n$. By Claim \ref{claim:(Discrete-aproximation)}
there exists a sequence of finitely supported measures $p_{n}\in\mm\left(Y,S\right)$
such that $p_{n}\underset{n\rightarrow\infty}{\rightharpoonup}\nu$ and $\nu\succ p_{n}\ \forall n$. By Lemma \ref{lem:equivalence}, there exists a sequence $q_n\in\mm(\{1,\ldots,|\text{supp}(p_n)|\},S)$ such that $q_n\succ p_n$ and $p_n\succ q_n$. By transitivity (see Claim \ref{lem:relation}) $\nu\succ q_n\ \forall n$ thus $\mu\succ q_n\ \forall n$ and by transitivity again $\mu\succ p_n\ \forall n$. By Lemma \ref{lem:ordcont} we conclude $\mu\succ\nu$.
\item If $\mu\succ\nu$ then by transitivity (see Claim \ref{lem:relation}) $\mu\succ_{n}\nu\ \forall n$.
\end{itemize}
\end{proof}

\begin{cor}
\label{cor:dominance partitions}
Let $\mu\in\mm\left(X,S\right)$ and $\nu\in\mm\left(Y,S\right)$ both satisfying \ref{Ass:Lyap} and norm bounded $\dm\in C(X,(S,w(\f))),\ \dn\in C(Y,(S,w(\f)))$ satisfying \ref{ass:Vec}. Then $\mu\succ\nu$ iff for any $n\in\mathbb{N}$ and $\{B_i\}_{i=1}^n$ a measurable partition of $Y$ there exists $\{A_i\}_{i=1}^n$ a measurable partition of $X$ such that $\m(A_i)=\n(B_i)\ \forall i$.
\end{cor}
\begin{proof}
Directly by Theorem \ref{thm:order} and Claim \ref{claim:finite dominance partitions}.
\end{proof}

As a corollary from Claim \ref{claim:(Discrete-aproximation)}, we conclude a result similar to Theorem \ref{thm:ExMinVec}. We recall that $$N_{\mu}:=\{v\in\mm(Y,S);\ \mu\succ\nu\}$$

\colorbox{thmcolor}{
\parbox{\linewidth}{
\begin{thm}
\label{thm:ExMinCont}
Assume $S=S^*=\f=\r^d$ and $\eta=\dm$ satisfying \ref{ass:Vec}. If $c,\eta$ are (Lipschitz) continuous, the dual problem (\ref{eq:dualproblem}) admits a (Lipschitz) continuous solution on a weak*-dense subset of $N_{\m}$.  
\end{thm}
}}
\begin{proof}
Let $\nu\in N_{\mu}$. We will show there exists some $\nu'\in N_{\mu}$ which is arbitrary close to $\nu$ in the weak* topology and for which the dual problem (\ref{eq:dualproblem}) admits a solution. By Claim \ref{claim:(Discrete-aproximation)}, there exist some $\nu_n\in N_{\mu}$ supported on  $Y_n\subset Y$ some finite subset of $Y$ of size $n$, arbitrary close to $\nu$ in the weak* topology (When $S=\f=\r^d$, the weak topology on $\mm(Y,S)=\mm(Y)^d$ is just the weak* topology generated by $C(Y,S)=C(Y)^d$). $\nu_n\in N_{\mu}\cap\mm(Y_n,S)$ so by Theorem \ref{thm:ExMin Semi-discrete} there exists some $\nu'\in N_{\mu}\cap\mm(Y_n,S)$ arbitrary close to $\nu_n$ in the finite-dimensional standard topology ($N_{\mu}\cap\mm(Y_n,S)$ is isomorphic to a subset of $\r^{nd}$) for which the dual problem (\ref{eq:dualproblem}) admits a solution. Therefore $\nu'$ is arbitrary close to $\nu$ in the weak* topology. If $c,\eta$ are Lipschitz continuous, then so are the solutions to the dual problem by Theorem \ref{thm:ExMin Semi-discrete}.
\end{proof}

\subsection{Existence of a transport map in the general case}
We showed that in the semi-discrete case, there exists an optimal transport map whenever $\mu\succ\nu$ (Theorem \ref{thm:ExTransMap}). We will now use Theorem \ref{thm:order} to formulate equivalent conditions for the existence of some transport map (not necessarily optimal) in the general case.

\begin{defn}
A partition $\left\{C_1,\ldots,C_m\right\}$ is a \emph{refinement}\index{refinement} of a partition $\left\{D_1,\ldots,D_n\right\}$ if $\forall i=1,\ldots,m\ \exists j=1,\ldots,n$ such that $C_i\subset D_j$.
\end{defn}

\begin{defn}
We say $\mu\in\mm(X,S)$ \emph{sequentially dominates} \index{sequential domination} $\nu\in\mm(Y,S)$ if for any $\left\{ \mathcal{B}^{n}:=\left\{B_{1}^{n},\ldots,B_{M_{n}}^{n}\right\}\right\} _{n=1}^{\infty}$
sequence of partitions of $Y$ such that $\mathcal{B}^{n+1}$ is a refinement of $\mathcal{B}^{n}$  for all $n\in\mathbb{N}$, there exists a sequence of partitions $\{\mathcal{A}^{n}=\left\{ A_{1}^{n},\ldots,A_{M_{n}}^{n}\right\}\}_{n=1}^\infty$ of $X$ such that:
\begin{itemize}
\item $\mathcal{A}^{n+1}$ is a refinement of $\mathcal{A}^{n}$.
\item $\mu\left(A_{i}^{n}\right)=\nu\left(B_{i}^{n}\right)$ for any $n\in\mathbb{N}$
and $i=1,\ldots,M_{n}$.
\item $B_{i}^{n+1}\subset B_{j}^{n}\Leftrightarrow A_{i}^{n+1}\subset A_{j}^{n}$
for any $n\in\mathbb{N}$, $j=1,\ldots,M_{n}$ and $i=1,\ldots,M_{n+1}$
\end{itemize}
\end{defn}

\begin{claim}
If
\begin{enumerate}
    \item $\mu\in\mm(X,S),\ \nu\in\mm(Y,S)$ satisfy Assumption \ref{Ass:Lyap}.
    \item $\dm\in C(X,(S,w(\f))),\ \dn\in C(Y,(S,w(\f)))$ are norm bounded and satisfy \ref{ass:Vec}.
    \item $\mu$ sequentially dominates $\nu$
\end{enumerate}
then $\mu\succ\nu$.
\end{claim}
\begin{proof}
Directly by Corollary \ref{cor:dominance partitions}.
\end{proof}

Under some conditions, sequential dominance is equivalent to the existence of some (not necessarily optimal) transport map:

\b

\colorbox{thmcolor}{
\parbox{\linewidth}{
\begin{thm}
\label{thm:ExTransMap generalcase}
Assume
\begin{itemize}
    \item $\mu\in\mm(X,S),\ \nu\in\mm(Y,S)$ satisfy \ref{Ass:Lyap}.
    \item $\dm$ satisfies \ref{ass:Vec}.
    \item $Y$ is complete.
\end{itemize}
Then $\mu$ sequentially dominates $\nu$ iff $\mathcal{T}(\mu,\nu)\neq\emptyset$.
\end{thm}
}}

\begin{proof}
There exists $\left\{ \mathcal{B}^{n}\right\} _{n=1}^{\infty}$ a
sequence of partitions of $Y$ such that
\begin{itemize}
    \item $\mathcal{B}^n:=\left\{B_{1}^{n},\ldots,B_{M_{n}}^{n}\right\}$ where $B_i^n$ are non-empty $\forall n\in\mathbb{N},\ i=1,\ldots,M_n$.
    \item $\mathcal{B}^{n+1}$ is a refinement of $\mathcal{B}^{n}$ for all $n\in\mathbb{N}$.
    \item $\lim_{n\rightarrow\infty}\max_{B\in\mathcal{B}^{n}}\text{diam}B=0$.
\end{itemize}
By sequentially domination there exists a sequence of partitions $\{\mathcal{A}^{n}=\left\{ A_{1}^{n},\ldots,A_{M_{n}}^{n}\right\}\}_{n=1}^\infty $
satisfying:
\begin{itemize}
\item $\mathcal{A}^{n+1}$ is a refinement of $\mathcal{A}^{n}$.
\item $\mu\left(A_{i}^{n}\right)=\nu\left(B_{i}^{n}\right)$ for any $n\in\mathbb{N}$
and $i=1,\ldots,M_{n}$.
\item $B_{i}^{n+1}\subset B_{j}^{n}\Leftrightarrow A_{i}^{n+1}\subset A_{j}^{n}$
for any $n\in\mathbb{N}$, $j=1,\ldots,M_{n}$ and $i=1,\ldots,M_{n+1}$
\end{itemize}
For any $n\in\mathbb{N}$ and $i=1,\ldots,M_{n}$ choose some $y_{i}^{n}\in B_{i}^{n}$
and define $T_n:X\rightarrow Y$ by $T_{n}x:=y_{i}^{n}\ \forall x\in A_{i}^{n}$.
We will show that for any $x\in X$, $T_{n}x$ is a converging sequence: Let $x\in X$
and $\varepsilon>0$. There exists $N\in\mathbb{N}$ s.t. $\text{diam}(B_{i}^{n})<\varepsilon$
for any $n\geq N$ and $i=1,\ldots,M_{n}$. Assume without loss of generality $x\in A_1^N$ so $T_Nx=y_1^N\in B^N_1$. For any $m>n>N$ there exists some $i,j$ such that $x\in A_i^m\subset A_j^n\subset A^N_1$ and therefore $T_mx\in B^m_i$, $T_nx\in B^n_j$ and $B^m_i\subset B^n_j\subset B^N_1$ so $T_{n}x,T_{m}x\in B_1^{N}$
and thus $d\left(T_{n}x,T_{m}x\right)<\varepsilon$. We conclude $\left\{ T_{n}x\right\} $
is Cauchy and by completeness of $Y$ it converges. We denote the limit
$T=\text{lim}_{n\rightarrow\infty}T_{n}$ and since $\left\{ T_{n}\right\} _{n\in\mathbb{N}}$
are measurable so is $T$. Define $p_{n}\in\mathcal{M}\left(Y,S\right)$
by $p_{n}=\sum_{i=1}^{M_{n}}\mu\left(A_{i}^{n}\right)\delta_{y_{i}^{n}}$.
For any $n\in\mathbb{N}$ and
$f\in C\left(Y,\f\right)$
\begin{equation}
\label{Eq:Ex trans map}
\int_{X}f\left(T_{n}x\right)\text{d}\mu\left(x\right)=\sum_{i=1}^{M_{n}}\int_{A_{i}^{n}}f\left(y_{i}^{n}\right)\text{d}\mu\left(x\right)=\sum_{i=1}^{M_{n}}f\left(y_{i}^{n}\right)\mu\left(A_{i}^{n}\right)=\int_{Y}f\text{d}p_{n}
\end{equation}
By the same arguments made in the proof of Claim \ref{claim:(Discrete-aproximation)}, $\int_{Y}f\text{d}p_{n}\rightarrow \int_{Y}f\text{d}\nu\ \forall f\in C(Y,\f)$.
For any $x\in X$ and $f\in C(Y,\f)$, $f(T_nx)\rightarrow f(Tx)$ so by the dominated convergence theorem
$$
\int_Xf(T_nx)\d\mu(x)
=
\int_X\left\langle f(T_nx),\dm(x)\right\rangle\d|\mu|(x)
\rightarrow
\int_X\left\langle f(Tx),\dm(x)\right\rangle\d|\mu|(x)
=
\int_Xf(Tx)\d\mu(x)
$$
Therefore, taking the limit $n\rightarrow\infty$ in (\ref{Eq:Ex trans map}) leads to $\int_{X}f\left(Tx\right)\text{d}\mu\left(x\right)=\int_{Y}f\text{d}\nu$
so $T_{\#}\mu=\nu$.

\bigskip

On other hand, if there exists $T:X\rightarrow Y$ such that $T_\#\mu=\nu$ then for any sequence of partitions $\mathcal{B}^n=\{B_1^n,\ldots,B_{M_n}^n\}$ one can take the sequence of partitions $\mathcal{A}^n=\{T^{-1}(B_1^n),\ldots,T^{-1}(B_{M_n}^n)\}$ to conclude $\mu$ sequentially dominates $\nu$.
\end{proof}

The notion of sequentially dominance may be hard to verify. A simpler but stronger condition is the following:

\begin{defn}
$\ $
\begin{enumerate}
    \item For any $\mu\in\mm(X,S)$ and a measurable $A\subset X$ we define the \emph{restricted measure}\index{restricted measure} $\mu|_A\in\mm(A,S)$ by $\mu|_A(E):=\mu(E)\ \forall E\subset A$.
    
    \item We say $\mu\in\mm(X,S)$ \emph{strongly dominates} \index{strongly dominates} $\nu\in\mm(Y,S)$ if $\mu(X)=\nu(Y)$ and $\mu(A)=\nu(B)\Rightarrow\mu|_A\succ\nu|_B$ for any two measurable subsets $A\subset X,\ B\subset Y$.
    
\end{enumerate}
\end{defn}

\begin{rem}
By taking $A=X,\ B=Y$ in the definition above we conclude that if $\mu$ strongly dominates $\nu$ then $\mu\succ\nu$.
\end{rem}

\begin{rem}
In the scalar case $S=S^*=\f=\r$ strong domination is equivalent to $\mu(X)=\nu(Y)$,  since in this case $\mu(A)=\nu(B)$ is equivalent to $\mu|_A\succ\nu|_B$ (see Remark \ref{rem:RelationScalar}).
\end{rem}

A simple case where this strong domination occur is when $\nu$ has a 1-dimensional image:

\begin{claim}
Let $\mu\in\mm(X,S)$ and $|\nu|\in\mm_+(Y)$ such that $|\nu|(Y)=|\mu|(X)$, then $\mu$ strongly dominates $\nu\in\mm(Y,S)$ defined by $\nu(E):=\frac{|\nu|(E)}{|\mu|(X)}\mu(X)$
\end{claim}

\begin{proof}
For any two subsets $A\subset X,\ B\subset Y$ such that $\mu(A)=\nu(B)$, we define the kernel $P\in\text{Ker}(A,B)$ by $P_x(E)=\frac{|\nu|(E)}{|\nu|(B)}\ \forall x\in X,\ E\subset B$. For every $E\subset B$ we have
\begin{multline*}
P(\mu|_A)(E)
=
\int_AP_x(E)\d\mu(x)
=
\frac{|\nu|(E)}{|\nu|(B)}\mu(A)
=\\=
\frac{|\nu|(E)}{|\nu|(B)}\nu(B)
=
\frac{|\nu|(E)}{|\nu|(B)}\frac{|\nu|(B)}{|\mu|(X)}\mu(X)
=
\frac{|\nu|(E)}{|\mu|(X)}\mu(X)=\nu(E)=\nu|_B(E)
\end{multline*}

thus  $\mu|_A\succ\nu|_B$ and therefore $\mu$ strongly dominates $\nu$.
\end{proof}

\fbox{
\parbox{15cm}{
\begin{example}
\label{ex:strongDom}
Let $\left(\mu_1,\mu_2\right)\in\mm_+(X)^2\subset\mm(X,\r^2)$ and $\nu\in\mm_+(Y)$ with $\nu(Y)=\mu_1(X)=\mu_2(X)$ then $(\mu_1,\mu_2)$ strongly dominates $(\nu,\nu)$. To see that just take  $|\left(\mu_1,\mu_2\right)|=\frac{\mu_1+\mu_2}{2}$ and then $(\nu,\nu)=(1,1)\nu=\frac{\left(\mu_1,\mu_2\right)(X)}{|\left(\mu_1,\mu_2\right)|(X)}\nu$.
\end{example}
}
}

\bigskip

Strong domination is strictly stronger than the relation $\mu\succ\nu$, even in the discrete case, as we demonstrate in the example below:

\bigskip

\fbox{
\parbox{15cm}{
\begin{example}
\label{ex:notstrongDom}
Let $X=Y=\{a,b,c,d\},\ S=\r^2,\ \mu=\nu=((2,0,2,0),(1,2,0,1))$. Clearly $\mu\succ\nu$, but for $A=\{a,d\},\ B=\{b,c\}$ we have $\mu(A)=(2,2)=\nu(B)$, and by example \ref{ex:discretedominance} $$\mu|_A=((2,0),(1,1))\not\succ((0,2),(2,0))=\nu|_B$$

More examples may be formulated even in the continuous case in the following way: Choose $\mu\in\mm(X,S),\ \nu\in\mm(Y,S)$ such that $\mu(X)=\nu(Y)$ and $\mu\not\succ\nu$. Now choose some $\mu'\in\mm(X,S),\ \nu'\in\mm(Y,S)$ such that $\mu'\succ\nu,\ \mu\succ\nu'$. Let $\tilde{X}=X\dot{\cup}X',\ \tilde{Y}=Y\dot{\cup}Y'$ be the disjoint union of two copies of $X,Y$ respectively (and extend the metrics in some proper manner) and define $\tilde{\mu}\in\mm(\tilde{X},S)$ by
$$\tilde{\mu}(A\dot{\cup}A')=\mu(A)+\mu'(A')\ \forall A\subset X,\ A'\subset X'$$
and similarly $\tilde{\nu}\in\mm(\tilde{Y},S)$ by
$$\tilde{\nu}(B\dot{\cup}B')=\nu(B)+\nu'(B')\ \forall B\subset Y,\ B'\subset Y'$$
To show $\tilde{\mu}\succ\tilde{\nu}$ we use $\tilde{P}\in\text{ker}(\tilde{X},\tilde{Y})$ defined by
$$\tilde{P}_x(B\dot{\cup}B')=\begin{cases}
      P_x(B')&;x\in X
      \\
      P'_x(B)&;x\in X'
\end{cases}$$
where $P\in\text{ker}(X,Y'),\ P'\in\text{ker}(X',Y)$ satisfy $P\mu=\nu'$ and $P'\mu'=\nu$. But for $A=X\subset{\tilde{X}},\ B=Y\subset{\tilde{Y}}$ we have $\mu(X)=\nu(Y),\ \tilde{\mu}|_X=\mu\not\succ\nu=\tilde{\nu}|_Y$ thus $\tilde{\mu}$ does not strongly dominates $\tilde{\nu}$.
\end{example}
}
}

\b

\begin{claim}
\label{claim:strongly implies sequentially}
Assume $\mu\in\mm\left(X,S\right)$ and $\nu\in\mm\left(Y,S\right)$ both satisfying \ref{Ass:Lyap} with norm bounded $\dm\in C(X,(S,w(\f))),\ \dn\in C(Y,(S,w(\f)))$ satisfying \ref{ass:Vec}. If $\mu$ strongly dominates $\nu$ then $\mu$ sequentially dominates $\nu$.
\end{claim}
\begin{proof}
Assume $\mu$ strongly dominates $\nu$. Let $\left\{ \mathcal{B}^{n}:=\left\{B_{1}^{n},\ldots,B_{M_{n}}^{n}\right\}\right\} _{n=1}^{\infty}$
sequence of partitions of $Y$ such that $\mathcal{B}^{n+1}$ is a refinement of $\mathcal{B}^{n}$  for all $n\in\mathbb{N}$.
Since $\mu$ strongly dominates $\nu$, we may define the following
sequence of finite partitions of $X$ inductively:
\begin{itemize}
\item By Corollary \ref{cor:dominance partitions}, there exists $\mathcal{A}^{1}=\left\{ A_{1}^{1},\ldots,A_{M_{1}}^{1}\right\} $
a partition of $X$ such that $\mu(A_{i}^{1})=\nu(B_{i}^{1})\ \forall i=1,\ldots,M_{1}$.

\item Given a partition $\mathcal{A}^{n}=\left\{ A_{1}^{n},\ldots,A_{M_{n}}^{n}\right\}$ such that $\mu(A_{i}^{n})=\nu(B_{i}^{n})\ \forall i=1,\ldots,M_{n}$,
define its refinement $\mathcal{A}^{n+1}=\left\{ A_{1}^{n+1},\ldots,A_{M_{n+1}}^{n+1}\right\} $
in the following way: Let $i=1,\ldots,M_{n}$. Since $\mu|_{A_{i}^{n}}\succ\nu|_{B_{i}^{n}}$ (by strong domination)
and
$$\mathcal{B}^{n+1,i}:=\left\{ B_{j}^{n+1};\ j=1,\ldots,M_{n+1},\ B_{j}^{n+1}\subset B_{i}^{n}\right\}$$

is a partition of $B_{i}^{n}$ ($\mathcal{B}^{n+1}$ is a refinement
of $\mathcal{B}^{n}$), again by Corollary \ref{cor:dominance partitions} we get a partition of $A^n_i$:
$$\mathcal{A}^{n+1,i}:=\left\{ A_{j}^{n+1};\ j=1,\ldots,M_{n+1},\ B_{j}^{n+1}\subset B_{i}^{n}\right\}$$
 such that $\mu(A_j^{n+1})=\nu(B_j^{n+1})\ \forall j$.
We define $\mathcal{A}^{n+1}=\cup_{i=1}^{M_{n}}\mathcal{A}^{n+1,i}$ which is a disjoint union of the partitions of $A_i^n$ and therefore a partition of $X$ which is a refinement of $\mathcal{A}^n$.
\end{itemize}
The sequence of partitions $\{\mathcal{A}^{n}=\left\{ A_{1}^{n},\ldots,A_{M_{n}}^{n}\right\}\}_{n=1}^\infty $
satisfies:
\begin{itemize}
\item $\mathcal{A}^{n+1}$ is a refinement of $\mathcal{A}^{n}$.
\item $\mu\left(A_{i}^{n}\right)=\nu\left(B_{i}^{n}\right)$ for any $n\in\mathbb{N}$
and $i=1,\ldots,M_{n}$.
\item $B_{i}^{n+1}\subset B_{j}^{n}\Leftrightarrow A_{i}^{n+1}\subset A_{j}^{n}$
for any $n\in\mathbb{N}$, $j=1,\ldots,M_{n}$ and $i=1,\ldots,M_{n+1}$
\end{itemize}
Therefore $\mu$ sequentially dominates $\nu$.
\end{proof}

\begin{cor}

\label{thm:ExTransMap strong}
Assume
\begin{itemize}
    \item $\mu\in\mm(X,S),\ \nu\in\mm(Y,S)$ satisfy Assumption \ref{Ass:Lyap}.
    \item $\dm\in C(X,(S,w(\f))),\ \dn\in C(Y,(S,w(\f)))$ norm bounded and satisfies \ref{ass:Vec}.
    \item $Y$ is complete.
\end{itemize}
If $\mu$ strongly dominates then $\nu$ then $\mathcal{T}(\mu,\nu)\neq\emptyset$.

\end{cor}
\begin{proof}
Directly by Theorem \ref{thm:ExTransMap generalcase} and Claim \ref{claim:strongly implies sequentially}.
\end{proof}
By the above corollary we get the following corollary that generalizes the existence of transport maps in the scalar case:

\begin{cor}
Let $X,Y$ be two compact metric spaces, $Y$ complete, $\left(\mu_1,\mu_2\right)\in\mathcal{P}(X)^2$ be a pair of non-atomic measures with a continuous Radon-Nikodym derivative wrt to some $|\mu|\in\mm_+(X)$ and $\nu\in\mathcal{P}(Y)$. There exists $T:X\rightarrow Y$ such that $T_\#\mu_i=\nu$ for $i=1,2$.
\end{cor}

\begin{proof}
By Example \ref{ex:strongDom}, $\mu$ strongly dominates $(\nu,\nu)$, and any finite-dimensional non-atomic measure satisfy \ref{Ass:Lyap} (By Lyapunov's convexity theorem) so by Corollary \ref{thm:ExTransMap strong} we are done.
\end{proof}

\begin{rem}
In the corollary above with $\mu_1=\mu_2:=\mu$, we get the classical result about the existence of a transport map sending $\mu$ to $\nu$.
\end{rem}

\subsection{The multi-range of a vector measure}
The range of $\mu\in\mm(X,S)$\index{range of a vector measure} is defined by
$$\mathcal{R}(\mu):=\{\mu(A)\in S;\ A\subset X\}$$
As we already mentioned, Lyapunov's convexity theorem \index{Lyapunov's convexity theorem} (Theorem 13.33 in \cite{Aliprantis2006Hitch}), which in \cite{Diestel1977} is referred by the authors as "One of the most beautiful and best-loved theorems of the theory of vector measures", states that the range of a finite-dimensional non-atomic vector measure is a compact and convex set.

The multi-range of a vector measure\index{multi-range of a vector measure} is:
$$\mathcal{R}_n(\mu):=\{(\mu(A_1),\ldots,\mu(A_n))\in S^n;\ \{A_i\}_{i=1}^n\ is\ a\ partition\ of\ X\}$$

A direct result of Corollary \ref{cor:Relation discrete Partitions} is that under Assumptions \ref{ass:Vec} and \ref{Ass:Lyap} we have
\begin{multline}
\label{eq:lyap}
\mathcal{R}_n(\mu)
=
\{(s_1,\ldots,s_n)\in S^n=\mm(\{1,\ldots,n\},S);\ \mu\succ (s_1,\ldots,s_n)\}
=\\
\left\{\left(\int_Xg_1\d\mu,\ldots,\int_Xg_n\d\mu\right)\in S^n;\ 0\leq g_i\in L^{\infty}(|\mu|)^n\ \forall i,\ \sum_{i=1}^ng_i\equiv1\right\}
\end{multline}
and
$$\mu\succ_n\nu\iff\mathcal{R}_n(\mu)\supset\mathcal{R}_n(\nu)$$

A known result that generalizes Lyapunov's theorem is that the multi-range of a finite-dimensional non-atomic vector measure is convex and compact (Theorems 1 and 3 in \cite{dvoretzky1951relations}). This is indeed a generalization since $\mathcal{R}(\mu)$ is the projection of $\mathcal{R}_n(\mu)$ to one of the components.

\bigskip

\fbox{
\parbox{15cm}{
\begin{example}
Let $X=[0,1],\ Y=\{0,1\},\ S=\r^2$ and $\mu\in\mm(X,S)$ with $|\mu|$ is the Lebesgue measure on $[0,1]$ and $\frac{\d\mu}{\d|\mu|}=(1,2x)$. In Example \ref{ex:semidiscretedominance} we showed that
$$\mu\succ\nu=((a,1-a),(b,1-b))\iff a\in[0,1],\ b\in[a^2,2a-a^2]$$
The vector measure $\nu$ represented as an element in $S^2$ is 
$$(\nu(0),\nu(1))=((a,b),(1-a,1-b))$$
 So by (\ref{eq:lyap})
 $$\mathcal{R}_2(\mu):=\{((a,b),(1-a,1-b));\ a\in[0,1],\ b\in[a^2,2a-a^2]\}$$
 $$\mathcal{R}(\mu):=\{(a,b);\ a\in[0,1],\ b\in[a^2,2a-a^2]\}$$
\end{example}
}
}

\bigskip

 Interestingly, Lyapunov's theorem fails in the infinite-dimensional case, as shown in \cite{Diestel1977} Chapter IX.1: Examples 1 and 2 in that chapter construct two non-atomic vector measures with values in $L^1[0,1]$ and $\ell^2$ (respectively) that fail to satisfy Assumption \ref{Ass:Lyap} with $\alpha=\frac{1}{2}$ and therefore do not have a convex range. Moreover, in Example 1, the range is non-compact as well.

 Using Corollary \ref{cor:Relation discrete Partitions} we may generalize the results about multi-ranges (and ranges) of infinite-dimensional vector measures:

\begin{thm}
\label{thm:LyapInf}
Let $n\in\mathbb{N}$ and $\mu\in\mm(X,S)$ satisfying \ref{ass:Vec} and \ref{Ass:Lyap}. Then $\mathcal{R}_n(\mu)$ is convex and weak-compact.
\end{thm}
\begin{proof}
Let  $\mu\in\mm(X,S)$. By (\ref{eq:lyap}) and Proposition \ref{prop:Balyage Convex} $\mathcal{R}_n(\mu)$ is convex, we will show it is weak-compact: Define
$$
W=\left\{(g_1,\ldots,g_n)\in L^{\infty}(|\mu|)^n;\ g_i\geq0\forall i,\ \sum_{i=1}^ng_i\equiv1\right\}\subset (L^1(|\mu|)^n)^*
$$
$W$ is norm-bounded and weak*-closed so by the Banach-Alaoglu\index{Banach-Alaoglu theorem} theorem (Theorem 6.21 in \cite{Aliprantis2006Hitch}) it is weak*-compact.
Define $T:L^{\infty}(|\mu|)^n\rightarrow S^n$ by
$$
T(g_1,\ldots,g_n)=\left(\int_Xg_1\d\mu,\ldots,\int_Xg_n\d\mu\right)
$$
$T$ is continuous when $L^{\infty}(|\mu|)^n$ is equipped with the weak* topology generated by $L^1(|\mu|)^n$ and $S^n$ is equipped with weak topology (by part 6 of Proposition \ref{prop:Bochner}) therefore $\mathcal{R}_n(\mu)=TW$ is weak-compact.

\end{proof}

For similar results in the case $n=1$ see Corollary 7 and Theorem 10 in \cite{Diestel1977} Chapter IX.1.

\begin{rem}
In the norm topology, the multi-range of $\mu\in\mm(X,S)$ is relatively-compact whenever $\dm$ is norm bounded (even without Assumption \ref{ass:Vec} and \ref{Ass:Lyap}) since the map $T$ defined above is compact  (see Theorem 2 in \cite{Diestel1977} Chapter III.2) and therefore it sends the bounded set $$\left\{(\chi_{A_1},\ldots,\chi_{A_n});\ \{A_i\}_{i=1}^n\ is\ a\ partition\ of\ X\right\}\subset L^{\infty}(|\mu|)^n$$ to a relatively compact set which is exactly the multi-range of $\mu$.
\end{rem}

\newpage{}

\section{Applications to various transport problems}

In this chapter, I will demonstrate how the abstract duality theorem (Theorem
\ref{Thm:AbsDuality}) can be applied to other optimal transport problems.
Along this chapter we assume that $\left(X,\mu\right),\left(Y,\nu\right)$
are two compact metric scalar measure spaces and $c\in C\left(X\times Y\right)$
is a cost function.

\subsection{Optimal  partial transport}

When $\mu\left(X\right)\neq\nu\left(Y\right)$, surely one cannot transport
$\mu$ to $\nu$, but it is possible to transport part of $\mu$ to
$\nu$ (if $\mu(X)>\nu(Y)$) or $\mu$ to part of $\nu$ (if $\mu(X)<\nu(Y)$). More generally, in the "optimal partial transport" problem, which is discussed in \cite{Caf2010} and \cite{Figalli2010}, we wish to transport a total mass of $m$ from $\m$ to $\n$. Formally: 
given $\mu\in\mm_+\left(X\right),\nu\in\mm_+\left(Y\right)$
and $m\in\left[0,\min\{\mu\left(X\right),\nu\left(Y\right)\}\right]$,
minimize the total cost $\int_{X\x Y} c\d\pi$ over all $\pi\in\mm_+(X\x Y)$ satisfying
\[
\int_{X\x Y}\psi\left(x\right)+\varphi\left(y\right)+\lambda\d\pi\left(x,y\right)\leq\int_{X}\psi\text{\ensuremath{\left(x\right)}d}\mu\left(x\right)+\int_{Y}\varphi\left(y\right)\d\nu\left(y\right)+\lambda m
\]
for all $\left(\psi,\varphi,\lambda\right)\in C\left(X,\r_{+}\right)\x C\left(Y,\r_{+}\right)\x\r$. We denote this set of measures by $\Pi_m(\mu,\nu)$.
By the abstract duality theorem (Theorem \ref{Thm:AbsDuality}), the dual problem is
\[
\sup\left\{ \int_{X}\psi\text{\ensuremath{\left(x\right)}d}\mu\left(x\right)+\int_{Y}\varphi\left(y\right)\d\nu\left(y\right)+\lambda m;\ \left(\psi,\varphi,\lambda\right)\in C\left(X,\r_{+}\right)\x C\left(Y,\r_{+}\right)\x\r,\ \psi\left(x\right)+\varphi\left(y\right)+\lambda\leq c\left(x,y\right)\right\} 
\]
This result is achieved  by taking
\begin{itemize}
\item $V=C\left(X\times Y\right)$
\item $K=C\left(X,\r_{+}\right)\times C\left(Y,\r_{+}\right)\times\r$
\item $u=c\left(x,y\right)$
\item $\q=\left\{ f\in V;\ f\left(x,y\right)\geq0 \ \forall x,y\right\} $
\item $h\left(\psi,\varphi,\lambda\right)=\psi\left(x\right)+\varphi\left(y\right)+\lambda$
\item $k^{*}\left(\psi,\varphi,\lambda\right)=\int_{X}\psi\text{\ensuremath{\left(x\right)}d}\mu\left(x\right)+\int_{Y}\varphi\left(y\right)\d\nu\left(y\right)+\lambda m$
\item Assumption \ref{assu:Sublinear} holds with $s(v):=(\|v\|_{\infty},0,0)$.
\end{itemize}

A particular case: If $m=\mu(X)\leq\nu(Y)$, we transport the whole measure $\mu$ to a part of a measure $\nu$ with greater or equal mass. The constraint on $\pi$ in this case is equivalent to
\[
\int_{X\x Y}\psi\left(x\right)+\varphi\left(y\right)\d\pi\left(x,y\right)\leq\int_{X}\psi\text{\ensuremath{\left(x\right)}d}\mu\left(x\right)+\int_{Y}\varphi\left(y\right)\d\nu\left(y\right)
\]
for all $\left(\psi,\varphi\right)\in C\left(X\right)\x C\left(Y,\r_{+}\right)$. Namely $\pi$ has first marginal $\mu$ and second marginal smaller than $\nu$. In this case the problem is actually a minimum of classical transport problems
\[
\min\left\{\min_{\pi\in\Pi(\mu,\lambda)}\int_{X\x Y} c\d\pi;\ \lambda\in\mm_+(Y),\ \int_Y\varphi\d\lambda\leq\int_Y\varphi\d\n\ \forall\varphi\in C(Y,\r_+),\ \lambda(Y)=\mu(X)\right\}
\]

\subsection{Optimal transport with capacity constraints}

When trying to transport one measure to another, one can apply a constraint
on the capacity of the transport and demand that the transport plan
must be smaller than a given measure $\overline{\pi}$. This kind
of problem is discussed in \cite{Rachev2006} and \cite{Mccann}.
 The set of transport plans in this problem, that we denote by $\Pi\left(\mu,\nu,\overline{\pi}\right)$\index{pimnpi@$\Pi\left(\mu,\nu,\overline{\pi}\right)$},
is the set of measures $\pi\in\Pi(\m,\n)$
that are smaller than $\overline{\pi}$: For every $\psi\in C\left(X\right),\varphi\in C\left(Y\right),\xi\in C\left(X\times Y,\mathbb{R}_{+}\right)$
\[
\int_{X}\psi\left(x\right)\d\mu\left(x\right)+\int_{Y}\varphi\left(y\right)\d\nu\left(y\right)+\int_{X\times Y}\xi\left(x,y\right)\d\overline{\pi}\left(x,y\right)\geq\int_{X\times Y}\psi\left(x\right)+\varphi\left(y\right)+\xi\left(x,y\right)\d\pi\left(x,y\right)
\]

By the abstract duality theorem (Theorem \ref{Thm:AbsDuality}), the dual problem to $\max_{\pi\in\Pi(\mu,\nu,\overline{\pi})}\int_{X\x Y}c\d\pi$ is:

\begin{equation}
    \label{eq:CapDuality}
\inf_{\left(\psi,\varphi,\xi\right)\in\Phi\left(c\right)}\int_{X}\psi\left(x\right)\d\mu\left(x\right)+\int_{Y}\varphi\left(y\right)\d\nu\left(y\right)+\int_{X\times Y}\xi\left(x,y\right)\d\overline{\pi}\left(x,y\right)
\end{equation}
where 
\[
\Phi\left(c\right):=\left\{ \left(\psi,\varphi,\xi\right)\in C\left(X\right)\times C\left(Y\right)\times C\left(X\times Y,\mathbb{R}_{+}\right):\psi\left(x\right)+\varphi\left(y\right)+\xi\left(x,y\right)\geq c\left(x,y\right)\forall x,y\right\} 
\]
or equivalently
$$
\inf_{\left(\psi,\varphi\right)\in C\left(X\right)\times C\left(Y\right)}\int_{X}\psi\left(x\right)\d\mu\left(x\right)+\int_{Y}\varphi\left(y\right)\d\nu\left(y\right)+\int_{X\times Y}\left[c\left(x,y\right)-\psi\left(x\right)-\varphi\left(y\right)\right]_{+}\d\overline{\pi}\left(x,y\right)
$$

This result is achieved by  by taking:

\begin{itemize}
\item $V=C\left(X\times Y\right)$
\item $K=C\left(X\right)\times C\left(Y\right)\times C\left(X\times Y,\mathbb{R}_{+}\right)$
\item $u=c\left(x,y\right)$
\item $\q=\left\{ f\in V;\ f\left(x,y\right)\geq0\ \forall x,y\right\} $
\item $h\left(\psi,\varphi,\xi\right)=\psi+\varphi+\xi$
\item $k^{*}\left(\psi,\varphi,\xi\right)=\int_{X}\psi\left(x\right)\d\mu\left(x\right)+\int_{Y}\varphi\left(y\right)\d\nu\left(y\right)+\int_{X\times Y}\xi\left(x,y\right)\d\overline{\pi}\left(x,y\right)$
\item Assumption \ref{assu:Sublinear} holds with $s(v):=(\|v\|_{\infty},0,0)$.
\end{itemize}

A known sufficient and necessary condition for $\Pi\left(\mu,\nu,\overline{\pi}\right)\neq\emptyset$
can also be obtained by Lemma \ref{lem:EqCon} (By \cite{Mccann},
this condition is due to H.G Kellerer):
\begin{claim}
$\Pi\left(\mu,\nu,\overline{\pi}\right)\neq\emptyset$ iff for every
$\left(\psi,\varphi\right)\in C\left(X\right)\times C\left(Y\right)$
\[
\int_{X\times Y}\left[\psi\left(x\right)+\varphi\left(y\right)\right]_{+}\d\overline{\pi}\left(x,y\right)-\int_{X}\psi\left(x\right)\d\mu\left(x\right)-\int_{Y}\varphi\left(y\right)\d\nu\left(y\right)\geq0
\]
\end{claim}

\begin{proof}
By Lemma \ref{lem:EqCon}, $\Pi\left(\mu,\nu,\overline{\pi}\right)\neq\emptyset$
iff $k^{*}$ is $h$-positive i.e.
\[
\int_{X}\psi\left(x\right)\d\mu\left(x\right)+\int_{Y}\varphi\left(y\right)\d\nu\left(y\right)+\int_{X\times Y}\xi\left(x,y\right)\d\overline{\pi}\left(x,y\right)\geq0
\]
 for every $\left(\psi,\varphi,\xi\right)\in K$ such that $\psi\left(x\right)+\varphi\left(y\right)+\xi\left(x,y\right)\geq0$.
We can always consider $\xi\left(x,y\right)=\left[-\psi\left(x\right)-\varphi\left(y\right)\right]_{+}$ since
it is the smallest non-negative $\xi$ satisfying $\psi\left(x\right)+\varphi\left(y\right)+\xi\left(x,y\right)\geq0$
given some $\psi,\varphi$, hence this condition is equivalent to

\[
\int_{X}\psi\left(x\right)\d\mu\left(x\right)+\int_{Y}\varphi\left(y\right)\d\nu\left(y\right)+\int_{X\times Y}\left[-\psi\left(x\right)-\varphi\left(y\right)\right]_{+}\d\overline{\pi}\left(x,y\right)\geq0
\]
for every $\left(\psi,\varphi\right)\in C\left(X\right)\x C\left(Y\right)$
and by changing $\left(\psi,\varphi\right)\leftrightarrow\left(-\psi,-\varphi\right)$
we are done.
\end{proof}

\subsection{Invariant Measures}

Given a map $T:Y\rightarrow Y$, a $T$-invariant measure is some
$\nu\in\mathcal{P}\left(Y\right)$ such that $T_{\#}\nu=\nu$. One can
ask what is the $T$-invariant measure \textquotedbl closest\textquotedbl{}
to a given measure $\mu\in\mathcal{P}\left( Y\right)$. A more general problem
(see also \cite{Lopes2012}) is to minimize the total cost $\int_{X\x Y}c\d\pi$ over all measures $\pi$ with marginal $\mu$ on $X$ and a $T$-invariant marginal on $Y$: This is the set $\Pi(\mu,T)$ consists of all $\pi\in\mathcal{P}(X\x Y)$ satisfying
\[
\begin{array}{cccc}
\int_{X\x Y} \psi\left(x\right)\d\pi\left(x,y\right) & = & \int_X \psi\left(x\right)\d\mu\left(x\right) & \ \forall \psi\in C\left(X\right)\\
\int_{X\x Y} \varphi\left(y\right)\d\pi\left(x,y\right) & = & \int_{X\x Y} \varphi\left(Ty\right)\d\pi\left(x,y\right) & \ \forall \varphi\in C\left(Y\right)
\end{array}
\]

When $X=Y$ is a metric space and $c$  is the metric

$$\min_{\Pi(\m,T)}\int_{X\x X}c\d\pi=\min\left\{\min_{\Pi(\mu,\nu)}\int_{X\x X}c\d\pi;\ T_\#\nu=\nu\right\}$$
is the Wasserstein distance \index{Wasserstein metric} between $\m$ and the set of $T$-invariant measures, and the minimizer $\pi\in\Pi(\m,T)$ will have a second marginal which is a $T$-invariant measure closest to $\mu$.

Using the abstract theorem we can find the corresponding dual problem

\begin{thm}
Let $\mu\in\mathcal{P}\left(X\right)$ and a measurable $T:Y\rightarrow Y$.
Define 
\[
\Phi\left(c\right):=\left\{ \psi\in C\left(X\right);\ \exists\varphi\in C(Y), \psi\left(x\right)+\varphi\left(y\right)+\varphi\left(Ty\right)\leq c\left(x,y\right)\ \forall x,y\right\} 
\]
Then 
\[
\min_{\pi\in\Pi\left(\mu,T\right)}\int_{X\times Y}c\d\pi=\sup_{\psi\in\Phi\left(c\right)}\int_{X}\psi\d\mu
\]
\end{thm}

\begin{proof}
Use Theorem \ref{Thm:AbsDuality} with:

\begin{itemize}
\item $V=C\left(X\times Y\right)$
\item $K=C\left(X\right)\times C\left(Y\right)$
\item $u=c\left(x,y\right)$
\item $\q=\left\{ f\in V;\ f\left(x,y\right)\geq0\ \forall x,y\right\} $
\item $h\left(\psi,\varphi\right)=\psi+\varphi+\varphi\circ T$
\item $k^{*}\left(\psi,\varphi\right)=\int_{X}\psi\d\mu$
\item Assumption \ref{assu:Sublinear} holds with $s(v):=(\|v\|_{\infty},0)$.
\end{itemize}
\end{proof}

\subsection{Multi-marginal problems}

A known generalization of the optimal transport problem is to find a minimizing measure on the product
space $X=\prod_{i=1}^{n}X_{i}$ with given marginals $\mu_{1},\ldots,\mu_{n}$.
The duality theorem corresponding to this problem is

\begin{multline*}
\min\left\{\int_{X}c\left(x_1,\ldots,x_n\right)\d\pi\left(x_1,\ldots,x_n\right);\ \int_X\psi(x_i)\d\pi(x_1,\ldots,x_n)=\int_{X_i}\psi\d\mu_i\ \forall i\right\}
=\\
\sup\left\{\sum_{i=1}^{n}\int_{X_i}\psi_{i}\left(x_i\right)\d\mu_{i}\left(x_i\right);\ \sum_{i=1}^{n}\psi_{i}(x_i)\leq c(x_1,\ldots,x_n)\right\}
\end{multline*}
The proof of the duality above using Theorem \ref{Thm:AbsDuality} is similar to the proof of the Theorem \ref{thm:Duality} as described in Remark \ref{rem:KantDual} and Example \ref{exa:Kant}: We take
\begin{itemize}
    \item $V=C(X)$
    \item $K=\Pi_{i=1}^nX_i$
    \item $h(\psi_1,\ldots,\psi_n)=\sum_{i=1}^n\psi_i(x_i)$
    \item $k^*(\psi)=\sum_{i=1}^{n}\int_{X_i}\psi_{i}\left(x_i\right)\d\mu_{i}\left(x_i\right)$
\end{itemize}

A different multi-marginal problem for two given measures $\mu\in\mathcal{P}(X\x Y),\nu\in\mathcal{P}(Y\x Z)$ is \index{pigmn@$\Pi_G(\m,\n)$}
\begin{equation}
    \min_{\pi\in\Pi_G(\m,\n)}\int_{X\x Y\x Z}c(x,y,z)\d\pi(x,y,z)
    \label{eq:GlueTra}
\end{equation}
where $\Pi_G(\m,\n)$ is the set of $\pi\in\mathcal{P}(X\x Y\x Z)$ satisfying
$$\int_{X\x Y\x Z}\psi(x,y)+\varphi(y,z)\d\pi(x,y,z)=\int_{X\x Y}\psi(x,y)\d\mu(x,y)+\int_{Y\x Z}\varphi(y,z)\d\nu(y,z)$$ for all $\left(\psi,\varphi\right)\in C(X\x Y)\x C(Y\x Z)$.
An immediate necessary condition for the existence of such $\pi$ is that the marginals of $\mu$ and $\nu$ on $Y$ coincide. This condition is also known to be sufficient due to the "gluing lemma"\index{gluing lemma} (see for example \cite{Villani2016}). We can obtain the same result using Lemma \ref{lem:EqCon}: Take
\begin{itemize}
    \item $K= C(X\x Y)\x C(Y\x Z)$
    \item $k^*(\psi,\varphi)=\int_{X\x Y}\psi(x,y)\d\mu(x,y)+\int_{Y\x Z}\varphi(y,z)\d\nu(y,z)$
    \item $V=C(X\x Y\x Z)$
    \item $h(\psi,\varphi)=\psi+\varphi$
    \item $s(v)=(\|v\|_{\infty},0)$
\end{itemize}
It is not to hard to verify that $k^*$ is $h$-positive iff the marginals on $Y$ coincide: If $\psi(x,y)+\varphi(y,z)\geq0$, then also $\min_{x\in X}\psi(x,y)+\varphi(y,z)\geq0$. If we assume the marginals on $Y$ are equal then
\[
\begin{array}{ccc}
     \int_{X\x Y}\psi(x,y)\d\mu(x,y)+\int_{Y\x Z}\varphi(y,z)\d\nu(y,z) & = \\ \int_{X\x Y}\psi(x,y)-\min_{x\in X}\psi(x,y)\d\mu(x,y)+\int_{Y\x Z}\varphi(y,z)+\min_{x\in X}\psi(x,y)\d\nu(y,z) & \geq & 0
     
\end{array}
\]
Whereas if the marginals differ, there exists some $\psi(y)\in C(Y)$ such that
$\int_{X\x Y}\psi(y)\d\mu(x,y)+\int_{Y\x Z}-\psi(y)\d\nu(y,z)<0$.
By the abstract theorem, the dual problem to (\ref{eq:GlueTra}) is
$$\max\int_{X\x Y}\psi(x,y)\d\mu(x,y)+\int_{Y\x Z}\varphi(y,z)\d\nu(y,z)$$ where the maximum is over all $\left(\psi,\varphi\right)\in C(X\x Y)\x C(Y\x Z)$ satisfying $\psi(x,y)+\varphi(y,z)\leq c(x,y,z)$.

In addition to the constraint on $\pi$ to have given marginals on $X\x Y$ and $Y\x Z$, we may change the problem by fixing the marginal on $X\x Z$ as well, and minimize the total cost over all measures $\pi\in\mathcal{P}(X\x Y\x Z)$ such that
$$\int_{X\x Y\x Z}\psi(x,y)+\varphi(y,z)+\xi(x,z)\d\pi(x,y,z)=\int_{X\x Y}\psi(x,y)\d\mu(x,y)+\int_{Y\x Z}\varphi(y,z)\d\nu(y,z)+\int_{X\x Z}\xi(x,z)\d\lambda(x,z)$$
for all $\left(\psi,\varphi,\xi\right)\in C(X\x Y)\x C(Y\x Z)\x C(X\x Z)$. This problem is more complex: The condition that each pair in $\mu,\nu,\lambda$ has the same marginals on the joint space is no longer sufficient for the existence of such $\pi$, but just necessary, as mentioned in Remark 2.3 in \cite{gladkov2018multistochastic}. As in the previous problem, by Lemma \ref{lem:EqCon} the equivalent condition to the existence of such $\pi$ is that
$$\int_{X\x Y}\psi(x,y)\d\mu(x,y)+\int_{Y\x Z}\varphi(y,z)\d\nu(y,z)+\int_{X\x Z}\xi(x,z)\d\lambda(x,z)\geq0$$ whenever $\psi(x,y)+\varphi(y,z)+\xi(x,z)\geq0\ \forall x,y,z$.

\subsection{Local cost constraints}

Let $X,Y$ be two compact subsets of $\rn$ and $c\left(x,y\right)=\left|x-y\right|$.
Consider the set of transport plans sending $\mu\in\mathcal{P}\left(X\right)$
to $\nu\in\mathcal{P}\left(Y\right)$ with the additional constraint
that $x$ must be sent only to close enough $y$: Given some $D>0$,
$\left(x,y\right)\in\text{supp}\pi\Rightarrow c\left(x,y\right)\leq D$.
Equivalently, For any $0\leq \xi\in C\left(X\x Y\right)$
\[
\int_{X\x Y}\xi\left(c-D\right)\d\pi\leq0
\]

Overall, such transport plan must satisfy for any $\left(\psi,\varphi,\xi\right)\in C\left(X\right)\x C\left(Y\right)\x C_+\left(X\x Y\right)$
\[
\int_{X\x Y}\psi\left(x\right)+\varphi\left(y\right)+\xi\left(x,y\right)\left[c\left(x,y\right)-D\right]\d\pi\left(x,y\right)\leq\int_{X}\psi\d\mu+\int_{Y}\varphi\d\nu
\] 
By Lemma \ref{lem:EqCon},  there exists such a transport plan iff
for any $\left(\psi,\varphi,\xi\right)\in C\left(X\right)\x C\left(Y\right)\x C\left(X\x Y,\r_+\right)$

$$\psi\left(x\right)+\varphi\left(y\right)+\xi\left(x,y\right)\left[c\left(x,y\right)-D\right]\geq0\ \forall x,y\Rightarrow\int_{X}\psi\d\mu+\int_{Y}\varphi\d\nu\geq0$$

\subsection{Strassen theorem} 

In many transport problems, we minimize the total cost over all transport plans with marginals $\mu\in\mathcal{P}\left(X\right)$
and $\nu\in\mathcal{P}\left(Y\right)$ which belong to a given weak*-closed and convex
subset of measures $\Gamma\subset\mathcal{P}\left(X\x Y\right)$. If the elements of $\Gamma$ can be characterized using some linear constraint, the problem can be formulated using this constraint (for example see some of the problems described above).
Using a separation theorem (Theorem 5.79 in \cite{Aliprantis2006Hitch}) on
the weak*-compact convex subset $\{\pi\}$ and the weak*-closed convex subset $\Gamma$, we may conclude that $\pi\in\mathcal{P}\left(X\x Y\right)$
belongs to $\Gamma$ iff $\int f\d\pi\leq\sup_{\gamma\in\Gamma}\int f\d\gamma$
for any $f\in C\left(X\x Y\right)$. Therefore the transport plans in $\Pi(\m,\n)\cap\Gamma$ are
exactly all $\pi\in\mathcal{P}\left(X\x Y\right)$ such that 
\[
\int_{X\x Y}f\left(x,y\right)+\psi\left(x\right)+\varphi\left(y\right)\d\pi\left(x,y\right)\leq\sup_{\gamma\in\Gamma}\int f\d\gamma+\int_{X}\psi\d\mu+\int_{Y}\varphi\d\nu
\]
for all $\left(\psi,\varphi,f\right)\in C\left(X\right)\x C\left(Y\right)\x C\left(X\x Y\right)$.
By  Theorem \ref{Thm:AbsDuality}
\[
\max_{\pi\in\Pi\left(\mu,\nu\right)\cap\Gamma}\int c\d\pi=\min\left\{ \sup_{\gamma\in\Gamma}\int f\d\gamma+\int_{X}\psi\d\mu+\int_{Y}\varphi\d\nu;\ f\left(x,y\right)+\psi\left(x\right)+\varphi\left(y\right)\geq c\left(x,y\right)\right\} 
\]
Of course, $\Gamma$ must not contain any transport plan at all. The
question of the existence of such transport plan is known from a theorem
by Strassen (see Theorem 7 in \cite{strassen1965}) which states the
equivalent condition: $\Pi\left(\mu,\nu\right)\cap\Gamma\neq\emptyset$
iff
\[
\int_{X}\psi\d\mu+\int_{Y}\varphi\d\nu\leq\sup_{\gamma\in\Gamma}\int\left[\psi\left(x\right)+\varphi\left(y\right)\right]\d\gamma\ \forall(\psi,\varphi)\in C(X)\x C(Y)
\]
The same condition can be obtained directly by the $h$-positivity of $k^{*}$
in Lemma \ref{lem:EqCon}.   

\section{Some more applications}
 
In this chapter, we introduce some additional examples of applications of the abstract Theorems \ref{Thm:AbsDuality} and \ref{thm:Convex Abstract}. 

\subsection{Linear and non-Linear programming}

Let $n,k\in\mathbb{N}$,
$A$ a $k\x n$ real matrix, $\q$ a convex cone in $V:=\rn$, $K\subset\r^{k}$ another convex cone, $b\in \r^k,c\in \rn$. Denote the dual cones $\q^*=\{x\in V;\ x^T q\geq0\forall q\in\q\}$ and $K^*=\{z\in \r^k;\ z^T y\geq0\forall y\in K\}$.

\b

\fbox{
\parbox{15cm}{
\begin{example}
$ $
\begin{itemize}
\item If $K=\r^{k}$ then $K^*=\{0\}$.
\item If $K$ is the cone of vectors with non-negative components then $K^*=K$.
\item If $K=\left\{ 0\right\} $ then $K^*=\r^{k}$.
\item If $K=\left\{ \left(x,y,z,0\right)\in\r^{4};\ x\geq0,y\leq0\right\} $
then $K^*=\left\{ \left(x,y,0,z\right)\in\r^{4};\ x\geq0,y\leq0\right\}$.
\end{itemize}
\end{example}
}
}

\bigskip
$A$ can be considered as a linear operator $h:K\rightarrow V$
defined by $h\left(y\right)=y^{T}A$ and $b$ as a linear functional
$k^{*}:\r^{k}\rightarrow\r$ defined by $k^{*}\left(y\right)=y^{T}b$.
A linear functional on $V$ is just an element $x\in V$, hence the
inequality $\pi\left(h\left(y\right)\right)\leq k^{*}\left(y\right)\ \forall y\in K$
is interpreted as $y^{T}Ax\leq y^{T}b\ \ \forall y\in K$ or equivalently
$Ax-b\in K^*$. We denote
\[
\begin{array}{ccc}
\Pi\left(b,A\right) & := & \left\{ x\in\q^*;\ Ax-b\in K^*\right\} \\
\Phi\left(c,A\right) & := & \left\{ y\in K;\ y^{T}A-c^{T}\in\q\right\} 
\end{array}
\]
Since $V=\rn$, by Remark \ref{rem:AssFD} we may replace Assumption \ref{assu:Sublinear} with the assumption $\Phi(v,h)\neq\emptyset\ \forall v\in V$, hence Theorem \ref{Thm:AbsDuality} gives us a version of the linear programming duality:
\begin{thm}
\label{thm:LinProg}
If for every $x\in\rn$ there exists $y\in K$ such that $y^TA- x^T\in\q$ then
\[
\max_{x\in\Pi (b,A)}c^{T}x=\inf_{y\in\Phi (c,A)}y^{T}b
\]
and the duality is finite iff $y^TA\in\q\Rightarrow y^Tb\geq 0$
\end{thm}
Usually in linear programming the dual problem admits an optimal solution whenever the primal problem does, see for example Theorem 23 in \cite{karloff2008linear}. We do not have such a result in the theorem above.

\b

\fbox{
\parbox{15cm}{
\begin{example}
Let $K=\r^k$ and $\q\subset\rn$ is the cone of vectors with non-negative components. If $A$ has full row rank ($\text{rank}(A)=k$) then for any $x\in\rn$ the equation $A^Ty=x$ has a solution $y\in\r^k$ which means $y^TA-x^T=0\in\q$ so $\Phi(v,h)\neq\emptyset\ \forall v\in V$. The above theorem translates into: If $A$ has full row rank then
\[
\max_{x\geq 0;\ Ax=b}c^{T}x=\inf_{y\in\r^k;\ A^Ty\geq c}y^{T}b
\]
The equivalent condition for the duality to be finite in this example is
known as "Farkas' lemma"\index{Farkas' lemma} (see Theorem 29 in \cite{karloff2008linear}).

We can always assume without loss of generality $\text{rank}(A)=k$. Otherwise, either one of the equations in the system $Ax=b$ is redundant and may be removed or there is no solution to $Ax=b$.
\end{example}
}
}

\bigskip
The abstract theorem can also be used to obtain a non-linear programming duality theorem when $B:\r^{k}\rightarrow\r$
is sublinear (for example $B\left(y\right):=\left|\left|Cy\right|\right|$
for some $l\x k$ matrix $C$) and 
\[
\Pi(B,A)  =  \left\{ x\in\q^*;\ y^{T}Ax\leq B\left(y\right)\ \forall y\in K\right\} \\
\]
then under the same condition as in Theorem \ref{thm:LinProg}
\[
\max_{x\in\Pi(B,A)}c^{T}x=\inf_{y\in\Phi(c,A)}B\left(y\right)
\]

\subsection{Nash equilibrium of continuous games}

A zero-sum continuous game\index{continuous game} is a triplet $\left(X,Y,F\right)$
where $X,Y$ are compact metric spaces and $F:X\x Y\rightarrow\r$ a measurable bounded function.
$X,Y$ are considered as the (pure) strategy sets of the game and $F$ as the utility. Alice
and Bob can play the game $\left(X,Y,F\right)$ in the following
way: Alice chooses $x\in X$ while Bob chooses  $y\in Y$ (both not
knowing their opponent's choice), then Bob pays Alice an amount of
$F\left(x,y\right)$ (may be negative) and the game ends. The players
can also use \textquotedbl mixed strategies\textquotedbl , which
are Borel probability measures on $X,Y$. The pure strategies of choosing
a single element $x\in X,y\in Y$ are identified with the delta measures
$\delta_{x},\delta_{y}$ respectively. If Alice chooses some mixed
strategy $\mu\in\mathcal{P}\left(X\right)$, then her expected payoff
will be $\int_{X}F\left(x,y\right)\d\mu\left(x\right)$ whenever
Bob chooses $y\in Y$. In the worst case, she will be payed $\min_{y\in Y}\int_{X}F\left(x,y\right)\d\mu\left(x\right)$.
To ensure herself a maximal payoff (in expectation), she will need to solve the following:

$$
\underline{V}
=
\sup_{\m\in\mathcal{P}(X)}\inf_{y\in Y}\int_{X}F\left(x,y\right)\d\mu\left(x\right)\
=
\sup_{\m\in\mathcal{P}(X)}\inf_{\nu\in \mathcal{P}(Y)}\int_{Y}\int_XF(x,y)\d\mu(y)\d\nu(x)
$$

Similarly, Bob will have to solve the problem

$$
\overline{V}
=
\inf_{\n\in\mathcal{P}(Y)}\sup_{x\in X}\int_{X}F\left(x,y\right)\d\nu\left(x\right)
=
\inf_{\n\in\mathcal{P}(Y)}\sup_{\mu\in \mathcal{P}(X)}\int_{X}\int_YF(x,y)\d\nu(x)\d\mu(y)
$$

Using the abstract duality theorem we will show the two problems are
dual to each other and conclude the existence of a Nash equilibrium.
\begin{thm}
\label{thm:ZSG Cont}
Let $\left(X,Y,F\right)$ be a zero-sum continuous game with $F:X\x Y\rightarrow\r$ be a measurable bounded function such that $F(\cdot,y)\in C(X)\ \forall y\in Y$ then
$\underline{V}=\overline{V}$. 
\end{thm}

\begin{proof}
We notice that:

$$
\underline{V}
 = 
 \sup\left\{ \alpha\in\r;\ \exists\mu\in\mathcal{P}\left(X\right),\ \alpha\leq\int_{X}\int_{Y}F\left(x,y\right)\d\nu\left(x\right)\d\mu\left(y\right)\ \forall\nu\in\mathcal{P}\left(Y\right)\right\}
$$

$$
\overline{V}
=
\inf\left\{ \beta\in\r;\ \exists\nu\in\mathcal{P}\left(Y\right),\ \beta\geq\int_{X}\int_{Y}F\left(x,y\right)\d\nu\left(y\right)\d\mu\left(x\right)\ \forall\mu\in\mathcal{P}\left(X\right)\right\}
$$
We may assume without loss of generality $F\geq0$ (otherwise take $\tilde{F}=F-\min F\geq0$ to conclude $\underline{V}-\min F=\overline{V}-\min F$ and therefore $\underline{V}=\overline{V}$).
Use Theorem \ref{Thm:AbsDuality} with:
\begin{itemize}
\item $V=\r\x C\left(X\right)$
\item $K=\r\x\mm_+\left(Y\right)$
\item $h\left(b,\nu\right)=\left(\n(Y),b-\int_{Y}F\left(\cdot,y\right)\d\nu\left(y\right)\right)$ (where $\int_{Y}F\left(\cdot,y\right)\d\nu\left(y\right)$ is continuous by the dominated convergence theorem).
\item $k^{*}\left(b,\nu\right)=b$
\item $u=\left(1,0\right)$
\item $\q=\r_+\x C\left(X,\r_+\right)$
\item Let $y_0\in Y$. Assumption \ref{assu:Sublinear} holds with $s(\beta,f)=(\max_{x\in X}\{|\beta|F(x,y_0)+f(x )\},|\beta|\delta_{y_0})$ since
\[
\begin{array}{cc}
     h\circ s(\beta,f)=(|\beta|,\max_{x\in X}\{|\beta|F(x,y_0)+f(x)\}-|\beta|F(\cdot,y_0))\geq (\beta,f)
      \\
     k^*\circ s(\beta,f)=\max_{x\in X}|\beta|F(x,y_0)+f(x)\underset{(\beta,f)\rightarrow0}{\longrightarrow} 0
\end{array}
\]
\item  $k^*$ is $h$-positive since if $h(b,\nu)\geq0$ then $b-\int_YF(x,y)\d\nu(y)\geq0\ \forall x\Rightarrow k^*(b,\n)=b\geq0$ (Here we used the assumption $F\geq 0$).
\end{itemize}
For any $v^{*}=\left(\alpha,\mu\right)\in V^{*}=\r\x\mm\left(X\right)$
the condition $v^{*}\left(h\left(k\right)\right)\leq k^{*}\left(k\right)$
for all $k=\left(b,\nu\right)\in K$ translates to
\[
\alpha\n(Y)+\int_{X}\left[b-\int_{Y}F\left(x,y\right)\d\nu\left(y\right)\right]\d\mu\left(x\right)=\alpha\n(Y)-\int_{X}\int_{Y}F\left(x,y\right)\d\nu\left(y\right)\d\mu\left(x\right)+\mu\left(X\right)b\leq b
\]
for all $b\in\r,\n\in\mm_+(Y)$ which is equivalent to $\mu\left(X\right)=1$ and $\alpha\leq\int_{X}\int_{Y}F\left(x,y\right)\d\nu\left(y\right)\d\mu\left(x\right)\ \forall\n\in\mathcal{P}(Y)$. The condition $v^*(q)\geq 0\ \forall q\in\q$ translates to $\m\in\mm_+(X),\alpha\geq0$
hence 
\[
\max_{v^{*}\in\Pi\left(k^{*},h\right)}v^{*}\left(c\right)=\max\left\{ \alpha\geq0;\ \exists\mu\in\mathcal{P}\left(X\right)\ \alpha\leq\int_{X}\int_{Y}F\left(x,y\right)\d\nu\left(y\right)\d\mu\left(x\right)\ \forall\nu\in\mathcal{P}\left(Y\right)\right\} =\underline{V}
\]
(the restriction $\alpha\geq0$ should not bother us since $F\geq 0$).
For any $k=\left(b,\nu\right)\in K$ the condition $h\left(k\right)\geq c$
translates to $\left(\n(Y),b-\int_{Y}F\left(x,y\right)\d\nu\left(y\right)\right)\geq\left(1,0\right)$
which is equivalent to $\n(Y)\geq1$ and $b\geq \int_{Y}F\left(x,y\right)\d\nu\left(y\right)\ \forall x\in X$
therefore
\[
\begin{array}{ccccc}
\inf_{k\in\Phi\left(c,h\right)}k^{*}\left(k\right)
& = &
\inf\left\{ b\in\r;\ \exists\nu\in\mathcal{P}\left(Y\right)\ b\geq\int_X\int_{Y}F\left(x,y\right)\d\nu\left(y\right)\d\mu(x)\ \forall \mu\in \mathcal{P}(X)\right\}  & = & \overline{V}
\end{array}
\]
Therefore by Theorem \ref{Thm:AbsDuality} $\underline{V}=\overline{V}$.
\end{proof}

Using the weak* compactness of $\mathcal{P}(X),\mathcal{P}(Y)$ (or by using the symmetry between the two problems), it is possible to show the supremum and infimum in $\underline{V},\overline{V}$ are attained. The maximizer and minimizer measures form a saddle-point:

\begin{claim}
The maximizer $\mu\in\mathcal{P}\left(X\right)$ and the minimizer
$\nu\in\mathcal{P}\left(Y\right)$ of $\underline{V},\overline{V}$
respectively form a saddle point, namely: For any other $\mu'\in\mathcal{P}\left(X\right)$,$\nu'\in\mathcal{P}\left(Y\right)$
\[
\int_{Y}\int_{X}F\left(x,y\right)\d\mu'\left(x\right)\d\nu\left(y\right)\leq\int_{Y}\int_{X}F\left(x,y\right)\d\mu\left(x\right)\d\nu\left(y\right)\leq\int_{Y}\int_{X}F\left(x,y\right)\d\mu\left(x\right)\d\nu'\left(y\right)
\]
\end{claim}

\begin{proof}
Let $\mu\in\mathcal{P}\left(X\right)$ be the maximizer
of $\underline{V}$ then
$$\underline{V}
=
\inf_{\nu'\in\mathcal{P}(Y)}\int_{Y}\int_{X}F\left(x,y\right)\d\mu\left(x\right)\d\nu'\left(y\right)
\leq
\int_{Y}\int_{X}F\left(x,y\right)\d\mu\left(x\right)\d\nu'\left(y\right)\ \forall\nu'\in\mathcal{P}\left(Y\right)$$

Let $\nu\in\mathcal{P}\left(X\right)$ be the minimizer
of $\overline{V}$ then $$\overline{V}
=
\sup_{\mu'\in\mathcal{P}(X)}\int_{Y}\int_{X}F\left(x,y\right)\d\mu'\left(x\right)\d\nu\left(y\right)
\geq
\int_{Y}\int_{X}F\left(x,y\right)\d\mu'\left(x\right)\d\nu\left(y\right)\ \forall\mu'\in\mathcal{P}\left(X\right)$$

Therefore $\forall\mu'\in\mathcal{P}\left(X\right),\ \nu'\in\mathcal{P}\left(Y\right)$
\[
\int_{Y}\int_{X}F\left(x,y\right)\d\mu'\left(x\right)\d\nu\left(y\right)\leq\overline{V}=\underline{V}\leq\int_{Y}\int_{X}F\left(x,y\right)\d\mu\left(x\right)\d\nu'\left(y\right)
\]
 and for $\mu'=\mu,\nu'=\nu$ we get equality above hence $\forall\mu'\in\mathcal{P}\left(X\right),\forall\nu'\in\mathcal{P}\left(Y\right)$
\[
\int_{Y}\int_{X}F\left(x,y\right)\d\mu'\left(x\right)\d\nu\left(y\right)\leq\int_{Y}\int_{X}F\left(x,y\right)\d\mu\left(x\right)\d\nu\left(y\right)\leq\int_{Y}\int_{X}F\left(x,y\right)\d\mu\left(x\right)\d\nu'\left(y\right)
\]
\end{proof}

For an arbitrary bounded measurable $F$, it is possible that $\overline{V}\neq\underline{V}$, as demonstrated in \cite{Parthasarathy}. However, if we restrict one of the players to measures which are absolutely continuous wrt some given measure $\lambda$ we may prove the following minimax theorem:
\begin{thm}
\label{thm:ZSG Meas}
Let $\lambda\in\mm_+(Y)$ and $F:X\x Y\rightarrow\r$ a bounded measurable function, then
$$\sup_{\mu\in\mathcal{P}(X)}\inf_{\lambda>>\nu\in\mathcal{P}(Y)}\int_{X}\int_YF(x,y)\d\nu(x)\d\mu(y)
=
\inf_{\lambda>>\nu\in\mathcal{P}(Y)}\sup_{\mu\in\mathcal{P}(X)}\int_{X}\int_YF(x,y)\d\nu(x)\d\mu(y)$$
\end{thm}
The proof of the theorem above is similar to the proof of Theorem \ref{thm:ZSG Cont} with the following slight changes:
\begin{itemize}
    \item $K=\r\x\left\{\nu\in\mm_+(Y);\ \nu<<\lambda\right\}$
    \item  Assumption \ref{assu:Sublinear} holds with $s(\beta,f)=\left(\max_{x\in X}\frac{|\beta|}{\lambda(Y)}\int_YF(x,y)\d\lambda(y)+f(x),\frac{|\beta|}{\lambda(Y)}\lambda\right)$ since
\end{itemize}    
\[
\begin{array}{cc}
     h\circ s(\beta,f)=\left(|\beta|,\max_{x\in X}\left\{\frac{|\beta|}{\lambda(Y)}\int_YF(x,y)\d\lambda(y)+f(x)\right\}-\frac{|\beta|}{\lambda(Y)}\int_YF(x,y)\d\lambda(y)\right)\geq (\beta,f)
      \\
     k^*\circ s(\beta,f)=\max_{x\in X}\frac{|\beta|}{\lambda(Y)}\int_YF(x,y)\d\lambda(y)+f(x)\underset{(\beta,f)\rightarrow0}{\longrightarrow} 0
\end{array}
\]

Theorem \ref{thm:ZSG Meas} generalizes Parthasarathy's theorem (see \cite{Parthasarathy}) in which $X=Y=[0,1]\subset\r$,  $\lambda$ is the Lebesgue measure, and all discontinuity points of $F$ lie on a finite number of curves in $X\x Y$. In Parthasarathy's theorem, one of the players in restricted to use only the mixed strategies which are absolutely continuous wrt the Lebesgue measure, so this player cannot use any pure strategy. In Theorem \ref{thm:ZSG Meas}, we may take for example $\lambda=\text{Leb}+\sum_{i=1}^n\delta_{x_i}$ (where $\text{Leb}$ is the Lebesgue measure and $\{x_i\}_{i=1}^n\subset X$) thus allowing the player use the absolutely continuous strategies and the pure strategies $\delta_{x_i},\ i=1,\ldots,n$.

\subsection{Moment problems}

Given a compact metric space $X$, $M\in C\left(X\right)^{n}$
and $m\in\rn$, a known problem is whether there exists some non-negative
measure $\mu\in\mm_{+}\left(X\right)$ such that
\begin{equation}
    \label{eq:MomentGeneral}
    \int_{X}M_{i}\d\mu=m_{i};\ i=1,\ldots,n
\end{equation}
We can use Lemma \ref{lem:EqCon} as follows: Take
\begin{itemize}
\item $V=C(X),\ V^*=\mm(X)$
\item $\q =$ non-negative functions in $V$, $K=\rn$
\item $h(\alpha)=\alpha\cdot M(x)$
\item $k^*(\alpha)=m\cdot\alpha$
\end{itemize}
If $M_1\equiv 1$ then one can easily verify Assumption \ref{assu:Sublinear}  by taking $s(v):=(||v||_{\infty},0,\ldots,0)$. Therefore by Lemma \ref{lem:EqCon} we can obtain a necessary and
sufficient condition for the existence of $\mu\in\mm_+ (X)$ s.t
$$\int_X h(\alpha)\d\mu=k^*(\alpha)\ \forall\alpha\in\rn$$
which is equivalent to (\ref{eq:MomentGeneral}). Such $\mu\in\mm_{+}\left(X\right)$ exists iff $k^*$ is $h$-positive, meaning that
for any $\alpha\in\rn$
\[
\sum_{i=1}^{n}\alpha_{i}M_{i}\left(x\right)\geq0\ \forall x\in X\Rightarrow\sum_{i=1}^{n}\alpha_{i}m_{i}\geq0
\]
 The convex cone alternative (see Corollary 5.84 in \cite{Aliprantis2006Hitch}), implies that this condition is equivalent to $m\in\overline{\text{cone}}\left(\left\{ M\left(x\right)\in\rn;\ x\in X\right\} \right)$, where $\overline{\text{cone}}(C)$\index{cone@$\overline{\text{cone}}$ - generated convex cone} is the smallest closed convex cone containing a certain set $C$.
Similar problems are discussed in \cite{akhiezer1962some} and \cite{schmudgen2017moment}. We give two examples:

\begin{example}
Probability measure with given expectation and variance.

For which $m_2\in\r\ ,m_3\geq0$ there exists $\mu\in\mathcal{P}\left(\r\right)$
with $\int_{-\infty}^{\infty}x\d\mu\left(x\right)=m_{2},\ \int_{-\infty}^{\infty}x^{2}\d\mu\left(x\right)=m_{3}$?
A necessary condition is $m_2^2\leq m_3$ since by Jensen's inequality
\[
m_{2}^{2}=\left(\int_{-\infty}^{\infty}x\d\mu\left(x\right)\right)^{2}\leq\int_{-\infty}^{\infty}x^{2}\d\mu\left(x\right)=m_{3}
\]
But we can show it is also sufficient: Take $\left(M_{1}\left(x\right),M_{2}\left(x\right),M_{3}\left(x\right)\right)=\left(1,x,x^{2}\right)$
then
\[
\overline{\text{cone}}\left(\left\{ M\left(x\right);\ x\in\r\right\} \right)=\left\{ \lambda\left(1,\alpha,\beta\right)\in\r^{3};\ \alpha^{2}\leq\beta,\ \lambda\geq0\right\} 
\]
Therefore $(1,m_2,m_3)\in \overline{\text{cone}}\left(\left\{ M\left(x\right);\ x\in\r\right\} \right)$ iff $m_2^2\leq m_3$.

\end{example}

\begin{example}
Constructing a measure from its Fourier coefficients (also known as
the trigonometric moment problem).

Let $P\subset\mathbb{Z}$ be a finite subset and $c\in\mathbb{C}^{P}$. Is
there $\mu\in\mm_{+}\left(\left[0,2\pi\right]\right)$ for which $c_{k}=\frac{1}{2\pi}\int_{0}^{2\pi}e^{-ikx}\d\mu\left(x\right)$
for all $k\in P$? Take $M_{k}\left(x\right)=e^{-ikx}$ for all $k\in P$
then
\[
\overline{\text{cone}}\left(\left\{ M\left(x\right);\ x\in[0,2\pi]\right\} \right)=\overline{\text{cone}}\left(\left\{ (e^{-ikx})_{k\in P};\ x\in\left[0,2\pi\right]\right\} \right)=\overline{\text{cone}}\left(\left\{ (z^{k})_{k\in P};\ z\in\mathbb{C},\ \left|z\right|=1\right\} \right)
\]
So there exists such a measure iff $c\in\overline{\text{cone}}\left(\left\{ (z^{k})_{k\in P};\ z\in\mathbb{C},\ \left|z\right|=1\right\} \right) $.
In the case $P=\left\{ 0,1,\ldots,n\right\} $, we define
\[
C=\left(\begin{array}{cccc}
c_{0} & c_{1} & \cdots & c_{n}\\
\overline{c_{1}} & c_{0} & \cdots & c_{n-1}\\
\vdots & \vdots & \ddots & \vdots\\
\overline{c_{n}} & \overline{c_{n-1}} & \cdots & c_{0}
\end{array}\right)
\]
Such a matrix is called "Hermitian Toeplitz\index{Toeplitz} matrix".
Moreover we define
\[
A_n(z)=\left(\begin{array}{cccc}
1 & z & \cdots & z^{n}\\
\overline{z} & 1 & \cdots & z^{n-1}\\
\vdots & \vdots & \ddots & \vdots\\
\overline{z^{n}} & \overline{z^{n-1}} & \cdots & 1
\end{array}\right)
\]
The condition $c\in\overline{\text{cone}}\left(\left\{ (z^{k})_{k=0}^n;\ z\in\mathbb{C},\ \left|z\right|=1\right\} \right) $ is equivalent to $C\in A_n:=\overline{\text{cone}}(\{A_n(z);\ z\in\mathbb{C},\left|z\right|=1\})$. $A_n(z)$ is positive semi-definite whenever $|z|=1$ since it has an eigenvalue $0$ with $n$ eigenvectors:
\[
\left( \begin{array}{c}
-z\\1\\0\\0\\ \vdots \\0
\end{array} \right)
\left( \begin{array}{c}
-z^2\\0\\1\\0\\\vdots\\0
\end{array} \right)
\ \cdots\ 
\left( \begin{array}{c}
-z^n\\0\\0\\\vdots\\0\\1
\end{array} \right)
\]
and an eigenvalue $n+1$ with an eigenvector
\[
\left( \begin{array}{c}
z^n\\z^{n-1}\\\vdots\\z\\1
\end{array} \right)
\]
Since $A_n$ consists only of positive semi-definite matrices, we conclude that the existence of a measure with Fourier coefficients $c$ implies that $C$ is positive semi-definite. As it turns out, this condition is also sufficient: If $C$ is positive semi-definite, there exists such $\mu\in\mm_{+}\left(\left[0,2\pi\right]\right)$ (see Theorem 11.5 in \cite{schmudgen2017moment}), so we may conclude:
\begin{cor}
The $(n+1)\x (n+1)$ Hermitian Toeplitz positive semi-definite matrices are exactly the elements of $A_n$.
\end{cor}
\begin{proof}
Let $C$ be a $(n+1)\x (n+1)$ Hermitian Toeplitz positive semi-definite matrix. By Theorem 11.5 in \cite{schmudgen2017moment} there exists $\mu\in\mm_+([0,2\pi])$ satisfying $c_k=\int_0^{2\pi}e^{-ikx}\d\mu(x)\ \forall i=0,\ldots,n$. By Lemma \ref{lem:EqCon} and the convex cone alternative (Corollary 5.84 in \cite{Aliprantis2006Hitch})  $C\in A_n$.
\end{proof}
\end{example}

\begin{rem}
It is also possible to add marginal constraints to moment problems, see for example
Theorem 4.6.12 in \cite{Rachev2006}.
\end{rem}

\subsection{The Fenchel-Rockafellar duality
theorem}
As an application of Theorem \ref{thm:Convex Abstract}, we will prove
the known \index{Fenchel-Rockafellar}Fenchel-Rockafellar duality
theorem. We will need the following definition:
\begin{defn}
Let $X$ be a tvs. The \emph{convex conjugate}\index{convex conjugate} of $f:X\rightarrow\r\cup\{\pm\infty\}$ is $f^*:X^*\rightarrow\r\cup\{\pm\infty\}$ defined by\index{1@$^*$}
$$
f^*(x^*):=\sup_{x\in X}\left\langle x^*,x\right\rangle-f(x)
$$
\end{defn}
\begin{thm}
\label{thm:FRduality}(Theorem 1.9 in \cite{Villani2003}) Let $E$
be a tvs, $\Theta,\Xi:E\rightarrow\r$
two convex functions and $\Theta$ upper semi-continuous at some point,
then
\[
\max_{x^{*}\in E^{*}}\left\{ -\Theta^{*}\left(-x^*\right)-\Xi^{*}\left(x^{*}\right)\right\} =\inf_{x\in E}\left\{ \Theta\left(x\right)+\Xi\left(x\right)\right\} 
\]
\end{thm}

The above theorem is often used to prove the Kantorovich duality (as
in \cite{Villani2003}) and other duality theorems and its proof
is based on the Theorem \ref{thm:HB Separation}, similarly to the proof of theorem
\ref{thm:Convex Abstract}. We prove Theorem \ref{thm:FRduality}
using Theorem \ref{thm:Convex Abstract}:
\begin{proof}
Assume $\Theta$ is upper semi-continuous at $x_{0}\in E$ and let
\begin{itemize}
\item $V=E$
\item $K=E\times E$
\item $h\left(x,y\right)=y-x$
\item $\q=\left\{ 0\right\} $
\item $k^{*}\left(x,y\right)=\Theta\left(x\right)+\Xi\left(y\right)-\Theta(x_0)-\Xi(x_0)$
\end{itemize}
Define $s:V\rightarrow K$
to be $s\left(v\right):=\left(x_{0}-v,x_{0}\right)\in\Phi\left(v,h\right)$,
then $h\circ s(v)-v=0\in\q$ and by upper semi-continuity of $\Theta$ at $x_0$,
$
k^{*}\circ s\left(v\right)=\Theta\left(x_{0}-v\right)-\Theta(x_0)
$
 is bounded above on a neighborhood of $0$, thus Assumption \ref{assu:convex} holds. We notice that
\[
\inf_{\left(x,y\right)\in h^{-1}(\q)}k^{*}\left(x,y\right)=\inf_{y-x=0}\Theta\left(x\right)+\Xi\left(y\right)-\Theta(x_0)-\Xi(x_0)=\inf_{x\in E}\Theta\left(x\right)+\Xi\left(x\right)-\Theta(x_0)-\Xi(x_0)
\]
 and 
\[
\begin{array}{ccc}
\max_{x^{*}\in E^{*}}\inf_{x,y\in E}\left[-\left\langle x^{*},h\left(x,y\right)\right\rangle +k^{*}\left(x,y\right)\right]
& = &
\max_{x^{*}\in E^{*}}\inf_{x,y\in E}\left[\Theta\left(x\right)+\Xi\left(y\right)-\Theta(x_0)-\Xi(x_0)-\left\langle x^{*},y-x\right\rangle \right]\end{array}
\]
hence by Theorem \ref{thm:Convex Abstract}
\[
\max_{x^{*}\in E^{*}}\inf_{x,y\in E}\left[\Theta\left(x\right)+\Xi\left(y\right)+\left\langle x^{*},x-y\right\rangle \right]=\inf_{x\in E}\Theta\left(x\right)+\Xi\left(x\right)
\]
which by the definition of the convex conjugate is equivalent to
\[
\max_{x^{*}\in E^{*}}\left\{ -\Theta^{*}\left(-x^{*}\right)-\Xi^{*}\left(x^{*}\right)\right\} =\inf_{x\in E}\left\{ \Theta\left(x\right)+\Xi\left(x\right)\right\} 
\]
\end{proof}

\begin{rem}
The Fenchel-Rockafellar theorem in \cite{Villani2003} is proved for $\Theta,\Xi$ which can have infinite values (as opposed to only finite values in the theorem above) and for $E$ which is a normed space (as opposed to a tvs in the theorem above).
\end{rem}

Similarly, one can prove a generalized version of the Fenchel-Rockafellar
duality theorem:
\begin{thm}
\label{thm:GPD}Let $E$ be a tvs, $\Theta_{i}:E\rightarrow\r ,\ i=1,\ldots,n$
convex functions where $\Theta_{1}$ is upper semi-continuous at some
point then 
\[
\left(\sum_{i=1}^{n}\Theta_{i}^{*}\right)^{*}\left(x\right)=\inf\left\{ \sum_{i=1}^{n}\Theta_{i}\left(x_{i}\right);\ \sum_{i=1}^{n}x_{i}=x\right\} 
\]
\end{thm}

\begin{rem}
The theorem above implies the Fenchel-Rockafellar theorem by taking $n=2$ and $x=0$.
\end{rem}

\begin{rem}
The RHS of the equality above is sometimes called the "infimal convolution":\index{infimal convolution}
\[
\Theta_{1}\square\cdots\square\Theta_{n}\left(x\right):=\inf\left\{ \sum_{i=1}^{n}\Theta_{i}\left(x_{i}\right);\ \sum_{i=1}^{n}x_{i}=x\right\} 
\]
Theorem \ref{thm:GPD} implies (by taking the conjugate of both sides)
the \textquotedbl infimal convolution formula\textquotedbl :
\[
\left(\Theta_{1}\square\cdots\square\Theta_{n}\right)^{*}=\sum_{i=1}^{n}\Theta_{i}^{*}
\]
which is known to hold for any $n$-tuple of functions $\Theta_{1},\ldots,\Theta_{n}$
(not necessarily under the assumptions of Theorem \ref{thm:GPD}) when $E$ is locally convex and Hausdorff
(see Proposition 2.3.8 part b in \cite{Bot2009}). 
\end{rem}

\section{Chain transport}
 Let $X$ be a compact metric space and $c:X\x X\rightarrow\r$  a function representing the cost of moving mass from one point directly to another. When moving mass from some $x\in X$ to $y\in X$, it is possible that there exists some $z\in X$ such that $c(x,z)+c(z,y)<c(x,y)$ (of course such $c$ cannot be a metric). So to minimize the cost, one prefers to move mass from $x$ to $y$ through $z$ and not directly. This idea leads to the definition of the \emph{reduced cost function}\index{reduced cost function}: (see \cite{LEVIN1997} and the references therein)
 \[
c_{0,n}\left(x,y\right)=\min\left\{ \sum_{i=0}^{n}c\left(x_{i},x_{i+1}\right);\ x_{1},\ldots,x_{n}\in X,\ x_{0}=x,\ x_{n+1}=y\right\} 
\]
When passing through $z\in X$, there may be some fine $f(z)$ (where $f:X\rightarrow\r$ is given) which will affect the cost of transporting. We define the \emph{weighted reduced cost function}\index{weighted reduced cost function}:
\[
c_{f,n}\left(x,y\right)=\min\left\{ \sum_{i=0}^{n}c\left(x_{i},x_{i+1}\right)-\sum_{i=1}^{n}f\left(x_{i}\right);\ x_{1},\ldots,x_{n}\in X,\ x_{0}=x,\ x_{n+1}=y\right\} 
\]

\begin{rem}
$c_{f,0}=c$.
\end{rem}

We will use the following recursive relation:
\begin{lem}
\label{lem:cfn}
For any $n\in\mathbb{N}$, $x,y\in X$ and $f:X\rightarrow\r$
$$c_{f,n}(x,y)=\min_{z\in X}c_{f,n-1}(x,z)+c(z,y)-f(z)$$
\end{lem}

A similar process can be done with the total cost functional:
$$
c_{\#}(\m,\n):=\inf_{\pi\in\Pi(\m,\n)}\int_{X\x X}c\d\pi
$$
When transporting a measure $\mu\in\mathcal{P}(X)$ another measure $\nu\in\mathcal{P}(X)$, we may restrict the transport to go through a finite sequence of medium measures with average $\lambda$:
\[
\left(c_{\#}\right)_{n}^{\lambda}\left(\m,\n\right):=\min\left\{ \sum_{i=1}^{n}c_{\#}\left(\rho_{i},\ro_{i+1}\right);\ \sum_{i=1}^{n}\ro_{i}=n\lambda,\ \ro_{0}=\m,\ \ro_{n+1}=\n\right\} 
\]

\begin{rem}
Notice that $\inf_{\lambda\in\mathcal{P}\left(X\right)}\left(c_{\#}\right)_{n}^{\lambda}=\left(c_{\#}\right)_{0,n}$.
\end{rem}

In this chapter, we will establish the following connection between $(c_{f,n})_{\#}$ and $\left(c_{\#}\right)_{n}^{\lambda}$ using  Theorem \ref{Thm:AbsDuality}:

\b

\colorbox{thmcolor}{
\parbox{\linewidth}{
\begin{thm}
\label{thm:DiscreteDyn}Let $X$ be a compact metric space, $c\in C\left(X\x X\right)$,
$\mu,\nu,\lambda\in\mathcal{P}\left(X\right)$ and $n\in\mathbb{N}$.
Then
\[
\left(c_{\#}\right)_{n}^{\lambda}\left(\m,\n\right)=\sup_{f\in C\left(X\right)}\left(c_{f,n}\right)_{\#}\left(\m,\n\right)+n\int f\d\lambda
\]
\end{thm}
}}

\begin{proof}
Along this proof we are using the following notation: For a measure $\pi\in\mm_+(X\x X)$ we denote by $\pi^i,\ i=1,2$ its first/second marginals respectively on $X$. We break the proof into four parts.
\begin{enumerate}
\item We first show the  following equality by proving inequality in both directions:
\begin{equation}
    \label{eq:Dyn1}
    \left(c_{\#}\right)_{n}^{\lambda}\left(\m,\n\right)=
    \min\left\{ \sum_{i=0}^{n}\int_{X\x X}c\d\pi_{i};\ \pi_{0}^{1}=\m,\ \pi_{n}^{2}=\n,\ \sum_{i=1}^{n}\pi_{i}^{1}=n\lambda,\ \pi_{i-1}^{2}=\pi_{i}^{1}\ \forall i=1,\ldots,n\right\}
\end{equation}

\begin{itemize}
    \item For any $\ro_{0},\ldots,\ro_{n+1}$ satisfying $\sum_{i=1}^{n}\ro_{i}=n\lambda,\ \ro_{0}=\m,\ \ro_{n+1}=\n$,
    define $\pi_{i}$ to be the optimal transport plan between $\ro_{i}$
    and $\ro_{i+1}$. Then $\sum_{i=0}^{n}c_{\#}\left(\ro_{i},\ro_{i+1}\right)=\sum_{i=0}^{n}\int_{X\x X}c\d\pi_{i}$
    and $\pi_{0}^{1}=\m,\ \pi_{n}^{2}=\n,\ \sum_{i=1}^{n}\pi_{i}^{1}=\sum_{i=1}^{n}\rho_{i}=n\lambda,\ \pi_{i-1}^{2}=\rho_i=\pi_{i}^{1}\ \forall i=1,\ldots,n$. Therefore there is an inequality $\geq$.
    \item For every $\pi_{0},\ldots,\pi_{n}$ such that $\pi_{0}^{1}=\m,\ \pi_{n}^{2}=\n,\ \sum_{i=1}^{n}\pi_{i}^{1}=n\lambda,\ \pi_{i-1}^{2}=\pi_{i}^{1}\ \forall i=1,\ldots,n$,
    define $\ro_{i}=\pi_{i}^{1}$ for every $i=0,\ldots,n$ and $\rho_{n+1}=\n$.
    Then $\sum_{i=0}^{n}c_{\#}\left(\ro_{i},\ro_{i+1}\right)\leq\sum_{i=0}^{n}\int c\d\pi_{i}$
    and $\sum_{i=1}^{n}\ro_{i}=\sum_{n=1}^n\pi_i^1=n\lambda,\ \ro_{0}=\m,\ \ro_{n+1}=\n$, so we get the reversed inequality as well.
\end{itemize}

\item Now we find the dual problem of the RHS of (\ref{eq:Dyn1}) using Theorem \ref{Thm:AbsDuality}:
\[
\begin{array}{c}
\min\left\{ \sum_{i=0}^{n}\int c\d\pi_{i};\ \pi_{0}^{1}=\m,\ \pi_{n}^{2}=\n,\ \sum_{i=1}^{n}\pi_{i}^{1}=n\lambda,\ \pi_{i-1}^{2}=\pi_{i}^{1}\ \forall i=1,\ldots,n\right\} =\\
\\
\sup\left\{ \int f_{0}\d\m-\int f_{n+1}\d\n+n\int f\d\lambda;\ f_{0}\left(x\right)-f_{1}\left(y\right)\leq c\left(x,y\right),\ f_{i}\left(x\right)-f_{i+1}\left(y\right)+f\left(x\right)\leq c\left(x,y\right)\ \forall i=1,\ldots,n\right\} 
\end{array}
\]
The set of all vector measures $\left(\pi_{0},\ldots,\pi_{n}\right)\in\mathcal{P}(X\x X)^{n+1}$
satisfying $\pi_{0}^{1}=\m,\ \pi_{n}^{2}=\n,\ \sum_{i=1}^{n}\pi_{i}^{1}=n\lambda,\ \pi_{i-1}^{2}=\pi_{i}^{1}\ \forall i=1,\ldots,n$,
are exactly the set of vector measures for which

\[
\begin{array}{cc}

 \int f_{0}(x)-f_{1}(y)\d\pi_0(x,y)+\int f_{1}(x)-f_{2}(y)+f(x)\d\pi_1+\cdots+\int f_{n}(x)-f_{n+1}(y)+f(x)\d\pi_n= \\ \int f_{0}\d\m-\int f_{n+1}\d\n+n\int f\d\lambda
     
\end{array}
\]

for all $f_{0},\ldots,f_{n+1},f\in C\left(X\right)$.
By Theorem \ref{Thm:AbsDuality}, the minimum of $\sum_{i=1}^n\int c\d\pi_i$ over all these measures
equals to the supremum of
$$
\int f_{0}\d\m-\int f_{n+1}\d\n+n\int f\d\lambda
$$
over all $f_{0},\ldots,f_{n+1},f\in C\left(X\right)$ satisfying
\[
\left(f_{0}\left(x\right)-f_{1}\left(y\right),f_{1}\left(x\right)-f_{2}\left(y\right)+f\left(x\right),\ldots,f_{n}\left(x\right)-f_{n+1}\left(y\right)+f\left(x\right)\right)\leq\left(c\left(x,y\right),\ldots,c\left(x,y\right)\right)
\]
(where the inequality is in each component).
To see that just take $V=C(X\x X)^{n+1}$, $\q=C(X\x X,\r_+)^{n+1}$, $K=C(X)^{n+3}$, 
$$
h(f_0,\ldots,f_{n+1},f)(x,y):=\left(f_{0}\left(x\right)-f_{1}\left(y\right),f_{1}\left(x\right)-f_{2}\left(y\right)+f\left(x\right),\ldots,f_{n}\left(x\right)-f_{n+1}\left(y\right)+f\left(x\right)\right),$$
$k^*(f_0,\ldots,f_{n+1},f)=\int f_{0}\d\m-\int f_{n+1}\d\n+n\int f\d\lambda$ and $u=(c,\ldots,c)$. Assumption \ref{assu:Sublinear} holds by taking $s(g_0,\ldots,g_n):=(\|g_0\|_{\infty},0,\ldots,0,\sum_{i=1}^n\|g_i\|_{\infty})$.
\item We show the dual problem can be modified to a maximization problem over 3 functions only (instead of $n+3$ functions):
\[
\begin{array}{c}
\sup\left\{ \int f_{0}\d\m-\int f_{n+1}\d\n+n\int f\d\lambda;\ f_{0}\left(x\right)-f_{1}\left(y\right)\leq c\left(x,y\right),\ f_{i}\left(x\right)-f_{i+1}\left(y\right)+f\left(x\right)\leq c\left(x,y\right)\ \forall i=1,\ldots,n\right\} \\
\\
=\sup\left\{ \int f_{0}\d\m-\int f_{n+1}\d\n+n\int f\d\lambda;\ f_{0}\left(x_{0}\right)-f_{n+1}\left(x_{n+1}\right)+\sum_{i=1}^{n}f\left(x_{i}\right)\leq\sum_{i=0}^{n}c\left(x_{i},x_{i+1}\right)\right\} 
\end{array}
\]
\begin{itemize}
    \item The inequality $\leq$ is immediate since every $f_{0},\ldots,f_{n+1},f$
    satisfying $f_{0}\left(x\right)-f_{1}\left(y\right)\leq c\left(x,y\right),\ f_{i}\left(x\right)-f_{i+1}\left(y\right)+f\left(x\right)\leq c\left(x,y\right)\ \forall i=1,\ldots,n$
    also satisfy $f_{0}\left(x_{0}\right)-f_{n+1}\left(x_{n+1}\right)+\sum_{i=1}^{n}f\left(x_i\right)\leq\sum_{i=0}^{n}c\left(x_{i},x_{i+1}\right)$ (by summing all the inequalities).
    \item  To show the other direction, for any
    $f_{0},f_{n+1},f$ satisfying
    $$
    f_{0}\left(x_{0}\right)-f_{n+1}\left(x_{n+1}\right)+\sum_{i=1}^{n}f\left(x_{i}\right)\leq\sum_{i=0}^{n}c\left(x_{i},x_{i+1}\right)
    $$
    We need to show there exist some $f_1,\ldots,f_n$ satisfying $f_{0}\left(x\right)-f_{1}\left(y\right)\leq c\left(x,y\right),\ f_{i}\left(x\right)-f_{i+1}\left(y\right)+f\left(x\right)\leq c\left(x,y\right)\ \forall i=1,\ldots,n$.
    We define
    \[
    \begin{array}{ccccc}
    f_{1}\left(y\right) & := & \max_{x\in X}\left\{ f_{0}\left(x\right)-c\left(x,y\right)\right\} \\
    f_i\left(y\right) & := & \max_{x\in X}\left\{ f_{i-1}\left(x\right)-c\left(x,y\right)+f\left(x\right)\right\}  & ; & 1<i\leq n
    \end{array}
    \]
    which implies 
    \[  
    \begin{array}{ccccc}
    f_{0}\left(x\right)-f_{1}\left(y\right) & = & f_{0}\left(x\right)-\max_{z\in X}\left\{ f_{0}\left(z\right)-c\left(z,y\right)\right\}  & \leq & c\left(x,y\right)\\
    f_{i}\left(x\right)-f_{i+1}\left(y\right)+f\left(x\right) & = & f_{i}\left(x\right)+f\left(x\right)-\max_{z\in X}\left\{ f_{i}\left(z\right)-c\left(z,y\right)+f\left(z\right)\right\}  & \leq & c\left(x,y\right)
    \end{array}
    \]
    for every $i=1,\ldots,n-1$. As for $i=n$ one can prove using induction arguments that
    \[
    f_n\left(y\right)=\max_{z\in X}\left\{ f_{0}\left(z\right)-c_{f,n-1}\left(z,y\right)\right\} 
    \]
     By the constraint $f_{0}\left(x_{0}\right)-f_{n+1}\left(x_{n+1}\right)+\sum_{i=1}^{n}f\left(x_{i}\right)\leq\sum_{i=0}^{n}c\left(x_{i},x_{i+1}\right)$
    we may assume 
    \[
    f_{n+1}\left(y\right)=\max_{z\in X}\left\{ f_{0}\left(z\right)-c_{f,n}\left(z,y\right)\right\} 
    \]
    hence 
    \[
    \begin{array}{ccccc}
    f_{n}\left(x\right)-f_{n+1}\left(y\right)+f\left(x\right) & = &     \max_{z\in X}\left\{ f_{0}\left(z\right)-c_{f,n-1}\left(z,x\right)\right\} -\max_{z\in X}\left\{ f_{0}\left(z\right)-c_{f,n}\left(z,y\right)\right\}+f\left(x\right) \\
    & \leq & \max_{z\in X}\left\{ c_{f,n}\left(z,y\right)-c_{f,n-1}\left(z,x\right)\right\} +f\left(x\right) \\ & \leq & c\left(x,y\right)
    \end{array}
    \]
    where the last inequality holds since by Lemma \ref{lem:cfn} for any $z\in X$
    \[
    c_{f,n}\left(z,y\right)\leq c_{f,n-1}\left(z,x\right)+c\left(x,y\right)-f\left(x\right)
    \]
\end{itemize}

\item By combining the 3 equalities obtained in the steps above we get
\[
\left(c_{\#}\right)_{n}^{\lambda}\left(\m,\n\right)=\sup\left\{ \int f_{0}\d\m-\int f_{n+1}\d\n+n\int f\d\lambda;\ f_{0}\left(x_{0}\right)-f_{n+1}\left(x_{n+1}\right)\leq\sum_{i=0}^{n}c\left(x_{i},x_{i+1}\right)-\sum_{i=1}^{n}f\left(x_{i}\right)\right\} 
\]
and by definition of $c_{f,n}$ we conclude
\[
\left(c_{\#}\right)_{n}^{\lambda}\left(\mu,\nu\right)=\sup\left\{ \int\psi\d\m-\int\varphi\d\n+n\int f\d\lambda;\ (\psi,\varphi,f)\in C(X)^3,\ \psi\left(x\right)-\varphi\left(y\right)\leq c_{f,n}\left(x,y\right)\right\} 
\]
Hence by the classical Kantorovich duality (Theorem \ref{thm:Duality})
\[
\left(c_{\#}\right)_{n}^{\lambda}\left(\m,\n\right)
=
\sup_{f\in C(X)}\left\{\sup_{\psi-\varphi\leq c_{f,n}} \left\{\int\psi\d\m-\int\varphi\d\n\right\}+n\int f\d\lambda\right\}
=
\sup_{f\in C(X)}\left(c_{f,n}\right)_{\#}\left(\m,\n\right)+n\int f\d\lambda
\]
\end{enumerate}
\end{proof}
\begin{cor}
\label{cor:c=000023n=00003Dcn=000023}$\left(c_{\#}\right)_{0,n}=\left(c_{0,n}\right)_{\#}$.
\end{cor}

\begin{proof}
$f\mapsto(c_{f,n})_\#(\m,\n)$ is concave for each $\m,\n\in\mathcal{P}(X),n\in\mathbb{N}$: Let $f,g\in C(X)$ and $t\in[0,1]$. By the definition of $c_{f,n}$ we get
$$
c_{tf+(1-t)g,n}(x,y)\geq tc_{tf,n}+(1-t)c_{g,n}(x,y)
$$

for all $x,y\in X$. Therefore
\[
\begin{array}{ccc}
   (c_{tf+(1-t)g,n})_\#(\m,\n)
   & = &
   \min_{\pi\in\Pi(\m,\n)}\int c_{tf+(1-t)g,n}\d\pi
   \\
   & \geq &
   \min_{\pi\in\Pi(\m,\n)}t\int c_{f,n}\d\pi+(1-t)\int c_{g,n}\d\pi
   \\
    & \geq &
    t(c_{f,n})_\#(\m,\n)+(1-t)(c_{g,n})_\#(\m,\n)
\end{array}
\]

Hence we may use Theorem \ref{thm:Convex Abstract} with $V=C(X),\ K=C(X)\x\r,\ k^*(f,\alpha)=-(c_{f,n})_\#-\alpha,\ h(f,\alpha)=nf-\alpha$ and $\q=C(X,\r_+)$ and get (together with Remark  \ref{rem:minmaxNonNeg}):
$$
\sup_{\lambda\in\mathcal{M}_+(X)}\inf_{(f,\alpha)\in K}-\left(c_{f,n}\right)_{\#}-n\int f\d\lambda-\alpha\cdot (1-\lambda(X))
=
\inf_{nf-\alpha\geq0}-(c_{f,n})_\#-\alpha
$$
which is equivalent to
$$
\inf_{\lambda\in\mathcal{M}_+(X)}\sup_{(f,\alpha)\in K}\left(c_{f,n}\right)_{\#}+n\int f\d\lambda+\alpha\cdot (1-\lambda(X))
=
\sup_{nf-\alpha\geq0}(c_{f,n})_\#+\alpha
$$
We may restrict the infimum only over probability measures and restrict the supremum on the RHS for $\alpha=n\cdot\min f$ thus
$$
\inf_{\lambda\in\mathcal{P}(X)}\sup_{f\in C(X)}\left(c_{f,n}\right)_{\#}+n\int f\d\lambda
=
\sup_f(c_{f,n})_\#+n\min f
$$
So Theorem \ref{thm:DiscreteDyn} above leads to
\[
\begin{array}{ccccc}
\left(c_{\#}\right)_{0,n}
& = & \inf_{\lambda\in\mathcal{P}(X)}\left(c_{\#}\right)_{n}^{\lambda} \\
& = & \inf_{\lambda\in\mathcal{P}(X)}\sup_{f}\left(c_{f,n}\right)_{\#}+n\int f\d\lambda  \\
& = & \sup_f(c_{f,n})_\#+n\min f \\
& = & \sup_{f\geq0}(c_{f,n})_\#
& = & \left(c_{0,n}\right)_{\#}
\end{array}
\]
\end{proof}

The above corollary implies a result about the Wasserstein metric\index{Wasserstein metric} on Euclidean spaces defined by $\mathcal{W}_p:=(c_\#)^{\frac{1}{p}}$ when $c(x,y)=|x-y|^p$.

\begin{cor}
For $X\subset\r^m$ and $p\geq 1$: $(\mathcal{W}_p^p)_{0,n}=n^{1-p}\mathcal{W}_p^p$.
\end{cor}
\begin{proof}
It is not to hard to verify that $c_{0,n}=n^{1-p}c$, so
by the last corollary:
\[
(\mathcal{W}_p^p)_{0,n}=(c_\#)_{0,n}=(c_{0,n})_\#=(n^{1-p}c)_\#=n^{1-p}c_\#=n^{1-p}\mathcal{W}_p^p
\]
\end{proof}

A possible direction for further research is to achieve a continuous result by taking the limit $n\rightarrow\infty$ in Theorem \ref{thm:DiscreteDyn}.  For example when $c(x,y)=d(x,y)^p$ ($d$ is the metric of the space):

\begin{equation}
\label{eq:WeightedBB}
\min_{\rho,v}\int_{0}^{1}\int_{X\x X}\left|v_t(x,y)\right|^{p}\d\rho_t(x,y)\d t =\sup_{f\in C(X)}\left(c_{f}^{p}\right)_{\#}\left(\mu,\nu\right)+\int_X f\d\lambda
\end{equation}

Where the minimum is taken over all continuous $\rho:\left[0,1\right]\rightarrow(\mathcal{P}\left(X\right),w^*)$ such that
$$
\int_{0}^{1}\rho_{t}\text{d}t=\lambda,\ro_{0}=\m,\ro_{1}=\n, \frac{\partial\rho_t}{\partial t}+\nabla\cdot(\rho_tv_t)=0
$$

and
$$
c_{f}^{p}\left(x,y\right):=\inf\left\{ \int_{0}^{1}\left|\dot{\gamma}\left(t\right)\right|^{p}-f\left(\gamma\left(t\right)\right)\mbox{d}t;\ \gamma\in C(\left[0,1\right],X),\ \gamma\left(0\right)=x,\ \gamma\left(1\right)=y\right\} 
$$

\begin{rem}
The meaning of the equation $\frac{\partial\rho_t}{\partial t}+\nabla\cdot(\rho_tv_t)=0$ is in the distributional sense: For each smooth $\phi\in C^{\infty}(X)$ with compact support
$$
\frac{\d}{\d t}\int_X\phi\d\rho_t=\int_X\nabla\phi\cdot v_t\d\rho_t
$$
\end{rem}

By taking the infimum over all $\lambda\in\mathcal{P}(X)$ on both sides of (\ref{eq:WeightedBB}) and using similar arguments as in Corollary \ref{cor:c=000023n=00003Dcn=000023} we get

\[
\begin{array}{ccc}
\min\left\{\int_{0}^{1}\int_{X\x X}\left|v_t(x,y)\right|^{p}\d\rho_t(x,y)\d t;\ \ro_{0}=\m,\ro_{1}=\n, \frac{\partial\rho_t}{\partial t}+\nabla\cdot(\rho_tv_t)=0\right\}
&=\\
\min_{\lambda\in\mathcal{P}(X)}\sup_{f\in C(X)}\left(c_{f}^{p}\right)_{\#}\left(\mu,\nu\right)+\int_X f\d\lambda
&=\\
\sup_{f\in C(X)}\min_{\lambda\in\mathcal{P}(X)}\left(c_{f}^{p}\right)_{\#}\left(\mu,\nu\right)+\int_X f\d\lambda
&=\\
\sup_{f\in C(X)}\left(c_{f}^{p}\right)_{\#}\left(\mu,\nu\right)+\min f
&=\\
\sup_{0\leq f\in C(X)}\left(c_{f}^{p}\right)_{\#}\left(\mu,\nu\right)
&=\\
\left(c_{0}^{p}\right)_{\#}\left(\mu,\nu\right)
&=& \mathcal{W}_p^p(\mu,\nu)
\end{array}
\]

Which is known as the "Benamou-Brenier formula" (see Theorem 8.1 in \cite{Villani2003}).

\section{Conclusion}

The main results of Chapter 2 are Theorems \ref{Thm:AbsDuality} and \ref{thm:Convex Abstract}. As similar results can be found in the literature, my contributions that make Theorems \ref{Thm:AbsDuality} and \ref{thm:Convex Abstract} relevant are the following:
\begin{itemize}
    \item The theorems deal with general tvs, which are not necessary normed/locally convex/Hausdorff.
    \item In Theorem \ref{Thm:AbsDuality} - we have a characterization of the cases in which the duality is finite.
    \item Theorem \ref{thm:ExMin} gives sufficient conditions for the existence of solutions for the dual problem.
    \item In Theorem \ref{thm:Convex Abstract} - our Assumption (\ref{assu:convex}) is weaker than assumptions given in other similar theorems.
\end{itemize}

The diagram below summarizes the connections between some of the different
theorems that appeared in this work.

\begin{tikzpicture}[scale=1]
  \hypersetup{hidelinks}
  \node  (MYCENTER) at (0,0) {\includegraphics[scale=0.5]{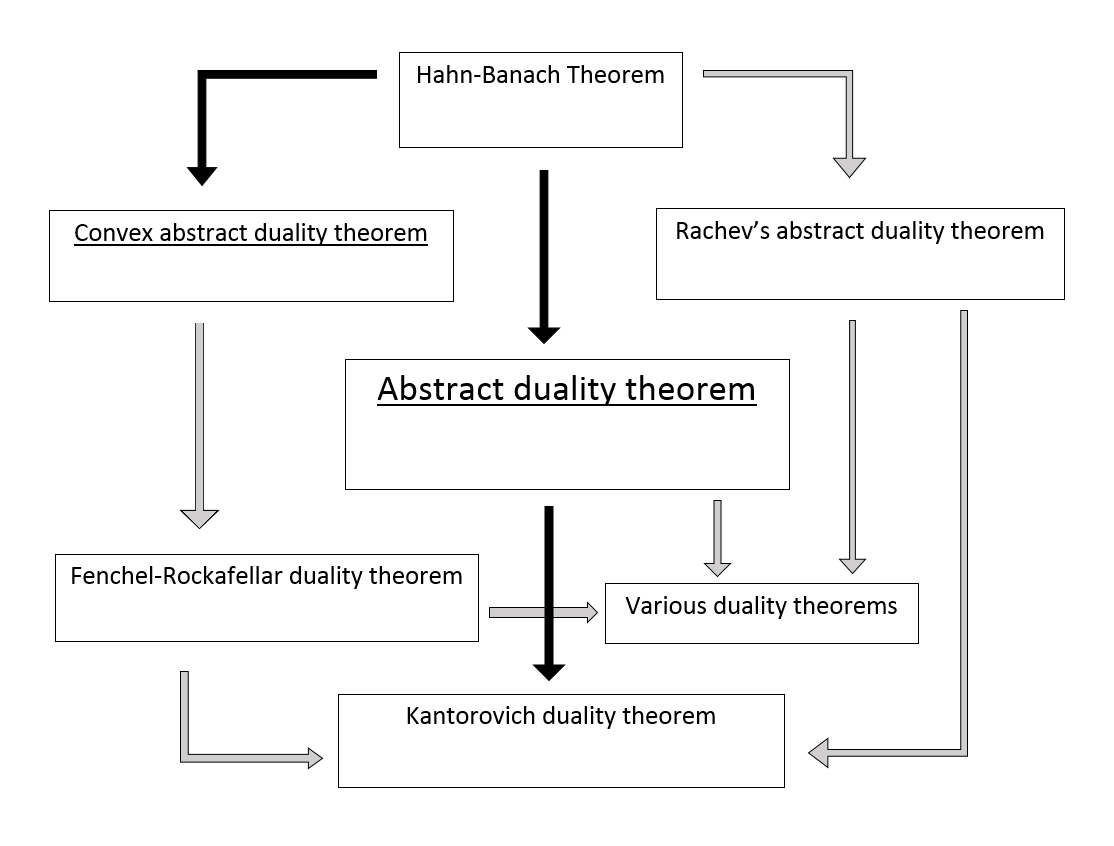}};
  \node  (BLUE)   at (-0.5,4)  {Theorems \ref{thm:Hahn-Banach}/\ref{thm:HB Separation}};
  \node  (GREEN)  at (-4,2)   {Theorem \ref{thm:Convex Abstract}};
  \node  (RED)    at (4,2)  {Theorem \ref{thm:Rachev}};
  \node  (ORANGE) at (0,-0.4)  {Theorem \ref{Thm:AbsDuality}};
  \node  (ORANGE) at (-4,-2.5)  {Theorem \ref{thm:FRduality}};
  \node  (ORANGE) at (0,-4.5)  {Theorem \ref{thm:Duality}};
\end{tikzpicture}

The dark arrows refer to the connections that have been established in Chapter 2. The two underlined theorems are the main results of Chapter 2. Some \textquotedbl various duality theorems\textquotedbl{}
that can be obtained by these theorems are mentioned in Chapters 4,5 and 6.

The main contributions of Chapter 3 are:
\begin{itemize}
    \item The optimal transport problem for vector measures and the corresponding duality theorem (Theorem \ref{thm:GMK}).
    \item Equivalent conditions to the existence of transport plans (Theorem \ref{thm:Blackwell}).
    \item Connection to the Martingale optimal transport problem (Theorem \ref{thm:MartingaleOT}).
    \item  Existence of a solution to the dual problem on a weak* dense subset of measures (Theorems \ref{thm:ExMinVec} and \ref{thm:ExMinCont}).
    \item Existence of optimal transport map in the semi-discrete case (Theorem \ref{thm:ExTransMap}) and some transport map in the general case (Theorem \ref{thm:ExTransMap generalcase}).
\end{itemize}

\newpage{}

\bibliographystyle{plain}
\phantomsection\addcontentsline{toc}{section}{\refname}\bibliography{Ref.bib}

\newpage{}

\printindex

\end{document}